\documentclass[11pt]{amsart}
\usepackage[a4paper,margin=25mm]{geometry}
\usepackage{amsmath,amssymb,amsthm,mathtools,mathrsfs}
\usepackage{booktabs,longtable,array,enumitem}
\usepackage{needspace,placeins}
\usepackage[T1]{fontenc}
\usepackage{lmodern,microtype}
\usepackage{xcolor}
\usepackage{tikz}
\usetikzlibrary{arrows.meta,positioning,calc}
\usepackage[colorlinks=true,linkcolor=blue!45!black,citecolor=blue!45!black,urlcolor=blue!45!black]{hyperref}
\usepackage{orcidlink}
\hypersetup{pdftitle={Crossed modularity of twisted VOA modules},pdfauthor={Shun Xu}}
\AtBeginDocument{\setlength{\abovedisplayskip}{8pt plus 2pt minus 3pt}\setlength{\belowdisplayskip}{8pt plus 2pt minus 3pt}\setlength{\abovedisplayshortskip}{4pt plus 2pt minus 2pt}\setlength{\belowdisplayshortskip}{6pt plus 2pt minus 2pt}}
\setlist[enumerate]{label=\textup{(\roman*)},leftmargin=*,itemsep=3pt}
\setlist[itemize]{leftmargin=*,itemsep=3pt}
\newcolumntype{P}[1]{>{\raggedright\arraybackslash}p{#1}}
\numberwithin{equation}{section}
\newtheorem{theorem}{Theorem}[section]
\newtheorem{proposition}[theorem]{Proposition}
\newtheorem{lemma}[theorem]{Lemma}
\newtheorem{corollary}[theorem]{Corollary}
\theoremstyle{definition}
\newtheorem{definition}[theorem]{Definition}
\theoremstyle{remark}
\newtheorem{remark}[theorem]{Remark}
\newcommand{\C}{\mathbb C}

\newcommand{\cC}{\mathcal C}

\newcommand{\bo}{\boxtimes}
\newcommand{\wt}{\operatorname{wt}}

\newcommand{\Hom}{\operatorname{Hom}}

\newcommand{\Irr}{\operatorname{Irr}}

\newcommand{\Aut}{\operatorname{Aut}}

\newcommand{\id}{\operatorname{id}}

\newcommand{\cB}{\mathcal B}
\newcommand{\FPdim}{\operatorname{FPdim}}
\newcommand{\Rep}{\operatorname{Rep}}
\newcommand{\ev}{\operatorname{ev}}
\newcommand{\coev}{\operatorname{coev}}
\title[Crossed modularity of twisted VOA modules]{Crossed modularity of twisted VOA modules}
\author[Shun Xu]{Shun Xu~\orcidlink{0009-0006-8080-8107}}
\address{School of Mathematical Sciences, Anhui University, Hefei 230601, Anhui, China}
\email{shunxu@ahu.edu.cn}
\date{}

\begin{document}
\begin{abstract}
Let $V$ be a strongly rational vertex operator algebra and
$G\leq\Aut(V)$ a finite group. Using the $G$-crossed braided fusion
structure of twisted modules constructed in \cite{Crossed}, we prove
that the conformal twist $e^{2\pi iL(0)}$ gives a crossed ribbon
structure and that the crossed category is modular in Kirillov's sense.
Consequently, its equivariantization is a modular tensor category.
We also obtain dimension and simple-object counting formulas for
the twisted sectors.
\end{abstract}
\maketitle

\section{Introduction}
Let $V$ be a strongly rational vertex operator algebra and
$G\leq\Aut(V)$ a finite group.  For each $g\in G$, let $\cC_g$ denote
the category of grading-restricted $g$-twisted $V$-modules, and put
$\cC_G=\bigoplus_{g\in G}\cC_g$, with the elements of $G$ retained as
sector labels.  Orbifold theory suggests that tensor products of
twisted modules should respect multiplication of these labels, while
exchanging a $g$-twisted field with another field should involve the
action of $g$.  The appropriate categorical framework is therefore a
$G$-crossed, rather than an ordinary braided, tensor category.

The representation theory underlying this expectation was developed
by Dong, Li and Mason, who established fundamental results on
finite-order twisted modules and twisted Zhu algebras
\cite{DLM}; their modular-invariance theorem concerns trace functions
in orbifold theory under its stated finiteness and rationality
hypotheses \cite{DLMModular}.  For untwisted modules, the tensor-product
theory of Huang and Lepowsky realizes products through intertwining
operators and their analytic continuation \cite{HL}.  Huang's
differential-equation results provide convergence and extension tools
for this theory \cite{HuangDE}, and his rigidity and modularity theorem
gives a modular tensor category of ordinary modules under the
standard semisimplicity, self-duality and cofiniteness assumptions
\cite{Huang}.  The crossed tensor structure on all twisted sectors is the
additional categorical structure needed for the problem considered
here.

Kirillov developed the theory of $G$-equivariant fusion categories,
with twisted modules as a motivating example, and introduced the
extended Verlinde algebra and its $S$-operator
\cite[Example~5.5 and Sections~7--10]{Kirillov}.  In this setting
modularity means invertibility of the extended $S$-operator; it is
stronger than requiring only the neutral component to be modular.
The corresponding VOA problem was formulated by Huang in
\cite[Conjecture~6.2]{HuangProblems}: for a finite automorphism group
and a vertex operator algebra satisfying the specified
self-contragredience, complete-reducibility and $C_2$-cofiniteness
conditions, the twisted sectors should form a $G$-crossed tensor
category.  Huang's preceding Conjecture~6.1 separately addresses
associativity, commutativity and modular invariance of twisted
intertwining operators.  The modularity question considered here is the invertibility of
Kirillov's extended $S$-operator.

An approach through the fixed-point algebra $V^G$ gives important
instances of the predicted crossed structure, subject to additional
hypotheses.  Carnahan and Miyamoto proved regularity of fixed points
of a finite-order automorphism of a simple, nonnegatively graded,
regular VOA with a nondegenerate invariant bilinear form, and thereby
obtained the corresponding result for finite solvable groups
\cite{CarnahanMiyamoto}.  McRae constructed a braided $G$-crossed
category of twisted modules and identified its equivariantization
with an ambient tensor category of $V^G$-modules, assuming the
requisite vertex tensor category and algebra-object setting
\cite{McRae}.  These constructions rely on the additional vertex-tensor-category
and algebra-object hypotheses described above.

Du and Huang subsequently developed an analytic theory of general
twisted intertwining operators, including skew-symmetry and
contragredient isomorphisms, $P(z)$-tensor products under suitable
conditions, and crossed commutativity and braiding maps
\cite{DH}.  They restated Huang's crossed tensor-category conjecture
and proposed a stronger crossed \emph{modular} tensor-category
conjecture in a suitable sense.  They expressly noted that the
modularity notions considered by Kirillov and Turaev differ, so the
stronger assertion requires a choice of categorical definition.
In the present paper, modularity is understood in Kirillov's sense.

Xu \cite[Theorem~1.1]{Crossed} constructs, for strongly rational
$V$ and finite $G$, a faithfully $G$-graded $G$-crossed braided fusion
category of twisted $V$-modules.  Its tensor products represent the
complete Du--Huang intertwining operators, with specified parallel
transports, positive-half-turn crossed braidings, action tensorators
and contragredient pairings. We use this established structure as the
categorical input. The remaining question is whether these concrete
operators carry the additional ribbon and nondegeneracy structures
required for Kirillov's modularity.

There are genuine compatibility issues beyond the existence of a
crossed braiding.  To apply Huang's modularity theorem to the neutral
component of the category in \cite{Crossed}, one must compare the
ordinary intertwining operators defined by the Jacobi identity with
the \emph{complete} intertwining operators of Du and Huang; the latter
impose absolute convergence and multivariable continuation
conditions that require verification.  Moreover, the conformal
monodromy $e^{2\pi iL(0)}$ must be interpreted as a crossed twist and
checked against the chosen positive-half-turn braiding, the action
tensorators and categorical duality.  Only after constructing this
crossed ribbon structure can the equivariantization be treated as a
ribbon category.  Finally, the nondegeneracy criterion of
\cite[Proposition~4.56(ii)]{DGNO} uses \emph{both} nondegeneracy of
the neutral component and faithfulness of the grading; the
comparison with Kirillov's extended $S$-operator then uses
\cite[Theorem~10.5]{Kirillov}.

Starting from Xu's established crossed-category construction, we
prove this categorical modular refinement. The neutral component is
identified with Huang's ordinary modular category, the conformal
twist is shown to satisfy crossed balancing and duality, and the
equivariantization becomes a nondegenerate ribbon fusion category.
Kirillov's theorem then gives invertibility of the extended
$S$-operator. Our modularity statement is categorical: it concerns the crossed
category and Kirillov's extended $S$-operator. The genus-one
transformation laws of twisted trace functions form a separate
analytic question.

\begin{theorem}\label{thm:main}
Let $V$ be strongly rational and $G\leq\Aut(V)$ finite, and let
$\cC_G$ be Xu's faithfully $G$-graded $G$-crossed braided fusion
category of grading-restricted finite-order
twisted $V$-modules \cite[Theorem~1.1]{Crossed}. We use its binary
$P(z)$-representing property and canonical operators
\cite[Theorem~4.3, equations~(4.3)--(4.4)]{Crossed}, its pullback
tensorators and monoidal action coherence
\cite[equation~(4.7), Proposition~8.5, equations~(8.15), (8.16),
(8.18), (8.19)]{Crossed}, its product--iterate normalization
\cite[Theorem~8.3, equations~(8.10)--(8.11)]{Crossed}, its
positive-half-turn crossed commutors and parallel-transport identities
\cite[Proposition~9.6, equations~(9.27)--(9.30)]{Crossed}, and its
contragredient pairings and rigidity
\cite[Lemma~10.1, equations~(10.1)--(10.2), Proposition~10.3]{Crossed}.
Then:
\begin{enumerate}[label=\textup{(\roman*)},leftmargin=*]
\item The maps $\theta_X=e^{2\pi iL_X(0)}:X\longrightarrow T_gX$ for $X\in\cC_g$ define a canonical crossed ribbon structure in Kirillov's conventions.
\item The equivariantization $\cC_G^G$, with its induced ribbon twist, is a modular tensor category, and $\cC_G$ is modular as a $G$-equivariant category in the sense of Kirillov. In particular the extended $\widetilde S$-operator is invertible.
\item The Frobenius--Perron dimensions satisfy
\[
\FPdim(\cC_g)=\FPdim(\Rep(V)),\qquad
\FPdim(\cC_G^G)=|G|^2\FPdim(\Rep(V)).
\]
\end{enumerate}
\end{theorem}

The comparison of ordinary and complete intertwining operators is
proved in Theorem~\ref{thm:neutralanalytic}, and the resulting
identification of the neutral tensor category appears in
Proposition~\ref{prop:neutralcomparison}.  Sections~4 and~5 establish
the crossed balancing and duality conditions, respectively.  The
equivariantization and nondegeneracy arguments, together with the
Frobenius--Perron dimension formulas, occupy Section~6.  That
section also derives a count of the simple objects in each twisted
sector in terms of the $g$-fixed simple untwisted modules
(Corollary~\ref{cor:simplecount}).

\section{Conventions and the canonical operators}

\subsection{Categorical, VOA, and twisted-module conventions}\label{subsec:twistedconventions}

A \emph{fusion category} is a $\C$-linear semisimple rigid monoidal
category with simple unit, finite-dimensional morphism spaces and
finitely many isomorphism classes of simple objects.  For a
\emph{braided} fusion category $\mathcal D$ with braiding $c$, its
\emph{M\"uger center} $\mathcal Z_2(\mathcal D)$ is the full
subcategory of objects $A$ for which
$c_{B,A}c_{A,B}=\id_{A\bo B}$ for all objects $B$.  Such a category
is \emph{nondegenerate} if $\mathcal Z_2(\mathcal D)$ is braided
equivalent to $\mathrm{Vec}$.  Here a \emph{modular tensor category} means a nondegenerate ribbon
fusion category.

We use the standard notion of a vertex operator algebra
$(V,Y,\mathbf1,\omega)$ over $\C$ \cite{FHL,DLM}.  In particular,
$V=\bigoplus_{d\in\mathbb Z}V_d$ has finite-dimensional homogeneous
subspaces and $V_d=0$ for $d\ll0$;
$\mathbf1\in V_0$, $\omega\in V_2$, and
$Y(u,x)=\sum_{n\in\mathbb Z}u_nx^{-n-1}$ is lower truncated on each
$v\in V$.  The vacuum and creation properties are
$Y(\mathbf1,x)=\id_V$ and $Y(u,x)\mathbf1\in u+xV[[x]]$.
The VOA Jacobi identity and the Virasoro relations for
$Y(\omega,x)=\sum_{n\in\mathbb Z}L(n)x^{-n-2}$ are understood as in
\cite{FHL}; in particular, $L(0)|_{V_d}=d\,\id$ and
$Y(L(-1)u,x)=\partial_xY(u,x)$.  Write $\wt u=d$ for
$u\in V_d$.  An automorphism preserves $Y$, $\mathbf1$ and $\omega$.
The VOA is \emph{of CFT type} if $V_d=0$ for $d<0$ and
$V_0=\C\mathbf1$; it is \emph{simple} if it has no proper nonzero
ideal.  Define
$C_2(V)=\operatorname{span}_{\C}\{u_{-2}v:u,v\in V\}$;
$C_2$-cofiniteness means $\dim_{\C}V/C_2(V)<\infty$.
The \emph{restricted contragredient} $V'$ is its graded dual with
the contragredient vertex operators \cite{FHL}.  The condition
$V\simeq V'$ as $V$-modules is called \emph{self-contragredience};
equivalently, $V$ admits a nondegenerate invariant bilinear form.
Here invariance means
$(Y(u,x)v,w)=(v,Y(e^{xL(1)}(-x^{-2})^{L(0)}u,x^{-1})w)$.

Fix a finite-order $g\in\Aut(V)$, let $N=\operatorname{ord}(g)$, and
write
\[
 V^{(r)}=\{u\in V:gu=e^{2\pi i r/N}u\},\qquad
 V=\bigoplus_{r=0}^{N-1}V^{(r)}.
\]
We recall the finite-order \emph{formal} twisted-module convention
of \cite[Definition~3.1]{DLM}.  A weak $g$-twisted module is a vector
space $M$ with a linear state--field map $Y_M^g$, such that for
$u\in V^{(r)}$ one has
\[
 Y_M^g(u,x)=\sum_{n\in r/N+\mathbb Z}u_n^M x^{-n-1},
 \qquad u_n^Mw=0\quad(n\gg0,\ w\in M),
\]
$Y_M^g(\mathbf1,x)=\id_M$, and the \emph{twisted Jacobi identity}
\begin{align}\label{eq:twistedjacobi}
 &x_0^{-1}\delta\!\left(\frac{x_1-x_2}{x_0}\right)
  Y_M^g(u,x_1)Y_M^g(v,x_2)
 -x_0^{-1}\delta\!\left(\frac{x_2-x_1}{-x_0}\right)
  Y_M^g(v,x_2)Y_M^g(u,x_1)\notag\\
 &\quad=x_2^{-1}\left(\frac{x_1-x_0}{x_2}\right)^{-r/N}
  \delta\!\left(\frac{x_1-x_0}{x_2}\right)
  Y_M^g(Y(u,x_0)v,x_2),\qquad u\in V^{(r)},\ v\in V.
\end{align}
Here $\delta(x)=\sum_{j\in\mathbb Z}x^j$ and the formal binomial
expressions are expanded in their prescribed nonnegative powers of
the second variable.  The fractional factor in
\eqref{eq:twistedjacobi} agrees with the \emph{positive} eigenvalue
convention above: $gu=e^{2\pi ir/N}u$ gives
$n\in r/N+\mathbb Z$, not $-r/N+\mathbb Z$.
Since $\omega\in V^{(0)}$, the field
$Y_M^g(\omega,x)=\sum_{m\in\mathbb Z}L_M(m)x^{-m-2}$
has integral modes; \eqref{eq:twistedjacobi} implies the usual
Virasoro relations and, for $u\in V_d$,
$[L_M(0),u_n^M]=(d-n-1)u_n^M$ \cite[Section~3]{DLM}.
A twisted-module morphism is a linear map commuting with all the
fields $Y_M^g(u,x)$ (equivalently, all twisted modes).

For $\lambda\in\C$, set
$M_{[\lambda]}=\bigcup_{k\ge1}\ker(L_M(0)-\lambda)^k$.
A \emph{generalized} $g$-twisted module is a weak twisted module
with $M=\bigoplus_{\lambda\in\C}M_{[\lambda]}$.
It is \emph{lower bounded} in the grading-restricted class used
here if $M_{[\lambda]}=0$ whenever $\Re\lambda<B$ for some
$B\in\mathbb R$, and \emph{grading restricted} if in addition
$\dim M_{[\lambda]}<\infty$ for every $\lambda$.
It is \emph{ordinary} if $L_M(0)$ acts semisimply; in that case
$M_\lambda=M_{[\lambda]}$ is the genuine eigenspace.
For general $M$ we reserve brackets for generalized eigenspaces,
and write $\pi^M_\lambda:M\to M_{[\lambda]}$ for the canonical
projection.  Thus
$u_n^M M_{[\lambda]}\subseteq
M_{[\lambda+d-n-1]}$ when $u\in V_d$.
For a grading-restricted generalized module put
$\overline M=\prod_\lambda M_{[\lambda]}$; when $M$ is ordinary,
this is $\prod_\lambda M_\lambda$, and $\pi^M_\lambda$
projects to $M_\lambda$.

For comparison with the preceding generalized-weight terminology,
an \emph{admissible} $g$-twisted module in the sense of
\cite[Definition~3.3]{DLM} is a weak twisted module with a grading
$M=\bigoplus_{j\in N^{-1}\mathbb Z_{\ge0}}M(j)$ satisfying
$u_n^M M(j)\subseteq M(j+d-n-1)$ for $u\in V_d$; no
semisimplicity of $L_M(0)$ or finite dimensionality of $M(j)$ is
required in this definition.  \emph{$g$-rationality} means complete
reducibility of admissible $g$-twisted modules, whereas
\emph{$g$-regularity} means that \emph{every weak} $g$-twisted module
is an algebraic direct sum of ordinary irreducible $g$-twisted
modules.  For $g=1$ these are called \emph{rationality} and
\emph{regularity}, respectively.  We call $V$ \emph{strongly
rational} if it is simple, of CFT type, self-contragredient,
rational and $C_2$-cofinite.  In particular, rationality here
concerns admissible untwisted modules, not arbitrary weak modules.

Du--Huang's generalized $g$-twisted modules without a $g$-action
are defined by analytic product--iterate duality, branch equivariance,
$L(0)$-grading and the $L(-1)$-derivative property
\cite[Section~2]{DH}.  The comparison with the formal convention
above is valid in the grading-restricted finite-order setting, as
follows.

\begin{proposition}\label{prop:modulecomparison}
Let $g$ have finite order and let $M$ be equipped with a direct sum of
finite-dimensional generalized $L_M(0)$-weight spaces whose real
parts have a common lower bound.  For the same state--field map,
the following are equivalent:
\begin{enumerate}[label=\textup{(\roman*)},leftmargin=*]
\item $M$ is a grading-restricted generalized analytic $g$-twisted
module without an auxiliary $g$-action in the sense of
\cite[Section~2]{DH};
\item $M$ is a grading-restricted generalized module satisfying the
finite-order formal weak $g$-twisted Jacobi identity of
\cite[Definition~3.1]{DLM}, with its indicated generalized weight
grading.
\end{enumerate}
The correspondence leaves the vertex operators and module maps
unchanged.  In particular, it applies to ordinary irreducible
finite-order formal twisted modules.
\end{proposition}
\begin{proof}
First consider (i).  For $u\in V^{(r)}$, compare the two sides of
Du--Huang's branch identity
$(Y_M^g)^{p+1}(gu,z)=(Y_M^g)^p(u,z)$.
For a fixed power $x^{-n-1}$, its logarithmic coefficient is a
polynomial $P(t)$ in $t=\log z$ satisfying
$e^{2\pi i(r/N-n-1)}P(t+2\pi i)=P(t)$.
The highest nonzero coefficient first forces
$e^{2\pi i(r/N-n-1)}=1$; comparison of the next coefficient
then forces $P$ to be constant.  Consequently
$Y_M^g(u,x)=\sum_{n\in r/N+\mathbb Z}u_n^M x^{-n-1}$ has no
logarithmic terms.  The $L(0)$-bracket identity follows from
analytic duality and the Virasoro field, while the stated lower
bound and the weight-shift formula give lower truncation; see
\cite[Section~2, following the definition of generalized twisted
modules]{DH}.

We apply the analytic--Jacobi comparison in
\cite[Theorems~2.6, 2.7 and 2.10]{HuangYangAssoc}.  Its auxiliary
$\overline{\mathbb C}_+$- and $g$-gradings may be supplied here from
the given $L_M(0)$-decomposition: translate the weight indices by
a sufficiently small real number so that their real parts are
nonnegative, and take the generalized eigenspace grading of
$g_M:=e^{-2\pi iL_M(0)}$.  This auxiliary operator supplies the grading compatibility required
in the analytic--Jacobi comparison; objects of $\cC_g$ retain their
given twisted-module data.  The $L(0)$-bracket and the integral weights
of $V$ give $g_MY_M^g(u,x)g_M^{-1}=Y_M^g(gu,x)$ and the
required grading compatibility.  Thus the grading and equivariance
conditions of \cite[Definition~2.1]{HuangYangAssoc} are satisfied.
In their notation the nilpotent part $\mathcal N$ of $g$ on $V$
is zero, so their logarithm-free field $Y_0^g$ is precisely
$Y_M^g$.  Theorem~2.10, (1)$\Rightarrow$(2)$\Rightarrow$(3),
then supplies the full twisted Jacobi identity.  Their expression
(2.29) is the standard twisted delta identity, equivalent to
\eqref{eq:twistedjacobi} with $a=r/N$.  The comparison is made by extracting formal coefficients from the
absolutely convergent product, commuted-product and iterate branches.

Conversely, start with (ii).  The formal twisted Jacobi identity
supplies the $L(0)$-bracket, Virasoro relations and
$L(-1)$-derivative, as in \cite[Section~3]{DLM}.  The same
construction of the auxiliary grading and $g_M$ applies; the
formal $g$-eigenspace condition gives branch equivariance.
Theorem~2.10 of \cite{HuangYangAssoc}, now in the direction
(3)$\Rightarrow$(4)$\Rightarrow$(2)$\Rightarrow$(1), gives
analytic duality, including \emph{absolute} convergence, rather
than only formal rationality.  More explicitly, its Theorem~2.6
establishes convergence in
$|z_1|>|z_2|>0$, $|z_2|>|z_1|>0$ and
$|z_2|>|z_1-z_2|>0$, with the prescribed compatible branches.
For $u\in V^{(r)}$, $v\in V^{(s)}$ and vectors paired with the
restricted dual, the common multivalued continuation has the form
\[
 z_1^{-r/N}z_2^{-s/N}
 \frac{P(z_1,z_2)}{z_1^a z_2^b(z_1-z_2)^c},
 \qquad P\in\mathbb C[z_1,z_2],\quad a,b,c\in\mathbb Z_{\ge0},
\]
with finite sums for inhomogeneous inputs.  Expansions of its
rational factors in the respective chambers are absolutely
convergent on compact subsets.  The restricted-dual contractions
therefore yield the required projected weight sums, not merely an
identity of formal Laurent series.  Finite-dimensional generalized
weight spaces and the lower bound give the other grading
requirements of \cite[Section~2]{DH}.  The proofs of
\cite[Theorems~2.6 and 2.10]{HuangYangAssoc} supply the branch and
product--iterate comparison used here.  Since the two realizations
use exactly the same operators, their intertwining morphisms also
coincide.
\end{proof}

Theorem~2.4 of \cite{YangTwAssoc} supplies a complementary
formal criterion: for finite-order $g$, the vacuum and
lower-truncation properties together with weak associativity
\emph{with the exponent in $r/N+\mathbb N$} imply the full twisted
Jacobi identity.  The absolute-convergence and branch statements needed to pass
between the analytic and formal formulations above are provided by
\cite{HuangYangAssoc}.

\subsection{The twisted-module category, duals, and dimensions}\label{subsec:modulecategory}

Under the present strong-rationality hypotheses the required
semisimplicity has a precise external source.  By
\cite[Theorem~4.5]{ABD}, $V$ is regular, since it is of CFT type,
rational and $C_2$-cofinite.  The simple VOA $V$ has nonnegative
$L(0)$-spectrum and a nondegenerate invariant form; hence
\cite[Corollary~5.26]{CarnahanMiyamoto} applies to every finite-order
$g$ and gives both $g$-rationality and $g$-regularity for the
formal weak twisted modules.  Then \cite[Theorem~8.1(b),(c)]{DLM}
gives finitely many irreducible admissible $g$-twisted types, all
ordinary; conversely ordinary irreducibles admit admissible
gradings by \cite[Lemma~3.4]{DLM}.  In particular, a grading-restricted
generalized analytic object is, by Proposition~\ref{prop:modulecomparison}, a
\emph{finite} direct sum of ordinary irreducibles: weak-module
regularity decomposes it into simples, and finite dimensionality of
a chosen nonzero weight space for each of the finitely many simple
types bounds the multiplicity of that type.  The nonemptiness of
each sector follows separately from
\cite[Theorem~9.1]{DLMModular}, which assumes simplicity,
$C_2$-cofiniteness and finite order of $g$.  The module-level formal--analytic comparison is supplied here by
Proposition~\ref{prop:modulecomparison}; the crossed tensor
products and their coherence are provided by Xu's construction
\cite[Theorem~1.1]{Crossed}.

Let $\cC_g$ be the category of grading-restricted generalized
$g$-twisted $V$-modules in the analytic sense used by \cite{DH,Crossed},
with the finite-order formal realization established by
Proposition~\ref{prop:modulecomparison}.  Its morphisms are the
weight-preserving twisted-module maps; for objects in one sector we
also denote these spaces by $\Hom_V(X,Y)$.  Define the \emph{external}
direct sum $\cC_G=\bigoplus_{g\in G}\cC_g$ to have objects finite
direct sums of pairs $(g,X)$, $X\in\cC_g$, and to have
$\Hom_{\cC_G}((g,X),(h,Y))=0$ when $g\ne h$ and
$\Hom_{\cC_g}(X,Y)$ when $g=h$, extended additively to sums.
The label $g$ is part of the object and is not determined by an
identification of the underlying vector spaces.  The $G$-crossed
tensor product, action and braiding on this additive category are
\emph{additional structures} imported from \cite{Crossed}, not
part of the definition of a twisted module.
Finally, $\Rep(V)$ denotes the category of ordinary untwisted
$V$-modules, with the Huang--Lepowsky vertex tensor product
\cite{HL}.  Under the hypotheses above this is a fusion category,
and in fact a modular tensor category, by Huang's rigidity and
modularity theorem \cite{Huang}.
For a grading-restricted generalized $g$-twisted module
$X=\bigoplus_{\lambda}X_{[\lambda]}$, its \emph{restricted
contragredient} is the algebraic direct sum
$DX=X'=\bigoplus_{\lambda}X_{[\lambda]}^*$, where
$X_{[\lambda]}^*=\Hom_{\C}(X_{[\lambda]},\C)$.
The weight-$\lambda$ generalized eigenspace of $DX$ is
$X_{[\lambda]}^*$; in particular, if $X$ is ordinary, then
$DX=\bigoplus_\lambda X_\lambda^*$.  With the evaluation
pairing $\langle\ ,\ \rangle:X'\times X\to\C$, its twisted
vertex operator is specified by
\begin{equation}\label{eq:restrictedcontragredient}
 \bigl\langle Y_{DX}^{g^{-1}}(u,x)\varphi,w\bigr\rangle
 =\bigl\langle\varphi,
 Y_X^g\bigl(e^{xL(1)}(-x^{-2})^{L(0)}u,x^{-1}\bigr)w\bigr\rangle,
 \quad u\in V,\ \varphi\in DX,\ w\in X.
\end{equation}
Here $(-x^{-2})^{L(0)}u=(-1)^d x^{-2d}u$ for $u\in V_d$;
the formula is expanded formally with the opposite coordinate
$x^{-1}$.  The factor $e^{xL(1)}$ is a finite sum on each
homogeneous $u$.  The contragredient construction of
\cite[Theorem~5.2.1]{FHL}, its admissible twisted analogue
\cite[Lemma~3.7]{DLM}, and the analytic generalized twisted
construction of \cite[Section~2]{DH} give a grading-restricted
$g^{-1}$-twisted module in the present setting.  Indeed,
for $u\in V^{(r)}$ the powers in
\eqref{eq:restrictedcontragredient} have mode indices in
$-r/N+\mathbb Z$, as required for $g^{-1}$.  This
construction uses the weightwise algebraic dual, not the full
linear dual $X^*$ or the completion $\overline X$.

For a twisted-module morphism $f:X\to Y$ in a fixed sector,
$Df:Y'\to X'$ is its weightwise transpose,
$\langle Df(\psi),w\rangle=\langle\psi,f(w)\rangle$.
Thus $D$ is contravariant and reverses the sector degree.
The canonical evaluation into the restricted double dual,
\[
 \delta_X:X\longrightarrow D^2X,\qquad
 \delta_X(w)(\varphi)=\varphi(w),
\]
is a natural twisted-module isomorphism: each generalized
weight space is finite-dimensional and the second application of
\eqref{eq:restrictedcontragredient} recovers the original
vertex operators, as in the double-contragredient identity
\cite[Proposition~5.3.1]{FHL}.  The comparison with the \emph{categorical} right dual $X^\vee$
is established from pairing and rigidity in Proposition~\ref{prop:dual}.

The pullback $T_k$ preserves $L(0)$ and the underlying graded
vector space.  Consequently $D(T_kX)$ and $T_k(DX)$ both belong
to $\cC_{kg^{-1}k^{-1}}$, and the identity on their common
restricted-dual vector space defines a natural isomorphism
$\tau_{k,X}:D(T_kX)\xrightarrow{ \sim }T_k(DX).$
For $u\in V$ and $w\in X$, this is an intertwiner since
\[
 \langle Y_{D(T_kX)}(u,x)\varphi,w\rangle
 =\langle\varphi,Y_X^g(k^{-1}e^{xL(1)}
 (-x^{-2})^{L(0)}u,x^{-1})w\rangle
 =\langle Y_{T_k(DX)}(u,x)\varphi,w\rangle,
\]
where $k^{-1}$ commutes with $L(1)$ and $L(0)$ because
$k$ fixes $\omega$.
With the underlying functor identity $T_kT_l=T_{kl}$, its
coherence is $\tau_{kl,X}=T_k(\tau_{l,X})\,
\tau_{k,T_lX}$.  These are natural comparisons of twisted-module functors. The
monoidal constraints enter in the later comparison with categorical
duals.

For a grading-restricted generalized module $M$, the
\emph{weight completion} is the direct product
$\overline M=\prod_\lambda M_{[\lambda]}$, as defined above,
with coordinate projections
$\pi^M_\lambda:\overline M\to M_{[\lambda]}$ extending the
projections on $M$.  The restricted pairing extends to
$M'\times\overline M$ by
$\langle(\varphi_\lambda)_\lambda,(m_\lambda)_\lambda\rangle
=\sum_\lambda\varphi_\lambda(m_\lambda)$; this sum is finite
because $M'$ is a \emph{direct sum}.  It separates points of
$\overline M$.  Since a twisted-module map $f:M\to N$ commutes
with $L(0)$, it preserves each generalized weight space and has
the canonical componentwise extension
\[
 \overline f:\overline M\longrightarrow\overline N,\qquad
 (m_\lambda)_\lambda\longmapsto
 (f|_{M_{[\lambda]}}m_\lambda)_\lambda.
\]
For ordinary modules, the same statements use eigenspaces
$M_\lambda$ in place of $M_{[\lambda]}$.  The completed
codomains of the canonical intertwining operators below are
understood in this sense, not as restricted double duals.

For a fusion category $\mathcal D$, the
\emph{Frobenius--Perron dimension} $\FPdim_{\mathcal D}(X)$
of an object is obtained from the positive dimension character
on its Grothendieck ring: on a simple object $X$ it is the
Perron--Frobenius eigenvalue of left multiplication by $[X]$,
and it is extended additively (and multiplicatively) to objects
\cite[Definition~3.3.3 and Proposition~3.3.6]{EGNO}.  In
particular, $\FPdim_{\mathcal D}(X)>0$ for $X\ne0$, and
$\FPdim(\mathcal D) =\sum_{[X]\in\Irr(\mathcal D)}\FPdim_{\mathcal D}(X)^2.$
The homogeneous component $\cC_g$ need not itself be monoidal.
Throughout, its dimension means
$\FPdim(\cC_g):= \sum_{[X]\in\Irr(\cC_g)}\FPdim_{\cC_G}(X)^2,$
with $X$ regarded as a labeled object of the \emph{whole}
fusion category $\cC_G$ furnished by \cite{Crossed}.
This is a finite positive real number because the grading is
faithful and $\cC_G$ is a fusion category.  On the neutral
fusion subcategory the dimension character restricts to its
intrinsic Frobenius--Perron dimension.  The later equality of
homogeneous-component dimensions uses the \emph{faithful} finite
$G$-grading in \cite[Proposition~8.20]{ENO}, while the
corresponding equivariantization formula uses a finite group
acting by tensor autoequivalences on a fusion category
\cite[Proposition~4.26]{DGNO}.

\subsection{The \texorpdfstring{$G$}{G}-crossed categorical structure}\label{subsec:crossedconventions}

The $G$-grading of the fusion category $\cC_G$ means that its
homogeneous summands $\cC_g$ are full abelian subcategories,
$\Hom_{\cC_G}(X_g,Y_h)=0$ for $g\ne h$, every object is a finite
direct sum of homogeneous objects, and the tensor unit $V$ belongs
to $\cC_1$.  The tensor product and the pullback functors have the
following conventions:
\begin{equation}\label{eq:conventions}
 \cC_g\bo\cC_h\subseteq\cC_{gh},\quad
 T_k:\cC_g\to\cC_{kgk^{-1}},\quad
 Y_{T_kX}(v,z)=Y_X(k^{-1}v,z),\quad T_kT_l=T_{kl}.
\end{equation}
The grading is \emph{faithful} if every $\cC_g$ is nonzero.
The categorical right dual of a homogeneous object of degree $g$
can be chosen in $\cC_{g^{-1}}$, by rigidity and the grading.
These grading and tensor structures are those established in
\cite[Theorem~1.1, Proposition~8.5, Proposition~9.6]{Crossed}; the
external sector label is retained even when the underlying vector
spaces of objects agree.

The functor $T_k$ acts as the identity on the underlying vector
spaces and linear module morphisms, but pulls back the vertex
operators as in \eqref{eq:conventions}.  Indeed, if
$X\in\cC_g$ and $v$ is an eigenvector of $kgk^{-1}$ with
eigenvalue $e^{2\pi ir/N}$, then $k^{-1}v$ is a $g$-eigenvector
with the same eigenvalue.  Thus the fractional modes and the
twisted-module identities are preserved.  The inverse functor is
$T_{k^{-1}}$, so $T_k$ is a $\C$-linear exact autoequivalence.
Applying two pullbacks gives
$Y_{T_kT_lX}(v,z)=Y_X(l^{-1}k^{-1}v,z)$, and the equality
$T_kT_l=T_{kl}$, as well as $T_1=\id$, holds literally on
objects and morphisms of the externally labeled category.  This
strictness concerns the \emph{underlying functors} only.

Each $T_k$ is equipped with a strong monoidal structure.  We use
the tensorators and unit maps
\[
 J^k_{X,Y}:T_kX\bo T_kY\xrightarrow{\sim}T_k(X\bo Y),
 \qquad j_k:\mathbf1\xrightarrow{\sim}T_k\mathbf1,
\]
where $\mathbf1$ in categorical formulas denotes the tensor
unit $V$.  On $V$ the canonical unit comparison of
\cite{Crossed} is $j_k(v)=k^{-1}v$.  In general neither
$J^k$ nor $j_k$ is the identity.  Write
$a_{X,Y,Z}:(X\bo Y)\bo Z\to X\bo(Y\bo Z)$ for the associator,
and $l_X:\mathbf1\bo X\to X$ and
$r_X:X\bo\mathbf1\to X$ for the unitors.  The tensorators
and unit maps are natural isomorphisms and obey
\begin{equation}\label{eq:actioncoherence}
 J^{kl}_{X,Y}=T_k(J^l_{X,Y})J^k_{T_lX,T_lY},
 \qquad j_{kl}=T_k(j_l)j_k,\qquad J^1=\id,\quad j_1=\id.
\end{equation}
Compositions are read from right to left.  Both sides of the
first identity belong to
$\Hom(T_{kl}X\bo T_{kl}Y,T_{kl}(X\bo Y))$; both
sides of the second belong to
$\Hom(\mathbf1,T_{kl}\mathbf1)$.  The associativity and
left- and right-unit conditions for each strong monoidal functor
are, with the same tensorator direction,
\begin{align*}
&T_k(a_{X,Y,Z})J^k_{X\bo Y,Z}(J^k_{X,Y}\bo\id_{T_kZ})\notag\\
&\qquad=J^k_{X,Y\bo Z}(\id_{T_kX}\bo J^k_{Y,Z})
 a_{T_kX,T_kY,T_kZ},\\
&T_k(l_X)J^k_{\mathbf1,X}(j_k\bo\id_{T_kX})
 =l_{T_kX},\notag\\
&T_k(r_X)J^k_{X,\mathbf1}(\id_{T_kX}\bo j_k)
 =r_{T_kX}.
\end{align*}
The first identity compares morphisms
$(T_kX\bo T_kY)\bo T_kZ\to T_k(X\bo(Y\bo Z))$;
the unit identities have domains $\mathbf1\bo T_kX$
and $T_kX\bo\mathbf1$, respectively.  In the general
formalism the action has monoidal compositors
$T_kT_l\Rightarrow T_{kl}$; here they are identity natural
transformations, without imposing strict monoidality on $T_k$.

For a homogeneous $X\in\cC_g$ and any $Y\in\cC_G$, the
crossed braiding is a natural isomorphism
$c_{X,Y}:X\bo Y\xrightarrow{\sim}T_gY\bo X$.
It is defined componentwise when $Y$ is a direct sum and is
extended additively in the first argument: for $X=\bigoplus_gX_g$, its target is the direct sum $\bigoplus_g(T_gY\bo X_g)$, with each summand acted on by the corresponding functor $T_g$.  If $Y\in\cC_h$, both endpoints
have degree $gh$, because
$(ghg^{-1})g=gh$.  The action-equivariance and two crossed
hexagons are the following equations:
\begin{align}
&T_k(c_{X,Y})J^k_{X,Y}
 =J^k_{T_gY,X}c_{T_kX,T_kY},\label{eq:crossedequiv}\\
&(\id_{T_gY}\bo c_{X,Z})a_{T_gY,X,Z}
 (c_{X,Y}\bo\id_Z)\notag\\
&\quad=a_{T_gY,T_gZ,X}
 ((J^g_{Y,Z})^{-1}\bo\id_X)c_{X,Y\bo Z}a_{X,Y,Z},
 \label{eq:crossedhexone}\\
&(c_{X,T_hZ}\bo\id_Y)a^{-1}_{X,T_hZ,Y}
 (\id_X\bo c_{Y,Z})\notag\\
&\quad=a^{-1}_{T_{gh}Z,X,Y}c_{X\bo Y,Z}a^{-1}_{X,Y,Z}.
 \label{eq:crossedhextwo}
\end{align}
In the first identity the common source is $T_kX\bo T_kY$
and the common target is $T_k(T_gY\bo X)$, using
$T_{kgk^{-1}}T_k=T_kT_g$.  For $X\in\cC_g$,
$Y\in\cC_h$ and arbitrary $Z$, the two hexagons have,
respectively, source and target
\[
 (X\bo Y)\bo Z\ \longrightarrow\ T_gY\bo(T_gZ\bo X),
 \qquad
 X\bo(Y\bo Z)\ \longrightarrow\ (T_{gh}Z\bo X)\bo Y.
\]
The inverse $(J^g_{Y,Z})^{-1}$ in the first hexagon is
necessary to expand $T_g(Y\bo Z)$, and the second uses
$T_gT_h=T_{gh}$ without interchanging $g$ and $h$.

The corresponding crossed unit identities are
\[
 r_Yc_{\mathbf1,Y}=l_Y,\qquad
 l_X(j_g^{-1}\bo\id_X)c_{X,\mathbf1}=r_X
 \quad (X\in\cC_g).
\]
where the first identity uses $T_1=\id$.  They are the
unit comparisons of the canonical crossed commutors in
\cite{Crossed}; in particular, $j_g$ must not be discarded.
Equations \eqref{eq:crossedequiv}--\eqref{eq:crossedhextwo}
are the non-strict forms of \cite[Definition~4.41]{DGNO}
and \cite[Definition~8.24.1, (8.105)--(8.107)]{EGNO}, with
$J^k$ in the \emph{strong monoidal} direction and identity
underlying action compositors.  The sources use action
compositors and may suppress associators; the displayed
formulas retain both tensorators and associators.  Kirillov's
$G$-equivariant fusion convention \cite[Section~2]{Kirillov}
uses action functors $R_k$ and suppresses canonical coherence
isomorphisms; its additional pivotal and twist data are not part
of the bare braided $G$-crossed definition.

For $X\in\cC_g$ put $\operatorname{deg}(X)=g$.  Fix right
duals $X^\vee$ with
$\ev_X:X\bo X^\vee\to\mathbf1$ and
$\coev_X:\mathbf1\to X^\vee\bo X$, satisfying the two
triangle identities.  The right-dual transposition of a map
$f:A\to B$ is $f^\vee:B^\vee\to A^\vee$.
The action--dual comparison $b_{k,X}:T_k(X^\vee)\to(T_kX)^\vee$
will be uniquely defined by evaluations in
Lemma~\ref{lem:dualactioncomparison}.

\begin{definition}\label{def:twistaxioms}
A \emph{crossed ribbon twist} in our (positive-half-turn,
right-dual) convention is a natural collection of isomorphisms
$\theta_X:X\to T_gX$ for $X\in\cC_g$ such that
\begin{enumerate}[label=\textup{(\roman*)},leftmargin=*]
\item $\theta_{\mathbf1}=j_1$ (under the canonical unit
identifications), and
$\theta_{T_kX}=T_k(\theta_X)$ with
$T_{kgk^{-1}}T_k=T_kT_g$;
\item if $X\in\cC_g$, $Y\in\cC_h$ and $q=ghg^{-1}$, then
\begin{equation}\label{eq:defbalance}
\theta_{X\bo Y}=J^{gh}_{X,Y}
 (T_q\theta_X\bo T_g\theta_Y)
 c_{T_gY,X}c_{X,Y};
\end{equation}
\item for $X\in\cC_g$, with the action--right-dual comparison
$b_{g,X}:T_g(X^\vee)\to(T_gX)^\vee$ of
Lemma~\ref{lem:dualactioncomparison}, the dual-twist identity is
\begin{equation}\label{eq:categoricaldualtwistaxiom}
 \theta_{X^\vee}
 =T_{g^{-1}}(\theta_X^\vee)\,
  T_{g^{-1}}(b_{g,X}):
 X^\vee\longrightarrow T_{g^{-1}}X^\vee.
\end{equation}
The equivalent identity on the restricted contragredient $DX$
is proved in Proposition~\ref{prop:dual}.
\end{enumerate}
These are the typed forms of Kirillov's four conditions in
\cite[Lemma~2.3]{Kirillov}; the unit, covariance, balancing and
dual compatibility in the source correspond to (i), (i), (ii)
and (iii) here.  Kirillov writes the group action as $R_k$ and
suppresses monoidal constraints; replacing $R_k$ by our $T_k$
requires the displayed $J^k$, $j_k$ and duality comparisons.
When $G=\{1\}$, the crossed braiding and twist reduce to the
usual ribbon-category data.
\end{definition}

\begin{definition}\label{def:crossedmodular}
A \emph{braided $G$-crossed fusion category} is a fusion category
$\cC=\bigoplus_{g\in G}\cC_g$ with a monoidal action $T_k$ of $G$,
$T_k\cC_g=\cC_{kgk^{-1}}$, and natural crossed braidings
$c_{X,Y}:X\bo Y\to T_gY\bo X$ for $X\in\cC_g$, satisfying the
unit, action-equivariance and two crossed hexagon axioms.
The grading is \emph{faithful} if every $\cC_g$ is nonzero.
A \emph{$G$-crossed ribbon category} has, in addition, a rigid
structure and natural crossed twists
$\theta_X:X\to T_gX$ ($X\in\cC_g$), with the unit,
$G$-equivariance, crossed balancing and duality conditions stated
explicitly in Definition~\ref{def:twistaxioms} above.

For this paper, such a $G$-crossed ribbon fusion category is
\emph{$G$-crossed modular} if its braided ribbon
equivariantization $\cC^G$ is nondegenerate.  By \cite[Theorem~10.5]{Kirillov}, nondegeneracy of the equivariantization
is equivalent to invertibility of the extended $\widetilde S$-operator
of \cite[Definition~10.1 and (9.3)]{Kirillov}.  Kirillov calls this
a ``$G$-equivariant modular category''; here we use the
synonymous phrase ``$G$-crossed modular category'' to emphasize
the crossed braiding.  Unlike the weaker condition that only
$\cC_1$ be modular, this condition requires nondegeneracy of
$\cC^G$.
\end{definition}

More concretely, the underlying space of Kirillov's extended
Verlinde algebra has bigraded pieces
\[
 \widetilde{\mathcal V}_{k,g}
 =\bigoplus_{[X]\in\Irr(\cC_g)}\Hom_{\cC}(X,T_kX),
\]
where one representative of each simple is chosen and a summand
vanishes unless $T_kX\simeq X$.  The extended
$\widetilde S$-operator constructed from crossed braiding and
categorical traces in \cite[(9.3)]{Kirillov} exchanges
$\widetilde{\mathcal V}_{k,g}$ and
$\widetilde{\mathcal V}_{g^{-1},k}$.  By
\cite[Theorem~10.5]{Kirillov}, its invertibility is equivalent
to the completely intrinsic M\"uger-center criterion of
Definition~\ref{def:crossedmodular}; that criterion is what
we verify below.  We use this categorical criterion throughout.

\subsection{Ordinary and complete intertwining operators}\label{subsec:IOdefinitions}
\begin{definition}\label{def:ordinaryIO}
Let $X,Y,Z$ be ordinary $V$-modules.  Write $Z\{x\}$ for
formal $Z$-valued series with complex powers of $x$; as usual,
$x^\alpha$ is a formal symbol and distinct powers are compared
coefficientwise.  An \emph{ordinary intertwining operator of type
$\binom Z{X\;Y}$} is a linear map
$\mathcal Y:X\otimes Y\to Z\{x\}$, written
$\mathcal Y(a,x)b=\sum_{n\in\C}a^{\mathcal Y}_n b\,x^{-n-1}$,
which satisfies the following properties for $a\in X$, $b\in Y$
and $u\in V$:
\begin{enumerate}[label=\textup{(\roman*)},leftmargin=*]
\item \emph{Lower truncation}: for every $n\in\C$,
$a^{\mathcal Y}_{n+l}b=0$ for all sufficiently large
integers $l\geq0$ (with the bound allowed to depend on $a,b,n$).
\item The $L(-1)$-derivative identity
$\partial_x\mathcal Y(a,x)b=\mathcal Y(L_X(-1)a,x)b$.
\item The Jacobi identity
\begin{align}\label{eq:IOJacobi}
&x_0^{-1}\delta\!\left(\frac{x_1-x_2}{x_0}\right)
 Y_Z(u,x_1)\mathcal Y(a,x_2)b
-x_0^{-1}\delta\!\left(\frac{x_2-x_1}{-x_0}\right)
 \mathcal Y(a,x_2)Y_Y(u,x_1)b\notag\\
&\qquad=x_2^{-1}\delta\!\left(\frac{x_1-x_0}{x_2}\right)
 \mathcal Y(Y_X(u,x_0)a,x_2)b.
\end{align}
\end{enumerate}
Here $\delta(x)=\sum_{m\in\mathbb Z}x^m$ and each formal
binomial is expanded in nonnegative powers of its second variable
as prescribed by the usual formal-variable convention.  This is
the Jacobi-identity definition of \cite[Definition~5.4.1]{FHL};
we permit complex rather than only rational powers in the
ambient notation $Z\{x\}$, as in the generalized formulation.
The $L(0)$-bracket follows from the Jacobi identity and derivative
property \cite[(5.4.12)--(5.4.14)]{FHL}.  Thus, for homogeneous
$a\in X_\alpha$ and $b\in Y_\beta$,
$a^{\mathcal Y}_n b\in Z_{\alpha+\beta-n-1}$; the nonzero
powers are constrained by the weights of $Z$, and for the
ordinary modules considered here they lie in finitely many
cosets of $\mathbb Z$.  Pre- and postcomposition with module
homomorphisms acts on intertwining operators in the evident
variance.  The resulting vector space is denoted
$\mathcal I_{\mathrm{HL}}\binom Z{X\;Y}$.
\end{definition}

\begin{definition}\label{def:completeIO}
Let $X,Y,Z$ be ordinary untwisted $V$-modules.  We call these operators \emph{complete Du--Huang intertwining
operators}; they are the twisted intertwining operators of
\cite[Definition~2.8]{DH} specialized to three untwisted modules.
It is a linear map $\mathcal Y:X\otimes Y\to Z\{x\}[\log x]$,
$\mathcal Y(a,x)b=\sum_{j=0}^{K}\sum_{n\in\C}
 a^{\mathcal Y}_{n,j}b\,x^{-n-1}(\log x)^j$,
with a finite logarithmic bound $K$, subject to these four
conditions:
\begin{enumerate}[label=\textup{(\roman*)},leftmargin=*]
\item For all $a,b,n,j$ as above,
$a^{\mathcal Y}_{n+l,j}b=0$ when $l\in\mathbb Z_{\geq0}$
is sufficiently large.
\item For $u\in V$, $a\in X$, $b\in Y$ and $z'\in Z'$, there
is a maximally extended multivalued analytic function
$f(z_1,z_2)$ on
$M^2=\{(z_1,z_2)\in(\C^\times)^2:z_1\ne z_2\}$,
with a preferred branch $f^e$ on $M_0^2$.  Here $M_0^2$
is obtained by cutting in the $z_1$-, $z_2$- and
$(z_1-z_2)$-planes along the positive real rays and taking
the upper-side boundary values, as in \cite[Section~2]{DH}.
With $\pi^M_\lambda$ the weight projection on $M$, the three
projected matrix-coefficient series are respectively
\begin{align*}
&\sum_\lambda\langle z',Y_Z(u,z_1)\pi^Z_\lambda
           \mathcal Y(a,z_2)b\rangle,\\
&\sum_\lambda\langle z',\mathcal Y(a,z_2)\pi^Y_\lambda
           Y_Y(u,z_1)b\rangle,\\
&\sum_\lambda\langle z',\mathcal Y(
           \pi^X_\lambda Y_X(u,z_1-z_2)a,z_2)b\rangle.
\end{align*}
They are absolutely convergent for, respectively,
$|z_1|>|z_2|>0$, $|z_2|>|z_1|>0$ and
$|z_2|>|z_1-z_2|>0$, and their sums equal the same
preferred branch $f^e$ on the corresponding subregions
\begin{align*}
&|\arg(z_1-z_2)-\arg z_1|<\pi/2,\\
&-3\pi/2<\arg(z_1-z_2)-\arg z_2<-\pi/2,\\
&|\arg z_1-\arg z_2|<\pi/2,
\end{align*}
within these three convergence chambers, respectively.
All powers and logarithms are evaluated on these prescribed
branches; continuation around the punctures gives the other
branches, as in \cite[Definition~2.8 and the following remark]{DH}.
\item For every integer $k\geq3$, $u_1,\ldots,u_{k-1}\in V$,
$a\in X$, $b\in Y$ and $z'\in Z'$, the projected multiple
product
\begin{equation*}
\sum_{\lambda_1,\ldots,\lambda_{k-1}}
\Big\langle z',Y_Z(u_1,z_1)\pi^Z_{\lambda_1}
 Y_Z(u_2,z_2)\pi^Z_{\lambda_2}\cdots
 Y_Z(u_{k-1},z_{k-1})\pi^Z_{\lambda_{k-1}}
 \mathcal Y(a,z_k)b\Big\rangle
\end{equation*}
is absolutely convergent in
$|z_1|>\cdots>|z_k|>0$ and admits a maximal
multivalued analytic continuation $f(z_1,\ldots,z_k)$
on $M^k=\{z_i\ne0,\ z_i\ne z_j\ (i\ne j)\}$.
For every fixed $z_k\in\C^\times$ and every choice of
$M_{ij}\in\mathbb Z_{\geq0}$ ($1\leq i<j\leq k-1$)
with $x^{M_{ij}}Y_V(u_i,x)u_j\in V[[x]]$, every
component-isolated singularity of the function
\[
\left.\prod_{1\leq i<j\leq k-1}(z_i-z_j)^{M_{ij}}
 f(z_1,\ldots,z_{k-1},\zeta)\right|_{\zeta=z_k}
\]
in the variables $z_1,\ldots,z_{k-1}$, of the form
$(z_1,\ldots,z_{k-1})-\beta=\delta$ with
$\beta\in\C^{k-1}$ and
$\delta\in\{0,\infty\}^{k-1}$, is regular in the
sense of Definition~\ref{def:corners}.  The factors cancel
possible OPE poles among the additional $V$-insertions;
no factor involving $z_k$ is imposed by this clause.
\item The $L(-1)$-derivative identity
$\partial_x\mathcal Y(a,x)b=\mathcal Y(L_X(-1)a,x)b$.
\end{enumerate}
Write $\mathcal I_{\mathrm{DH}}\binom Z{X\;Y}$ for this
space in the stated ordinary untwisted setting.  The
convergence and preferred-branch requirements distinguish this
definition from Definition~\ref{def:ordinaryIO}.  The
$L(0)$-commutator formula follows from the Du--Huang duality
property (\cite[Remark following Definition~2.8]{DH}); for
ordinary modules it forces the logarithmic coefficients of
positive degree to vanish, as reflected in \eqref{eq:weight}.
\end{definition}

\begin{definition}\label{def:corners}
We use the component-isolated and regular-singularity
terminology of \cite[Definition~2.6 and its preceding paragraph]{DH}.
Let $F$ be a multivalued analytic function in $m$ independent
variables.  A choice $\delta=(\delta_1,\ldots,\delta_m)$ with
$\delta_i\in\{0,\infty\}$ is component-isolated if the
domain of $F$ contains a region
$\Delta^\times(\delta,r)$ specified, for some $r_i>0$, by
$0<|\eta_i|<r_i$ when $\delta_i=0$ and
$|\eta_i|>r_i$ when $\delta_i=\infty$.
The same terminology applies after an invertible affine
change of independent variables
$\eta=(z_1,\ldots,z_m)A-\beta$,
$A\in\operatorname{GL}(m,\C)$; component isolation is tested
in the new independent $\eta$-coordinates, not in the old ones.

At a component-isolated singularity, take all local branches
$F^b$ with cuts along the positive real rays.  The singularity
is \emph{regular} if there exist a common integer $B\geq0$
and, for each $i$, an exponent set $D_i$ independent of the
branch, of the form
$D_i=\bigcup_{q=1}^{N_i}(r_{iq}+\mathbb Z_{\geq0})$
when $\delta_i=0$ or
$D_i=\bigcup_{q=1}^{N_i}(r_{iq}-\mathbb Z_{\geq0})$
when $\delta_i=\infty$, with $r_{iq}\in\C$, such that for
each branch $b$ one has an absolutely convergent expansion
\begin{equation*}
F^b(\eta)=
\sum_{a_1\in D_1}\cdots\sum_{a_m\in D_m}
\sum_{j_1,\ldots,j_m=0}^{B}
\alpha^{(b)}_{\boldsymbol a,\boldsymbol j}
\prod_{i=1}^{m}\eta_i^{a_i}(\log\eta_i)^{j_i}
\end{equation*}
on some $\Delta^\times(\delta,r_b)$ (which may depend on
$b$).  In the present ordinary untwisted setting, rational
functions have zero logarithmic degree, but the definition is
not restricted to rational functions.  Definition~2.8 of
\cite{DH} tests this regularity only after fixing its final
insertion $z_k$ and in the indicated translation coordinates;
Definition~2.6 itself permits general invertible affine
changes.  A singularity along a diagonal divisor need not be
component-isolated in the original coordinates.
\end{definition}

We shall prove
$\mathcal I_{\mathrm{HL}}\binom Z{X\;Y}
 =\mathcal I_{\mathrm{DH}}\binom Z{X\;Y}$.
The analytic requirements in Definition~\ref{def:completeIO}
are those to be checked, rather than conditions inferred from
a formal Jacobi identity alone.

\subsection{Canonical \texorpdfstring{$P(z)$}{P(z)} products and normalizations}
Let $X\in\cC_g$, $Y\in\cC_h$ and $Z\in\cC_{gh}$, and fix
$z\in\C^\times$ and $L\in\C$ with $e^L=z$.  The order of the
inputs corresponds to the marked points $(\infty,z,0)$.  In this
subsection an intertwining operator between twisted modules means an
operator satisfying the full Du--Huang twisted intertwining-operator
conditions \cite[Definition~2.8]{DH}; in particular these include
the multiple-insertion analytic conditions.  The term
\emph{complete} will retain this meaning for such operators.  This
is not the Jacobi-identity notion of ordinary untwisted intertwining
operator in Definition~\ref{def:ordinaryIO}.

For a twisted intertwining operator $\mathcal Y$ of type
$\binom Z{X\;Y}$, its evaluation at $(z,L)$ is the linear map
\[
 \operatorname{ev}_L(\mathcal Y)=I_{\mathcal Y,L}:
 X\otimes Y\longrightarrow\overline Z,
 \qquad
 \pi^Z_\gamma I_{\mathcal Y,L}(x\otimes y)
 =\left.\pi^Z_\gamma\mathcal Y(x,t)y
 \right|_{t^\rho=e^{\rho L},\,\log t=L}.
\]
Thus, when a projected coefficient contains $t^\rho(\log t)^j$,
its value is $e^{\rho L}L^j$.  The evaluation is taken separately in each finite-dimensional weight space $Z_\gamma$, with the resulting value in $\overline Z$.  Here and below the
ordinary sector hypothesis ensures that each fixed output-weight
projection has a single exponent and no logarithmic terms, as shown
below.  We write
\[
 \mathcal I_{P(z),L}\binom Z{X\;Y}
 :=\operatorname{im}\!\left(
 \operatorname{ev}_L:\mathcal I_{\mathrm{tw}}\binom Z{X\;Y}
 \longrightarrow\Hom_\C(X\otimes Y,\overline Z)\right),
\]
where $\mathcal I_{\mathrm{tw}}$ denotes the full Du--Huang twisted
operator space for the indicated sectors.  A \emph{twisted
$P(z)$-intertwining map} is precisely an element of this image, not
an arbitrary linear map into $\overline Z$.  For $g=h=1$, the usual
Huang--Lepowsky $P(z)$-maps arise from ordinary intertwining
operators; their comparison with the above complete operators is established
in Section~3.  These definitions
follow \cite[Section~4]{DH} and, in the untwisted case, \cite{HL}.

A \emph{$P(z)$-product} of $(X,Y)$ in $\cC_{gh}$ is a pair
$(U,I)$ with $U\in\cC_{gh}$ and
$I\in\mathcal I_{P(z),L}\binom U{X\;Y}$; its morphisms satisfy
$\overline f\circ I=I'$ for $f:U\to U'$.  A
\emph{$P(z)$-tensor product} is an initial such product
$(X\bo_{P(z)}Y,I^z_{X,Y})$.  Equivalently, for every
$W\in\cC_{gh}$ the map
\[
 \Hom_{\cC_{gh}}(X\bo_{P(z)}Y,W)
 \xrightarrow{\ \sim\ }
 \mathcal I_{P(z),L}\binom W{X\;Y},\qquad
 f\longmapsto\overline f\circ I^z_{X,Y}
\]
is a bijection natural in $W$.  Naturality in $X,Y$ gives the
bifunctorial morphism $f\bo_{P(z)}q$, characterized by
$\overline{f\bo_{P(z)}q}\,I^z_{X,Y}
 =I^z_{X',Y'}(f\otimes q)$. The universal pair is defined by the
property in \cite[Section~4]{DH} and is unique up to the unique
isomorphism preserving its universal map. Xu's construction supplies
its existence in every twisted sector and ensures that the representing
operators satisfy the \emph{complete} analytic conditions
\cite[Theorem~4.3, equation~(4.3)]{Crossed}.

We choose the representative $X\bo Y=X\bo_{P(1)}Y$ and its
universal map $I^1_{X,Y}$ with $\log1=0$.  The inverse evaluation
correspondence in the ordinary twisted sectors determines a unique
complete formal operator
\[
 \mathcal Y_{X,Y}:X\otimes Y\longrightarrow
 (X\bo Y)\{t\}[\log t],\qquad
 \operatorname{ev}_{0}(\mathcal Y_{X,Y})=I^1_{X,Y};
\]
its evaluated values lie in $\overline{X\bo Y}$, not in the
algebraic tensor-product module in general.  Explicitly, for
$x\in X_\alpha$, $y\in Y_\beta$ and any target weight $\gamma$,
the $L(0)$-commutator identity
\[
 L_{X\bo Y}(0)\mathcal Y_{X,Y}(x,t)y
 -\mathcal Y_{X,Y}(x,t)L_Y(0)y
 =\mathcal Y_{X,Y}(L_X(0)x,t)y
   +t\partial_t\mathcal Y_{X,Y}(x,t)y
\]
shows that, writing the projected formal series as
$\sum_\rho t^\rho P_\rho(\log t)$, each coefficient satisfies
$P_\rho'(s)+(\rho-\gamma+\alpha+\beta)P_\rho(s)=0$.
A polynomial solution is zero unless $\rho=\gamma-\alpha-\beta$,
in which case it is constant.  Hence evaluation at one selected
logarithm is injective on these ordinary-module operator spaces;
this is the weightwise recovery used in the universal-property
construction of \cite[Sections~2 and~4]{Crossed}.The resulting weightwise formula is stated in \eqref{eq:weight} below.
The corresponding logarithmic degree is zero; for generalized
(nonsemisimple) $L(0)$-modules, polynomial logarithmic factors from
the nilpotent parts can occur.

For another $(z,L)$, the chosen universal $P(z)$-product is
identified with $X\bo Y$ only by specified representing
isomorphisms or parallel transport.  If $\gamma$ is a path from
$z_1$ to $z_2$ in $\C^\times$, continue the logarithm used for
the universal operator at $z_2$ backwards along $\gamma$ to a
value $L_1^\gamma$ over $z_1$.  The universal property determines
\[
 \mathsf T_\gamma:X\bo_{P(z_1)}Y\xrightarrow{\sim}
 X\bo_{P(z_2)}Y,\qquad
 \overline{\mathsf T_\gamma}\,I^{z_1}_{X,Y}
 =\operatorname{ev}_{L_1^\gamma}
   (\mathcal Y^{z_2}_{X,Y}),
\]
where $\mathcal Y^{z_2}_{X,Y}$ is the formal operator associated
with the chosen universal map at $z_2$.  Transport depends on the
endpoint-fixed homotopy class of $\gamma$; opposite and concatenated
paths give inverse and composite transports.  Winding may produce
nontrivial monodromy.  In the common representatives of
\cite[Proposition~9.6, equation~(9.28)]{Crossed}, transport
along the positive real interval with real logarithms is the
identity.  These transport conventions agree with
\cite[Section~4, parallel transport]{DH}.

For homogeneous $X\in\cC_g$ and $Y\in\cC_h$, the crossed
commutor $c_{X,Y}:X\bo Y\to T_gY\bo X$ is the Du--Huang
$P(1)$ crossed commutativity isomorphism followed by transport
from $-1$ to $1$ through the upper half-plane, as fixed in \cite[Sections~2 and~9]{Crossed}.
The target $T_gY\bo X$ has degree $(ghg^{-1})g=gh$.
With $y$ viewed as the same underlying vector of $T_gY$, its
normalization is the identity of complete formal operators
\begin{equation}\label{eq:comm}
\overline{c_{X,Y}}\mathcal Y_{X,Y}(x,z)y
=e^{zL(-1)}\mathcal Y_{T_gY,X}(y,e^{\pi i}z)x,
\qquad \log(e^{\pi i}z)=\log z+\pi i.
\end{equation}
Both sides lie in $\overline{T_gY\bo X}$, with $L(-1)$ on the
right acting on this output module; the exponential is defined
weightwise, since each fixed target-weight projection receives
only finitely many contributions from lower weights.  The notation $e^{\pi i}z$ specifies analytic continuation through a
positive half-turn, with the continued logarithm fixed by
$\log(e^{\pi i}z)=\log z+\pi i$. We use the branch-sensitive
identity \cite[Proposition~9.6, equation~(9.30)]{Crossed} as an input
from Xu.

For $k\in G$, the pullback functors preserve the underlying
$L(0)$-graded vector spaces and use
$Y_{T_kW}(v,t)=Y_W(k^{-1}v,t)$.  The same canonical formal
operator, with $X,Y,X\bo Y$ replaced by their pullbacks, is a
complete operator of type
$\binom{T_k(X\bo Y)}{T_kX\;T_kY}$: its defining analytic
identities are those for the original operator with each
$V$-insertion $u$ replaced by $k^{-1}u$.  By the twisted
universal property this fixes the action tensorator
$J^k_{X,Y}:T_kX\bo T_kY\to T_k(X\bo Y)$ via
\begin{equation}\label{eq:pull}
\overline{J^k_{X,Y}}\mathcal Y_{T_kX,T_kY}(x,z)y
=\mathcal Y_{X,Y}(x,z)y,
\end{equation}
where the right side is regarded as a completed vector in
$\overline{T_k(X\bo Y)}$ through the unchanged underlying
weight-space identification.   It holds first
at $P(1)$ by the representing property and then coefficientwise
as a formal-operator identity by the injectivity of evaluation
proved above; the same identity is valid with every consistently
continued logarithm.  This is the tensorator normalization
\cite[equation~(4.7) and Proposition~8.5]{Crossed}.
Uniqueness of the representing maps, together with the strict
underlying identities $T_kT_l=T_{kl}$, gives
$J^{kl}_{X,Y}=T_k(J^l_{X,Y})J^k_{T_lX,T_lY}$ as in
\eqref{eq:actioncoherence}; the tensorators remain part of the
monoidal action data. The paper \cite{Crossed} also supplies the associator,
crossed hexagons and contragredient pairings used below.

For the chosen $P(1)$ representative, the preceding $L(0)$
calculation gives, for $x\in X_\alpha$ and $y\in Y_\beta$,
\begin{equation}\label{eq:weight}
 \pi_\gamma\mathcal Y_{X,Y}(x,z)y
 =z^{\gamma-\alpha-\beta}\,v_\gamma(x,y),\qquad
 v_\gamma(x,y)\in(X\bo Y)_\gamma,
\end{equation}
where $v_\gamma(x,y)=\pi_\gamma I^1_{X,Y}(x\otimes y)$ is
independent of $z$ and $z^\rho=e^{\rho L}$ on the chosen branch.
Finite-dimensionality of $(X\bo Y)_\gamma$ comes from ordinary
grading restriction, not from convergence in the full completion.

\section{The neutral sector and mixed rationality}\label{sec:neutral}

\subsection{A locality-to-rationality lemma}
Throughout this subsection a matrix coefficient containing a completed
module vector is defined by homogeneous projections.  For homogeneous
inputs the conformal-weight shift fixes the intermediate weight at
any specified tuple of mode indices, so the coefficient involves no
infinite sum of weight components.  We distinguish a coefficientwise
formal distribution from its expansion in a prescribed ordering of
formal variables.

\begin{lemma}\label{lem:mixedrational}
Let $X,Y,Z$ be ordinary grading-restricted $V$-modules and
$\mathcal Y$ an ordinary Jacobi-identity intertwining operator of
 type $\binom{Z}{X\;Y}$. Fix homogeneous $a\in X_\alpha$,
$b\in Y_\beta$, $z'\in(Z')_\gamma$ and
$u_i\in V_{d_i}$, $1\le i\le m$, and put
$\rho=\gamma-\alpha-\beta-\sum_{i=1}^m d_i$.
The nested matrix coefficient
\begin{equation*}
 F(\boldsymbol{x};t)=
 \big\langle z',Y_Z(u_1,x_1)\cdots Y_Z(u_m,x_m)
                    \mathcal Y(a,t)b\big\rangle
\end{equation*}
is the ordered formal expansion, for
$|x_1|>\cdots>|x_m|>|t|>0$, of
\begin{equation}\label{eq:mixedrational}
 t^\rho R(x_1/t,\ldots,x_m/t),\qquad
 R\in\C\bigl[s_1^{\pm1},\ldots,s_m^{\pm1},
 (s_i-1)^{-1},(s_i-s_j)^{-1}\ (i<j)\bigr].
\end{equation}
After fixing $\log t$, the ordered expansion is absolutely and
normally convergent on compact subsets of this chamber.  Its
individual scalar terms agree, without infinite regrouping, with
the homogeneous intermediate-weight terms in the multiple-product
series of \cite[Definition~2.8]{DH}. For fixed $t\ne0$ the only
possible finite polar divisors in the additional insertion variables
are $x_i=0$, $x_i=t$ and $x_i=x_j$.
\end{lemma}
\begin{proof}
We first identify the coefficient spaces used below.  The
$L(0)$-bracket, obtained from the intertwining-operator Jacobi
identity, gives
\[
 \pi^Z_\delta\mathcal Y(a,t)b
 =t^{\delta-\alpha-\beta}v_\delta,
 \qquad v_\delta\in Z_\delta,
\]
with no logarithmic terms.  The grading restriction bounds the
possible exponents of $t$ from below.  For each fixed exponent,
$v_\delta$ is an actual vector in a finite-dimensional weight
space, and the matrix coefficient of the $m$ ordinary vertex
operators on it is a well-defined nested formal Laurent series.
Thus, up to its fixed fractional $t$-power, $F$ belongs to the
ordered iterated Laurent field
$\C((x_1))((x_2))\cdots((x_m))((t))$, with the
expansion order $|x_1|>\cdots>|x_m|>|t|$.  In this space the expansions of
$(x_i-x_j)^{-1}$ and $(x_i-t)^{-1}$ are fixed, and multiplication
by these differences is injective.  The same need not hold in an
unrestricted space of bilateral formal distributions.

Choose $A_i\ge0$ with $u_i(n)a=0$ for all integers $n\ge A_i$.
The commutator form of the Jacobi identity for $\mathcal Y$
contains only the finitely many terms
$\mathcal Y(u_i(n)a,t)$, $0\le n<A_i$, multiplied by
derivatives of the formal delta distribution at $x_i=t$.
Multiplication by $(x_i-t)^{A_i}$ kills each of these terms,
giving, as a coefficientwise identity in the two expansion orders,
\begin{equation}\label{eq:weakmixed}
 (x_i-t)^{A_i}Y_Z(u_i,x_i)\mathcal Y(a,t)
  =(x_i-t)^{A_i}\mathcal Y(a,t)Y_Y(u_i,x_i).
\end{equation}
The exponent depends only on $u_i,a$, not on $b$ or any other
mode indices.  Likewise, choose $A_{ij}\ge0$ so that
$u_i(n)u_j=0$ for $n\ge A_{ij}$ (increasing $A_{ij}$ if
necessary). The vertex-algebra Jacobi identity gives, on both
$Y$ and $Z$,
\begin{equation}\label{eq:weakvv}
 (x_i-x_j)^{A_{ij}}Y_M(u_i,x_i)Y_M(u_j,x_j)
 =(x_i-x_j)^{A_{ij}}Y_M(u_j,x_j)Y_M(u_i,x_i),
 \qquad M=Y,Z.
\end{equation}
These are formal-distribution identities after multiplication by
the indicated polynomial, not assertions that the unmultiplied
series in opposite order are equal.  There are only finitely many
pairs. Set
\begin{equation*}
 Q(\boldsymbol{x},t)=
 \prod_i(x_i-t)^{A_i}\prod_{i<j}(x_i-x_j)^{A_{ij}},
 \qquad D=\deg Q.
\end{equation*}
In particular $Q$ is homogeneous of total degree $D$ in
$\boldsymbol{x},t$.

For clarity about changing order, we perform the following
coefficientwise construction before dividing by $Q$.  A coefficient
of $x_1^{p_1}\cdots x_m^{p_m}$ in $F$ is obtained by using exactly
the modes $u_i(-p_i-1)$; its $\mathcal Y$-output weight is forced
to be $\gamma-\sum_i(d_i+p_i)$, and hence its $t$-power is
$\rho-\sum_i p_i$.  Thus the coefficient is a single scalar,
with no implicit infinite sum.  Multiplication by a polynomial
is a finite convolution of these scalar coefficients.  Applying
\eqref{eq:weakmixed}--\eqref{eq:weakvv} to each adjacent
interchange, with the intervening modes fixed, therefore proves
that the \emph{cleared} coefficient array $QF$ agrees with the
array obtained from any ordering of the $u_i$'s, with a chosen
$u_i$ also moved past $\mathcal Y$ onto $Y$.  

We now prove bounds independent of every other mode index.
Fix $i$ and, in the cleared array, place $u_i$ to the far right,
after $\mathcal Y$.  The final field acts on the fixed vector
$b$ as $u_i(n)b$, which vanishes for every $n\ge N_i$ for
some $N_i$ depending only on $u_i,b$.  Before multiplication by
$Q$ the resulting powers $p_i=-n-1$ satisfy $p_i\ge-N_i$.
As $Q$ has no negative powers of $x_i$, the same lower bound
holds for $QF$, regardless of all other indices.

For the upper bound, place $u_i$ at the far left, keeping
$\mathcal Y$ to its right.  If $w\in Z_\lambda$ is any weight
component supplied by the remaining operators, then
$\langle z',u_i(n)w\rangle\ne0$ requires
$\lambda=\gamma-d_i+n+1$.  Since $Z$ is lower bounded,
there is $B_Z\in\mathbb R$ with
$\Re\lambda\ge B_Z$ for all its weights. Consequently
$n\ge B_Z-\Re\gamma+d_i-1,$
where $n\in\mathbb Z$.  Hence the powers $p_i=-n-1$
are uniformly bounded above; multiplication by $Q$ increases
an $x_i$-exponent by at most $\deg_{x_i}Q$.  This bound too is
independent of the remaining mode indices.  Repeating this
argument for every $i$ confines the support of $QF$ to a finite
Cartesian product of intervals in $\mathbb Z^m$.  Both bounds
refer to the \emph{same cleared coefficient array}, so they may
be imposed simultaneously.  It follows that $QF$ has only
finitely many monomials in $\boldsymbol{x}$.

For the remaining $t$-powers, $u_i(n)$ shifts weight by
$d_i-n-1=d_i+p_i$, while $\mathcal Y(a,t)b$ has a
weight-$\delta$ component proportional to
$t^{\delta-\alpha-\beta}$.  Thus a monomial
$\boldsymbol{x}^{\boldsymbol p}$ of $F$ has $t$-degree
$\rho-\sum_i p_i$.  Multiplication by any monomial of $Q$
raises total degree by $D$.  The finite support just proved
therefore yields
$QF=t^{\rho+D}P(x_1/t,\ldots,x_m/t), P\in\C[s_1^{\pm1},\ldots,s_m^{\pm1}].$
This is initially an equality of coefficient arrays.  Its left
side is also the product of $Q$ with the original ordered
iterated Laurent series.  In that ordered series space $Q$ is
invertible, so division uniquely gives the ordered expansion
of $t^{\rho+D}P/Q$.  Since
$Q=t^D\prod_i(s_i-1)^{A_i}\prod_{i<j}(s_i-s_j)^{A_{ij}}$,
this is precisely \eqref{eq:mixedrational}.  This use of the
ordered Laurent space is essential: in a bilateral distribution
space, $QH=0$ need not imply $H=0$.

We verify analytic convergence without appealing to formal
rationality as an analytic theorem.  In the stated chamber
$|t/x_i|<1$ and $|x_j/x_i|<1$ for $i<j$.  Expand each
$(x_i-t)^{-A_i}$ as
$x_i^{-A_i}(1-t/x_i)^{-A_i}$ and each
$(x_i-x_j)^{-A_{ij}}$ as
$x_i^{-A_{ij}}\sum_{n\ge0} \binom{A_{ij}+n-1}{n}(x_j/x_i)^n (A_{ij}>0),$
with the factor understood as $1$ when the exponent is $0$.
All of these binomial series converge absolutely and locally
uniformly in their respective ratios.  There are finitely many
factors, and $P$ is a finite Laurent polynomial, so their
product converges absolutely and normally on compact subsets
of the chamber.  The factor $t^\rho=e^{\rho\log t}$ uses the
fixed branch.  The uniqueness of the ordered formal expansion
proved above identifies \emph{each} scalar monomial coefficient
with that of the original matrix coefficient, rather than only
identifying a formal rational expression.

Finally, insert the projections in the multiple series of
\cite[Definition~2.8]{DH}.  Write $\lambda_0=\gamma$, and
let $\lambda_i$ be the intermediate weight immediately to the
right of $Y_Z(u_i,x_i)$.  Because the $u_i$ have integral
weights, a nonzero term with mode $n_i\in\mathbb Z$ obeys
$\lambda_{i-1}=\lambda_i+d_i-n_i-1, n_i=\lambda_i+d_i-1-\lambda_{i-1}.$
Hence the $m$ intermediate weights determine the $m$ modes
uniquely, and the modes recursively determine the weights,
with $\lambda_m$ specifying the single $t$-power of the
intertwining operator.  Each projection is onto the correct
finite-dimensional weight component, so each tuple gives
exactly the same scalar as its corresponding monomial in the
ordered expansion.  This is a bijection of nonzero scalar
terms, not a regrouping of several terms into one.  Their
absolute convergence therefore follows from that of the
ordered rational expansion.  Arbitrary (not necessarily
homogeneous) inputs are finite linear combinations of
homogeneous vectors.  Finally, the rational function has no
polar divisors outside the factors displayed in
\eqref{eq:mixedrational}; its dependence on the fixed $t$ may
carry the selected power branch but introduces no monodromy in
the additional variables.
\end{proof}

\begin{lemma}\label{lem:rationalcorners}
Fix $t\in\C^\times$. Let $R$ be a rational function of the
remaining $m$ variables. At every component-isolated corner in
the sense of \cite[Definition~2.6]{DH}, including one specified by
an invertible affine change of variables and any mixture of zero
and infinity components, $R$ has an absolutely convergent
regular-singularity expansion with integral exponents and no
logarithmic terms. The exponent sets can be chosen independently
of the branch. In particular, this applies to the component-isolated
translation corners required by \cite[Definition~2.8]{DH}.
\end{lemma}
\begin{proof}
Write $\eta=zA-\beta$, where $A\in\mathrm{GL}_m(\C)$,
and put $\xi_i=\eta_i$ if $\eta_i\to0$ and
$\xi_i=\eta_i^{-1}$ if $\eta_i\to\infty$.
All changes are rational and invertible on the relevant coordinate
tori. Thus the transformed function, denoted $f(\xi)$, belongs to
$\C(\xi_1,\ldots,\xi_m)$. Clearing negative powers in a
numerator and denominator gives
$f(\xi)=P(\xi)/Q(\xi)$ with $P,Q\in\C[\xi_1,\ldots,\xi_m]$
and $Q\ne0$. In particular, $f$ defines a meromorphic germ
at $\xi=0$ in the sense of a quotient of holomorphic germs;
this alone does not yet yield a coordinatewise Laurent expansion.

By component isolation, there are $\epsilon_i>0$ such that $f$
is holomorphic on the entire product of punctured discs
$\Delta^\times_\epsilon =\{\xi:0<|\xi_i|<\epsilon_i\text{ for every }i\}.$
Its multivariable Laurent coefficients are therefore the iterated
Cauchy integrals
\[
 a_{\boldsymbol n}=\frac{1}{(2\pi i)^m}
 \int_{|w_1|=s_1}\!\cdots\!\int_{|w_m|=s_m}
 \frac{f(\boldsymbol w)}
 {w_1^{n_1+1}\cdots w_m^{n_m+1}}\,
 dw_m\cdots dw_1,\qquad 0<s_i<\epsilon_i,
\]
independent of the radii. The corresponding Laurent series
converges absolutely and locally uniformly on
$\Delta^\times_\epsilon$. We show that its negative powers
are bounded uniformly in each variable, rather than assuming
this from meromorphy.

Set $d_i=\deg_{\xi_i}Q$. For fixed values of the other
coordinates outside the proper algebraic set on which
$Q(\xi_i,\widehat\xi_i)$ vanishes identically as a polynomial
in $\xi_i$, the one-variable rational function
$\xi_i\mapsto f(\xi_i,\widehat\xi_i)$ has at $\xi_i=0$ a
pole of order at most $d_i$. Its one-variable Laurent
coefficients with exponent $n_i<-d_i$ consequently vanish.
For any circle $|\xi_i|=s_i$ inside the punctured disc,
these coefficients are Cauchy integrals and vary holomorphically
with the remaining coordinates on their punctured polydisc.
Their vanishing on the indicated dense set therefore implies
vanishing everywhere. Integrating the resulting zero function
in the remaining variables shows that
$a_{\boldsymbol n}=0$ whenever some $n_i<-d_i$.
The multivariable Laurent series hence has the form
\begin{equation}\label{eq:rationalcorner}
 f(\xi)=\xi_1^{-N_1}\cdots\xi_m^{-N_m}H(\xi),
 \qquad N_i=d_i\ge0,\quad
 H\in\C\{\xi_1,\ldots,\xi_m\}.
\end{equation}
Indeed the Laurent series of $H=\prod_i\xi_i^{N_i}f$
has only nonnegative powers, so its power series extends
holomorphically across every coordinate hyperplane. It is
absolutely convergent on a sufficiently small polydisc and
normally convergent on compact subsets. Multiplication by
$\xi^{-\boldsymbol N}$ preserves absolute and normal convergence
on the corresponding punctured polydisc.

This also verifies the polar-divisor step directly. The
factorization in \eqref{eq:rationalcorner} shows that the polar
divisor germ is supported on coordinate hyperplanes. Equivalently,
if a reduced denominator in the local ring
$\C\{\xi_1,\ldots,\xi_m\}$ had an irreducible non-coordinate
polar factor through the origin, its hypersurface would have
points with every $\xi_i\ne0$ arbitrarily near the origin.
At a generic such point the coprime numerator does not vanish
on that factor, contradicting holomorphy on
$\Delta^\times_\epsilon$. The irreducibility assertion follows
from the local analytic Nullstellensatz and the fact that an
irreducible hypersurface germ contained in a finite union of
coordinate hyperplanes is one of those hyperplanes; the
coordinate-factor conclusion also follows without these results
from the preceding Laurent-coefficient argument. In particular,
no additional codimension-two obstruction is left unaddressed.

Expand $H(\xi)=\sum_{\boldsymbol q\in\mathbb Z_{\ge0}^m}
 h_{\boldsymbol q}\xi^{\boldsymbol q}$. For a zero component,
the exponents in the original variable $\eta_i$ belong to
$-N_i+\mathbb Z_{\ge0}$; for an infinity component they belong to
$N_i-\mathbb Z_{\ge0}$, since $\xi_i=\eta_i^{-1}$.
These sets, chosen simultaneously for all components, have
exactly the forms required in \cite[Definition~2.6]{DH}.
The series is absolutely convergent in a sufficiently small
mixed zero/infinity region, with logarithmic degree $K=0$.
A rational function remains single-valued under the coordinate
changes, so the same exponent sets and coefficients apply to
all its branches. Finally, when $R$ arises from
Lemma~\ref{lem:mixedrational}, fixing the nonzero intertwining
coordinate $t$ makes $t^\rho$ a scalar on each selected
$\log t$ branch. Multiplication by this scalar, or by any
of the integral $V$--$V$ pole-clearing polynomials in
\cite[Definition~2.8]{DH}, preserves rationality in the
additional coordinates. The preceding argument then supplies
branch-independent integral exponent sets and zero logarithmic
degree for the resulting function. Thus every component-isolated
translation corner tested there is covered.
\end{proof}

\subsection{Equality of the two intertwining-operator spaces}
\begin{theorem}\label{thm:neutralanalytic}
For a strongly rational $V$ and all ordinary $V$-modules $X,Y,Z$,
every ordinary intertwining operator satisfies the complete
Du--Huang analytic requirements of \cite[Definition~2.8]{DH}.
In particular, for every $m\geq2$ and every choice of additional
states of $V$, the prescribed nested correlation functions converge
absolutely, extend to the full configuration space and have regular
component-isolated corners after multiplication by the prescribed
$V$--$V$ difference factors. Moreover,
$\mathcal I_{\mathrm{DH}}\binom Z{X Y} =\mathcal I_{\mathrm{HL}}\binom Z{X Y}.$
\end{theorem}
\begin{proof}
We first prove the inclusion from right to left. Let
$\mathcal Y\in\mathcal I_{\mathrm{HL}}\binom Z{X\;Y}$.
Its coefficientwise lower truncation and $L(-1)$-derivative
property are precisely the first and fourth conditions of
\cite[Definition~2.8]{DH}. For homogeneous
$a\in X_\alpha$, $b\in Y_\beta$, the $L(0)$-bracket
\cite[(5.4.12)--(5.4.14)]{FHL} implies
$\pi^Z_\gamma\mathcal Y(a,x)b
=x^{\gamma-\alpha-\beta}v_\gamma(a,b)$.
Thus its exponents have the requisite weight-difference support,
and it has logarithmic degree zero.

We verify the three parts of the one-$V$ duality condition,
including their prescribed branches. Take homogeneous $u\in V_d$
and $z'\in(Z')_\gamma$. The case $m=1$ of
Lemma~\ref{lem:mixedrational}, with $x_1=z_1$ and $t=z_2$,
gives a rational function $R$ such that the product series with
its intermediate $Z$-weight projection is the convergent expansion
of
$f(z_1,z_2)=e^{\rho\ell(z_2)}R(z_1/z_2), \rho=\gamma-\alpha-\beta-d,$
where $\ell(z_2)$ is the logarithm on the preferred branch of
\cite[Definition~2.8]{DH}. Here $R$ has possible poles only
at $z_1/z_2=0,1,\infty$. For fixed homogeneous data,
$f$ consequently has a multivalued analytic continuation on
$M^2$, arising only from continuation of $z_2^\rho$.
Its germs along paths in $M^2$ provide the maximal analytic
continuation of the initial preferred germ.

The other two matrix coefficients are not obtained simply by
rearranging a convergent infinite series. The ordinary
intertwining-operator Jacobi identity
\cite[Definition~5.4.1, (5.4.4)]{FHL} gives weak commutativity
between $Y_Z(u,z_1)$ and $\mathcal Y(a,z_2)$ and the
product--iterate identity
\cite[(5.4.17), (5.4.24)]{FHL}.
Clear a sufficiently high power of $z_1-z_2$ in these formal
identities. Uniqueness of the corresponding ordered expansions
identifies the reversed product and iterate, with the
$Y$- and $X$-weight projections in
Definition~\ref{def:completeIO}, as the expansions of this
\emph{same} $f$ in the respective variables
$(z_1,z_2)$ and $(z_1-z_2,z_2)$.
In the second chamber $|z_2|>|z_1|>0$ expand
$(z_1-z_2)^{-1}$ in $z_1/z_2$; in the third chamber
$|z_2|>|z_1-z_2|>0$ expand $z_1^{-1}$ in
$(z_1-z_2)/z_2$. The remaining possible denominators are
monomials. Each expansion is a finite Laurent sum multiplied
by finitely many geometric or binomial series, and hence is
absolutely and locally normally convergent. The $L(0)$ weight
relations identify its individual terms with the required
intermediate-weight projections, without an infinite regrouping.
On the three subregions specified in
Definition~\ref{def:completeIO}(ii), the formal substitutions
use the preferred determinations of $\ell(z_2)$ and the
associated logarithms of $z_1$ and $z_1-z_2$. Thus all three
sums represent the same preferred analytic germ; its continuation
along paths gives the other branches. This argument uses neither
an adjoint intertwining operator nor the additional assumptions
of \cite[Theorem~5.6.2]{FHL}.

For the multiple-$V$ condition, first take homogeneous
$u_i\in V$, $a,b$ and $z'$. Apply
Lemma~\ref{lem:mixedrational} with $m=k-1$ for each $k\geq3$.
Its proof identifies the $k-1$ intermediate output-weight
projections of Definition~\ref{def:completeIO}(iii) term by term
with the modes in the ordered expansion. In particular the
\emph{projected} multiple series converges absolutely in
$|z_1|>\cdots>|z_k|>0$. The analytic germ has the form
$z_k^{\rho} R(z_1/z_k,\ldots,z_{k-1}/z_k),$
with the $z_k$ power taken on the selected logarithmic branch.
Every possible pole of $R$ lies among $z_i=0$,
$z_i=z_k$, and $z_i=z_j$ $(i<j<k)$; after substituting the
nonzero variable $z_k$, these are precisely divisors excluded
from $M^k$. Thus the expression, with analytic continuation
of $z_k^\rho$ around $0$, gives a multivalued analytic
continuation throughout $M^k$. It is the maximal continuation of the germ defined by the convergent multiple series, with the branch of $z_k^\rho$ retained throughout.

Now fix $z_k\ne0$, including one logarithmic branch, as
required in \cite[Definition~2.8(3)]{DH}. For \emph{every}
choice of nonnegative integers $M_{ij}$ for $1\leq i<j\leq k-1$
with $x^{M_{ij}}Y_V(u_i,x)u_j\in V[[x]]$, multiplying the
continued function by
$\prod_{1\leq i<j\leq k-1}(z_i-z_j)^{M_{ij}}$
leaves a rational function of the additional $k-1$ positions,
up to the fixed nonzero scalar $z_k^\rho$. By
Lemma~\ref{lem:rationalcorners} every component-isolated
translation corner required in Definition~2.8 has an
absolutely convergent regular-singularity expansion, with
integer exponent sets common to all branches of the
additional variables and logarithmic degree zero. The fixed
scalar can change with the chosen branch of $z_k$ but does
not change these exponent sets. No common choice of leading
exponents as $z_k$ or the admissible $M_{ij}$ vary is asserted
or required. Finite linearity in the input vectors proves the
claims for arbitrary, not necessarily homogeneous, inputs.
This establishes
$\mathcal I_{\mathrm{HL}}\binom Z{X\;Y}
\subseteq\mathcal I_{\mathrm{DH}}\binom Z{X\;Y}$.

For the reverse inclusion, take
$\mathcal Y\in\mathcal I_{\mathrm{DH}}\binom Z{X\;Y}$.
The lower-truncation and derivative axioms already have the
ordinary form. The remark immediately following
\cite[Definition~2.8]{DH} derives the $L(0)$-commutator formula
from analytic duality (the conformal vector is fixed by all
twists). To make the absence of logarithms explicit, take
$a\in X_\alpha$, $b\in Y_\beta$ and project to
$Z_\gamma$. For the coefficient of $x^{-n-1}$ write
$P_{n,\gamma}(y)=\sum_{j=0}^K
\pi^Z_\gamma(a^{\mathcal Y}_{n,j}b)y^j$ with $y=\log x$.
The $L(0)$-bracket gives the polynomial differential equation
$P'_{n,\gamma}(y)= (\gamma-\alpha-\beta+n+1)P_{n,\gamma}(y).$
A polynomial solution is zero unless
$\gamma-\alpha-\beta+n+1=0$; in the latter case it is
constant. Therefore all coefficients of $(\log x)^j$ for
$j>0$ vanish, and each remaining term has the ordinary
weight-difference exponent. In particular
$\mathcal Y:X\otimes Y\to Z\{x\}$.

We now obtain the full ordinary Jacobi identity at the level
of matrix coefficients, independently of a commutator-only
criterion. Fix homogeneous $a,b,u,z'$ as above and the three
absolutely convergent sums from Definition~\ref{def:completeIO}(ii).
The $L(0)$-bracket implies simultaneous homogeneity:
$(z_1\partial_{z_1}+z_2\partial_{z_2})f=\rho f$,
where $\rho=\gamma-\alpha-\beta-d$.
For fixed $z_2\ne0$ and its logarithmic determination,
the three convergent expansions show that $f$ is meromorphic
in $z_1$ near $0$, $z_2$ and $\infty$, with finite pole
orders at each point. Indeed near $0$ the reversed product
has only finitely many negative powers of $z_1$ by
$u_n b=0$ for large $n$; near $z_2$ the iterate has only
finitely many negative powers of $z_1-z_2$ by
$u_n a=0$ for large $n$; and near infinity the product
has only finitely many positive powers of $z_1$, because
$z'$ has fixed weight and the output $Z$ is lower bounded.
The expansions have integral powers in the respective
local variable, so they have no local monodromy as $z_1$
encircles these punctures. Since these loops generate the
fundamental group of $\mathbb P^1\setminus\{0,z_2,\infty\}$,
the preferred germ is single-valued in $z_1$ for fixed
$z_2$ and extends meromorphically to the sphere. Hence
it is rational in $z_1$. Homogeneity then identifies it with
$z_2^\rho R(z_1/z_2)$ for a rational $R$ with possible
poles only at $0,1,\infty$.

The product, reversed product and iterate are its three
uniquely determined formal expansions in
$|z_1|>|z_2|$, $|z_2|>|z_1|$ and
$|z_2|>|z_1-z_2|$, respectively. Apply the elementary
formal delta-function identity for these three expansions
(after clearing powers of $z_1$, $z_2$ and $z_1-z_2$);
this is the rationality--commutativity--associativity to
Jacobi calculation of \cite[Proposition~3.4.1]{FHL}, applied
to the present mixed matrix coefficients. The resulting
identity is precisely \eqref{eq:IOJacobi} after contraction
with $z'$. The restricted-dual pairings separate the
homogeneous output coefficients, so the full formal
intertwining-operator Jacobi identity holds. In particular,
this supplies the formal justification of the assertion in
\cite[Remark~2.11]{DH}: that remark also states that a
Definition~2.8 operator is an ordinary intertwining operator
for $V^{\langle g_1,g_2\rangle}$, including Jacobi rather than
merely the commutator formula. Here $g_1=g_2=g_3=1$, so
that fixed-point algebra is precisely $V$. We have therefore
established the Jacobi identity, lower truncation, derivative
property and absence of logarithms, giving
$\mathcal Y\in\mathcal I_{\mathrm{HL}}\binom Z{X\;Y}$.

Both inclusions preserve the actual mode coefficients: neither
constructs a new intertwining operator. Composition with module
maps in the two inputs and the output commutes with every
projection and formal identity used above, so the equality is
natural in $X,Y,Z$ (contravariantly in the inputs, covariantly
in the output). It also preserves weightwise evaluation at any
$P(z)$ with the same prescribed logarithm; a change of path or
branch acts on the same formal coefficients on both sides.
\end{proof}

In the next proposition, write $\bo_D$, $c^D$ and $a^D$
for the tensor product, crossed commutor restricted to degree $1$,
and associator constructed in \cite{Crossed}; write
$\bo_H$, $c^H$ and $a^H$ for their Huang--Lepowsky
counterparts.  The notation distinguishes structures on the same
underlying ordinary module category, not different classes of
modules.

\begin{proposition}\label{prop:neutralcomparison}
Let $V$ be strongly rational and $G\leq\Aut(V)$ finite, and let
$\cC_G$ be the category constructed by Xu in
\cite[Theorem~1.1]{Crossed}. Its binary $P(z)$-representing property
and canonical operators are given in
\cite[Theorem~4.3, equations~(4.3)--(4.4)]{Crossed}; its pullback
tensorators and action coherence are given in
\cite[equation~(4.7), Proposition~8.5]{Crossed}. The product--iterate
and unit normalizations are established in
\cite[Theorem~8.3, equations~(8.10)--(8.11), Proposition~8.5]{Crossed},
and the positive-half-turn commutor and parallel-transport identities
in \cite[Proposition~9.6, equations~(9.27)--(9.30)]{Crossed}. Then
there is a canonical braided monoidal equivalence
\[
 F:(\cC_1,\bo_D,c^D,a^D)\xrightarrow{\sim}
 (\Rep(V),\bo_H,c^H,a^H)
\]
whose underlying functor is the identity on modules and morphisms.
For every $z\in\C^*$ and every specified logarithm $L$ with $e^L=z$, the comparison also preserves the corresponding $P(z)$-representing objects and parallel transports.
\end{proposition}
\begin{proof}
By the regularity and semisimplicity established in Section~2, the neutral sector $\cC_1$ and $\Rep(V)$ have the same ordinary $V$-modules and the same module homomorphisms. The tensor products, however, are initially different structures: $\bo_D$ is the binary product imported from \cite{Crossed}, whereas $\bo_H$ is the Huang--Lepowsky product. The $P(z)$ universal property on the Huang--Lepowsky side is the one in \cite[Definition~3.1 and Proposition~3.4]{HL}; the standard associator and braiding belong to the vertex-tensor-category construction of \cite[Theorem~4.4]{HL}. We compare the structures through their representing maps, not merely through their common underlying objects.

Fix $z\ne0$ and a logarithm $L$ of $z$. For $E=D,H$, let $I^{E,z}_{X,Y}:X\otimes Y\to\overline{X\bo_{E,P(z)}Y}$ denote the corresponding universal $P(z)$-intertwining map. Here $\bo_{D,P(z)}$ is the $P(z)$-product from \cite{Crossed} and $\bo_{H,P(z)}$ is the Huang--Lepowsky product, with the same chosen evaluation logarithm. For each ordinary target $W$, Theorem~\ref{thm:neutralanalytic} identifies the complete Du--Huang intertwining operators with the ordinary Huang--Lepowsky intertwining operators coefficient by coefficient. Evaluating both spaces at the same $L$ therefore identifies their $P(z)$-map spaces as the same subspace of $\Hom_{\C}(X\otimes Y,\overline W)$. The binary representing property in \cite{Crossed} and the Huang--Lepowsky universal property thus represent the same functor of $W$.

Consequently there is a unique $V$-module isomorphism
\begin{equation}\label{eq:neutralphi}
 \Phi^z_{X,Y}:X\bo_{D,P(z)}Y\xrightarrow{\sim}X\bo_{H,P(z)}Y,
 \qquad \overline{\Phi^z_{X,Y}}I^{D,z}_{X,Y}=I^{H,z}_{X,Y}.
\end{equation}
For completeness, the inverse is the unique map $\Psi^z_{X,Y}:X\bo_{H,P(z)}Y\to X\bo_{D,P(z)}Y$ with $\overline{\Psi^z_{X,Y}}I^{H,z}_{X,Y}=I^{D,z}_{X,Y}$. The two universal properties give
\[
 \overline{\Psi^z_{X,Y}\Phi^z_{X,Y}}I^{D,z}_{X,Y}=I^{D,z}_{X,Y},\qquad
 \overline{\Phi^z_{X,Y}\Psi^z_{X,Y}}I^{H,z}_{X,Y}=I^{H,z}_{X,Y};
\]
injectivity of the representing bijections forces both composites to be identities. This also proves uniqueness of $\Phi^z_{X,Y}$ among maps preserving the universal maps.

The comparison is natural in the inputs. If $f:X\to X'$ and $q:Y\to Y'$ are module maps, then
$\Phi^z_{X',Y'}\circ(f\bo_{D,P(z)}q) =(f\bo_{H,P(z)}q)\circ\Phi^z_{X,Y},$
because both sides, after completion and precomposition with $I^{D,z}_{X,Y}$, give $I^{H,z}_{X',Y'}\circ(f\otimes q)$. Naturality in the target object is the naturality in $W$ of the two representing bijections. In the ordinary sector, evaluation at a fixed logarithm is injective weight by weight, by the $L(0)$-bracket and the formula in \eqref{eq:weight}. Hence the evaluated identity in \eqref{eq:neutralphi} lifts to the identity of formal complete operators
$\overline{\Phi^z_{X,Y}}\,\mathcal Y^D_{X,Y}(x,t)y =\mathcal Y^H_{X,Y}(x,t)y,$
with the same formal coefficients and the same continued logarithm. Equivalently, $\overline{(\Phi^z_{X,Y})^{-1}}\mathcal Y^H_{X,Y}=\mathcal Y^D_{X,Y}$.

This comparison also commutes with parallel transport, not just with the existence of an object isomorphism. Let $\gamma$ be a path from $z_1$ to $z_2$ in $\C^*$, and let $L_1^\gamma$ be the continuation back to $z_1$ of the selected logarithm $L_2$ at $z_2$. Denote the transports in the two constructions by $\mathsf T^D_\gamma$ and $\mathsf T^H_\gamma$. Their defining identities and \eqref{eq:neutralphi} give
\[
 \overline{\Phi^{z_2}_{X,Y}\mathsf T^D_\gamma}I^{D,z_1}_{X,Y}
 =\operatorname{ev}_{L_1^\gamma}(\mathcal Y^H_{X,Y})
 =\overline{\mathsf T^H_\gamma\Phi^{z_1}_{X,Y}}I^{D,z_1}_{X,Y}.
\]
The representing property in \cite{Crossed} then yields the typed equality
$\Phi^{z_2}_{X,Y}\circ\mathsf T^D_\gamma =\mathsf T^H_\gamma\circ\Phi^{z_1}_{X,Y}.$
The identity is pathwise: if $\gamma$ winds around the origin, both
sides retain the same branch change and monodromy.

Take $z=1$ and $L=0$. The functor $F$ is the identity on the common module category, its unit map is $F_0=\id_V$, and its tensorator, in the convention $F_2:F(X)\bo_HF(Y)\to F(X\bo_DY)$, is
$F_{2,X,Y}=(\Phi^1_{X,Y})^{-1}:X\bo_HY\longrightarrow X\bo_DY.$
The strong-monoidal-functor convention fixes the tensorator as
$F_{2,X,Y}=(\Phi^1_{X,Y})^{-1}$.

We next verify associativity. Put
$U_H=(X\bo_HY)\bo_HZ, U_D=(X\bo_DY)\bo_DZ, V_D=X\bo_D(Y\bo_DZ).$
The two sides of the required coherence equation are the following maps with the same source and target:
\[
\begin{aligned}
 U_H&\xrightarrow{\,F_{2,X,Y}\bo_H\id_Z\,}
 (X\bo_DY)\bo_HZ
 \xrightarrow{\,F_{2,X\bo_DY,Z}\,}U_D
 \xrightarrow{\,F(a^D_{X,Y,Z})\,}V_D,\\
 U_H&\xrightarrow{\,a^H_{X,Y,Z}\,}
 X\bo_H(Y\bo_HZ)
 \xrightarrow{\,\id_X\bo_HF_{2,Y,Z}\,}
 X\bo_H(Y\bo_DZ)
 \xrightarrow{\,F_{2,X,Y\bo_DZ}\,}V_D.
\end{aligned}
\]
In particular, each occurrence of $F_2$ is applied to a pair of objects in the common ordinary module category, so the intermediate objects above are well typed.

For $E=D,H$, the defining normalization of the corresponding associator is the product--iterate identity
\begin{equation*}
\begin{split}
 \overline{a^E_{X,Y,Z}}\,
 \mathcal Y^E_{X\bo_EY,Z}
 \bigl(\mathcal Y^E_{X,Y}(x,z_1-z_2)y,z_2\bigr)z
 ={}&\mathcal Y^E_{X,Y\bo_EZ}(x,z_1)
       \mathcal Y^E_{Y,Z}(y,z_2)z,
\end{split}
\end{equation*}
on the common positive overlap $z_1>z_2>z_1-z_2>0$, with all logarithms real. The left side denotes the convergent sum of intermediate-weight
projections defining the iterate, with values in the completion of
$(X\bo_EY)\bo_EZ$ before applying $a^E$. The right side is the product, taking values in the completion of $X\bo_E(Y\bo_EZ)$. These are precisely the associator normalization imported from \cite[Theorem~8.3, equation~(8.11)]{Crossed} and the Huang--Lepowsky associator normalization. All three modules and the intermediate tensor products are ordinary, so Theorem~\ref{thm:neutralanalytic} applies to each binary operator appearing here.

Apply the two composites displayed above to the canonical iterated coefficients of the $H$-product $U_H$. The identity $\overline{F_{2,A,B}}\mathcal Y^H_{A,B}=\mathcal Y^D_{A,B}$, naturality of the binary products, and the defining product--iterate identity show that the first composite sends each such coefficient to the corresponding $D$-product coefficient. The second composite does the same, by first using the $H$ product--iterate identity and then applying the same binary comparison to its two factors. Thus, after pairing with any vector in the restricted dual of the common target module, the two composites give the same scalar analytic function on the positive overlap. The restricted dual separates the weight completion, so their values on every canonical iterated coefficient agree. The coefficients of a universal binary map generate its representing module: otherwise the quotient by the submodule they generate would give a nonzero morphism annihilating the universal map, contrary to injectivity of the representing bijection. Applying this observation successively to the two binary products shows that the canonical iterated coefficients generate $U_H$. The two module maps are therefore equal. This proves
\begin{align*}
 &F(a^D_{X,Y,Z})F_{2,X\bo_DY,Z}
                 (F_{2,X,Y}\bo_H\id_{FZ})\\
 &\qquad=F_{2,X,Y\bo_DZ}(\id_{FX}\bo_HF_{2,Y,Z})
                  a^H_{FX,FY,FZ}.\notag
\end{align*}
No associator coherence is used to prove this identity; only the two independently specified product--iterate normalizations and binary universality enter.

For braiding, both sides of
\begin{equation}\label{eq:neutralbraid}
 F(c^D_{X,Y})F_{2,X,Y}=F_{2,Y,X}c^H_{FX,FY}
\end{equation}
are maps $X\bo_HY\to Y\bo_DX$. The crossed commutor from \cite{Crossed} on the neutral sector is the positive-half-turn map of \eqref{eq:comm}; the Huang--Lepowsky braiding is defined by the same $\Omega_+$ skew transformation and the same upper-half-plane transport. In both formulas the continued logarithm is $\log(e^{\pi i}z)=\log z+\pi i$, and $e^{zL(-1)}$ acts on the output module $Y\bo X$. Since $T_1=\id$ and the two canonical formal operators are related by $F_2$, no nontrivial pullback comparison or action tensorator occurs in this neutral-sector calculation. Precomposing the two composites with $I^H_{X,Y}$ gives the same evaluated $\Omega_+$ operator by \eqref{eq:comm} and the Huang--Lepowsky braiding formula. The $P(1)$ representing property, which is injective on morphisms, gives \eqref{eq:neutralbraid}.

Finally, both unit constraints are normalized by the same vacuum intertwiners: the left unit uses $Y_X$ and the right unit uses $\Gamma_X=\Omega_-Y_X$, with the branch fixed in Section~2. The unit object is $V$ on both sides, so set $F_0=\id_V:V_H\to F(V_D)$. Evaluating on the universal maps gives the full unit equations
\[
 F(l^D_X)\circ F_{2,V,X}\circ(F_0\bo_H\id_{FX})=l^H_{FX},\qquad
 F(r^D_X)\circ F_{2,X,V}\circ(\id_{FX}\bo_HF_0)=r^H_{FX},
\]
which are the left and right unit diagrams in the chosen strong-monoidal convention. Thus $F$ is a strong monoidal functor, and \eqref{eq:neutralbraid} makes it braided. Since its underlying functor is the identity equivalence and every $F_{2,X,Y}$ and $F_0$ is an isomorphism, it is a braided monoidal equivalence. Its tensorator is unique among those carrying the specified universal maps to one another, and the displayed transport identity proves the asserted $P(z)$-branch and parallel-transport compatibility.
\end{proof}

\begin{remark}\label{rem:scopeofneutral}
Theorem~\ref{thm:neutralanalytic} concerns ordinary intertwining
operators. Its proof uses weak locality with the distinguished
ordinary intertwining operator and integer-powered vacuum fields.
The rationality argument applies to the variables of the additional
$V$ insertions with the intertwining-operator coordinate fixed.
\end{remark}

\section{The conformal crossed twist}
\begin{lemma}\label{lem:theta}
For $g\in G$ and $X\in\cC_g$, the conformal-weight operator induces
an isomorphism of $g$-twisted modules
$\theta_X=e^{2\pi iL_X(0)}:X\xrightarrow{\,\sim\,}T_gX.$
It is natural in $X$, satisfies
$\theta_{T_kX}=T_k(\theta_X)$ for every $k\in G$ under the
strict underlying pullback convention, and obeys
$\theta_V=\id_V=j_1$.
\end{lemma}
\begin{proof}
By the regularity result recalled in Section~2, $X$ is an ordinary
grading-restricted $g$-twisted module. Thus
$X=\bigoplus_{\lambda\in\C}X_\lambda$ and
$\theta_X|_{X_\lambda}=e^{2\pi i\lambda}\id$ defines an
invertible linear operator on the algebraic direct sum, with
inverse $e^{-2\pi iL_X(0)}$. Both operators preserve each weight
space. Since automorphisms of $V$ fix $\omega$, the pullback $T_gX$
has the same $L(n)$, in particular $L(0)$ and $L(-1)$, on the
same graded vector space.

Let $N=\operatorname{ord}(g)$ and choose $v\in V_d$ in the
$g$-eigenspace of eigenvalue $e^{2\pi ir/N}$, where $d\in\mathbb{Z}$
and $0\le r<N$. The finite-order convention of Section~2,
\cite[Definition~3.1]{DLM}, and the equivariance convention in
\cite[Section~2, definition of generalized twisted modules]{DH}
give
$Y_X(v,z)=\sum_{n\in r/N+\mathbb{Z}}v_nz^{-n-1}$.
Indeed, for an ordinary finite-order twisted field the latter
analytic equivariance reads
$Y_X(gv,e^{2\pi i}z)=Y_X(v,z)$ on continued branches,
so the monodromy of $z^{-n-1}$ is
$e^{-2\pi i(n+1)}=e^{-2\pi ir/N}$.
Since $Y_X(\omega,z)$ has integral modes, the twisted
$L(0)$-bracket gives
$[L_X(0),Y_X(v,z)]=(z\partial_z+d)Y_X(v,z), [L_X(0),v_n]=(d-n-1)v_n.$
For $x\in X_\lambda$, the vector $v_nx$ lies in
$X_{\lambda+d-n-1}$. Consequently, comparing the actions on
this weight space yields the operator identity
\[
 e^{2\pi iL_X(0)}v_ne^{-2\pi iL_X(0)}
  =e^{2\pi i(d-n-1)}v_n
  =e^{-2\pi ir/N}v_n=(g^{-1}v)_n.
\]
The same equation holds for arbitrary $v\in V$ by its finite
decomposition into homogeneous conformal weights and $g$-eigenspaces.
It follows coefficientwise in $\operatorname{Hom}(X,X)((z^{1/N}))$
(with the usual lower truncation on each vector) that
$\theta_XY_X(v,z)=Y_X(g^{-1}v,z)\theta_X =Y_{T_gX}(v,z)\theta_X.$
Here $T_gX\in\cC_{ggg^{-1}}=\cC_g$ and the last expression
uses the defining pullback field. Thus $\theta_X:X\to T_gX$
is a degree-preserving twisted-module morphism, rather than an
endomorphism of the original labeled module. Multiplying this
identity by $\theta_X^{-1}$ on the appropriate sides gives
$\theta_X^{-1}Y_{T_gX}(v,z)=Y_X(v,z)\theta_X^{-1}$;
therefore $\theta_X^{-1}:T_gX\to X$ is also a twisted-module
morphism. Both maps commute with $L(0)$ and $L(-1)$ and extend
weightwise to the corresponding completions.

For any $g$-twisted module morphism $f:X\to Y$ one has
$fL_X(0)=L_Y(0)f$ by applying $f$ to the conformal-vector
field. Hence, as maps $X\to T_gY$,
$T_g(f)\circ\theta_X=\theta_Y\circ f.$
On morphisms $T_g(f)$ is the same underlying linear map, but
its source and target are the pulled-back labeled modules.
For $k\in G$, put $h=kgk^{-1}$. The two maps
\[
 \theta_{T_kX}:T_kX\longrightarrow T_hT_kX,
 \qquad T_k(\theta_X):T_kX\longrightarrow T_kT_gX
\]
have the same target since $T_hT_k=T_{hk}=T_{kg}=T_kT_g$
as strict underlying pullback functors. Because
$Y_{T_kX}(\omega,z)=Y_X(\omega,z)$, they act on each
weight space by the same scalar $e^{2\pi i\lambda}$, proving
$\theta_{T_kX}=T_k(\theta_X)$. No monoidal tensorator or
action compositor is involved.

The assertion is sectorwise: for an object
$\bigoplus_gX_g$ of the externally graded sum, the collection
assembles into the morphism
$\bigoplus_g\theta_{X_g}:\bigoplus_gX_g\to
\bigoplus_gT_gX_g$, not a morphism to $T_g$ of the entire
object for a single $g$. Finally $V\in\cC_1$ is the tensor unit,
$T_1=\id$, and $V=\bigoplus_{d\in\mathbb{Z}}V_d$. Hence
$e^{2\pi iL_V(0)}=\id_V=j_1$, the unit normalization in
\eqref{eq:actioncoherence}.
\end{proof}

\begin{lemma}\label{lem:double}
Let $X\in\cC_g$, $Y\in\cC_h$, and put $q=ghg^{-1}$. The composite
\[
 c_{T_gY,X}\circ c_{X,Y}:
 X\bo Y\xrightarrow{\ \sim\ }T_qX\bo T_gY
\]
is an isomorphism. For $x\in X$ and $y\in Y$ it satisfies the
identity of completed canonical intertwining operators
\begin{equation}\label{eq:turn}
 \overline{c_{T_gY,X}\circ c_{X,Y}}\,
 \mathcal Y_{X,Y}(x,z)y
 =\mathcal Y_{T_qX,T_gY}(x,e^{2\pi i}z)y,
\end{equation}
where the right-hand side is evaluated on the branch
$\log(e^{2\pi i}z)=\log z+2\pi i$ obtained by two consecutive
positive half-turns.
\end{lemma}
\begin{proof}
Put $U=X\bo Y$, $B=T_gY\bo X$, and $C=T_qX\bo T_gY$.
Since $T_gY\in\cC_q$ and $T_qX\in\cC_{qgq^{-1}}$, both $B$ and $C$
are in $\cC_{qg}=\cC_{gh}$. In particular the two arrows
\[
 c_1:=c_{X,Y}:U\longrightarrow B,
 \qquad c_2:=c_{T_gY,X}:B\longrightarrow C
\]
are composable isomorphisms in the displayed order of $g,h$.
Write $\mathcal Z=\mathcal Y_{T_gY,X}$ and
$\mathcal W=\mathcal Y_{T_qX,T_gY}$.  The formal normalization
\eqref{eq:comm}, applied first to $(X,Y)$ and then to $(T_gY,X)$,
gives the following two equalities, respectively in the completed
outputs $\overline B$ and $\overline C$:
\[
 \overline{c_1}\mathcal Y_{X,Y}(x,z)y
   =e^{zL_B(-1)}\mathcal Z(y,w)x,
 \qquad w=e^{\pi i}z,
\]
$\overline{c_2}\mathcal Z(y,w)x =e^{wL_C(-1)}\mathcal W(x,e^{\pi i}w)y.$
Here $y$ in the first equality is the underlying vector of $T_gY$,
$x$ in the second is the underlying vector of $T_qX$, and the
logarithms are continued successively as
$\log w=\log z+\pi i$ and
$\log(e^{\pi i}w)=\log z+2\pi i$.  The second equality uses
$\deg(T_gY)=q$, not $h$.  It applies at the inherited logarithm
because \eqref{eq:comm} is an identity of complete \emph{formal}
operators, with the indicated logarithmic continuation, rather than
only an identity at a single $P(1)$ evaluation.

As a morphism of twisted $V$-modules, $c_2$ intertwines the action of
the conformal vector, so
$c_2L_B(-1)=L_C(-1)c_2$.  This relation also holds after
weightwise completion. Indeed, $L(-1)$ raises conformal weight by one,
and for an ordinary lower-bounded output every fixed-weight projection
of $e^{zL(-1)}H(z)$ has only finitely many contributions from lower
weights.  Thus, coefficientwise in the completed output,
$\overline{c_2}\,e^{zL_B(-1)}H(z) =e^{zL_C(-1)}\overline{c_2}H(z).$
Applying this to the first equality and then substituting the second,
we obtain an equality in $\overline C$:
\[
\begin{aligned}
 \overline{c_2\circ c_1}\,\mathcal Y_{X,Y}(x,z)y
 &=e^{zL_C(-1)}\overline{c_2}\,\mathcal Z(y,w)x\\
 &=e^{zL_C(-1)}e^{wL_C(-1)}
     \mathcal W(x,e^{\pi i}w)y\\
 &=\mathcal W(x,e^{2\pi i}z)y.
\end{aligned}
\]
For the last equality, $w=-z$ as a complex coordinate; both
translation operators now act on $C$.  Their product is the identity:
for each fixed output weight, its coefficient of $L_C(-1)^n$ is
$\sum_{j=0}^{n}z^jw^{n-j}/(j!(n-j)!)=(z+w)^n/n!$.
This is a finite sum in every fixed-weight projection, so no
rearrangement of an unconditionally convergent infinite series is
being assumed.

The position $e^{2\pi i}z$ is numerically $z$, but the continued
logarithm is not reset.  For homogeneous $x\in X_\alpha$,
$y\in Y_\beta$, equation \eqref{eq:weight} applied to $\mathcal W$
gives on the $\gamma$-weight component the factor
$e^{2\pi i(\gamma-\alpha-\beta)}$ relative to the initial
branch; the pullbacks preserve these weights.  No logarithmic
polynomials occur for the ordinary objects of $\cC_G$.  Finite
homogeneous decompositions extend the assertion to all $x,y$.
The two commutors already include their chosen upper-half-plane
transports.  The argument follows from \eqref{eq:comm} and the indicated module
structures.
\end{proof}

\begin{theorem}\label{thm:balance}
The conformal twist satisfies the crossed balancing relation
\begin{equation}\label{eq:bal}
\theta_{X\bo Y}=J^{gh}_{X,Y}
 \bigl(T_{ghg^{-1}}\theta_X\bo T_g\theta_Y\bigr)
 c_{T_gY,X}\,c_{X,Y},\quad X\in\cC_g,\ Y\in\cC_h.
\end{equation}
\end{theorem}
\begin{proof}
Put $q=ghg^{-1}$, $U=X\bo Y$, and $k=gh=qg$.
The tensor-product objects in the displayed composite lie in
$\cC_k$.  The constituent arrows have the following types:
\[
\begin{gathered}
 c_{X,Y}:U\longrightarrow T_gY\bo X,\qquad
 c_{T_gY,X}:T_gY\bo X\longrightarrow T_qX\bo T_gY,\\
 T_q(\theta_X):T_qX\longrightarrow T_qT_gX=T_kX,\qquad
 T_g(\theta_Y):T_gY\longrightarrow T_gT_hY=T_kY,\\
 J^k_{X,Y}:T_kX\bo T_kY\longrightarrow T_kU.
\end{gathered}
\]
The successive pullbacks use the underlying action functors, while
the tensorator $J^k$ remains part of the composite.
The composite on the right of \eqref{eq:bal}, denoted by $B_{X,Y}$,
is thus a morphism $U\to T_kU$.  Since $U\in\cC_k$,
Lemma~\ref{lem:theta} gives a morphism
$\theta_U:U\to T_kU$ of the same type.

We compare the two morphisms by precomposing their completed
extensions with the universal formal intertwining operator
$\mathcal Y_{X,Y}$; all operator identities below are interpreted
coefficientwise in $\overline{T_kU}$.  The positive full-turn
identity \eqref{eq:turn} first gives
\[
 \overline{c_{T_gY,X}c_{X,Y}}\,
 \mathcal Y_{X,Y}(x,z)y
 =\mathcal Y_{T_qX,T_gY}(x,e^{2\pi i}z)y,
\]
with logarithm $L+2\pi i$ if the initial logarithm of $z$
is $L$.  This equation is in
$\overline{T_qX\bo T_gY}$.  For twisted-module morphisms
$f:A\to A'$ and $p:B\to B'$, the bifunctoriality of the chosen
$P(1)$-tensor products gives
$\overline{f\bo p}\,I^1_{A,B}(a\otimes b) =I^1_{A',B'}(fa\otimes pb).$
The same identity holds for the associated complete formal
operators at every consistently continued logarithm: both sides
are obtained from the corresponding natural complete
intertwining operators and agree coefficientwise by the
$L(0)$-grading and the evaluation--coefficient correspondence
of Section~2.  In particular, with
$f=T_q(\theta_X)$ and $p=T_g(\theta_Y)$, it gives
\[
\begin{aligned}
 &\overline{T_q(\theta_X)\bo T_g(\theta_Y)}\,
 \mathcal Y_{T_qX,T_gY}(x,e^{2\pi i}z)y\\
 &\qquad=\mathcal Y_{T_kX,T_kY}
       \bigl(T_q(\theta_X)x,e^{2\pi i}z\bigr)
       T_g(\theta_Y)y.
\end{aligned}
\]
The map $f\bo p$ is selected by this universal property.

First take $x\in X_\alpha$, $y\in Y_\beta$.
All pullbacks preserve $L(0)$, so their underlying graded
vectors satisfy
$T_q(\theta_X)x=e^{2\pi i\alpha}x, T_g(\theta_Y)y=e^{2\pi i\beta}y.$
The two right-hand inputs in the preceding display are now
vectors of $T_kX$ and $T_kY$, respectively.  Applying the
completed tensorator and its defining formal normalization
\eqref{eq:pull}, with $k=gh$, yields the identity in
$\overline{T_kU}$
\[
\begin{aligned}
 \overline{B_{X,Y}}\,\mathcal Y_{X,Y}(x,z)y
 &=\overline{J^k_{X,Y}}\,
 \mathcal Y_{T_kX,T_kY}
 \bigl(T_q(\theta_X)x,e^{2\pi i}z\bigr)
 T_g(\theta_Y)y\\
 &=e^{2\pi i(\alpha+\beta)}\,
 \overline{J^k_{X,Y}}\,
 \mathcal Y_{T_kX,T_kY}(x,e^{2\pi i}z)y\\
 &=e^{2\pi i(\alpha+\beta)}\,
 \mathcal Y_{X,Y}(x,e^{2\pi i}z)y.
\end{aligned}
\]
In the last term the coefficients are viewed in $T_kU$ through
its fixed underlying weight-space identification.  This use of
\eqref{eq:pull} is valid on the continued branch because the
identity is a normalization of complete \emph{formal} operators,
not merely an equality of numerical $P(1)$ evaluations.  The
extension $\overline{J^k_{X,Y}}$ preserves each conformal
weight since $J^k_{X,Y}$ is a twisted-module morphism.

For a weight $\gamma$ of $U$, equation \eqref{eq:weight}
gives, on the initial branch,
$\pi_\gamma\mathcal Y_{X,Y}(x,z)y
=z^{\gamma-\alpha-\beta}v_\gamma(x,y)$.
The positively continued branch in the preceding equality
therefore satisfies
\[
\begin{aligned}
 \pi_\gamma\overline{B_{X,Y}}\,
      \mathcal Y_{X,Y}(x,z)y
 &=e^{2\pi i(\alpha+\beta)}
   e^{2\pi i(\gamma-\alpha-\beta)}
      \pi_\gamma\mathcal Y_{X,Y}(x,z)y\\
 &=e^{2\pi i\gamma}
      \pi_\gamma\mathcal Y_{X,Y}(x,z)y\\
 &=\pi_\gamma\overline{\theta_U}\,
      \mathcal Y_{X,Y}(x,z)y.
\end{aligned}
\]
Here $\pi_\gamma$ refers throughout to the $\gamma$-weight
space of the common output $T_kU$, which is the same graded
vector space as $U$.  Thus the last scalar is exactly the
weightwise action of $e^{2\pi iL_U(0)}:U\to T_kU$.
The positive full-turn factor is essential: a negative full turn
would produce the inverse factor and would not prove
\eqref{eq:bal} with this twist.  Ordinary grading restriction
eliminates logarithmic polynomial terms and makes both
completed morphisms and every weight projection well defined.
Equality of all weight projections implies equality in the
weight-product completion.

Every vector of an ordinary module is a finite sum of its
homogeneous-weight components, so the preceding identity holds
for arbitrary $x\in X$, $y\in Y$ by linearity.  Evaluating the
resulting formal equality at $z=1$ with $\log1=0$ gives
\[
 \overline{B_{X,Y}}\,I^1_{X,Y}
 =\overline{\theta_U}\,I^1_{X,Y}:
 X\otimes Y\longrightarrow\overline{T_kU}.
\]
The binary representing property of the chosen twisted
$P(1)$-tensor product identifies
\[
 \Hom_V(U,T_kU)
 \xrightarrow{\ \sim\ }
 \mathcal I_{P(1),0}\binom{T_kU}{X\;Y},\qquad
 f\longmapsto\overline f\,I^1_{X,Y}.
\]
Its injectivity, applied to the two morphisms $B_{X,Y}$ and
$\theta_U$, proves \eqref{eq:bal}.  This argument uses only the
canonical operator, its naturality, the tensorator normalization,
Lemmas~\ref{lem:theta} and \ref{lem:double}, and the imported
binary representing property.  For objects containing several
homogeneous $G$-components, the balancing identity is extended
additively from pairs of homogeneous summands; it does not
assign a single degree to an inhomogeneous object.
\end{proof}

\section{Contragredients and the ribbon axiom}

\subsection{Restricted contragredients and categorical duals}\label{subsec:contragredientduals}
We use the right-dual evaluations and coevaluations of Section~2.
The rigidity result in \cite{Crossed} supplies the categorical duals;
the contragredient-pairing lemma of \cite{Crossed}
supplies, independently of categorical duality, natural morphisms
$e_X:X\bo DX\to\mathbf1$ which are nonzero for $X\ne0$.
The restricted-dual pullback comparison $\tau_{k,X}$ is the one
fixed in Section~2. We use the pairing naturality and its
$G$-covariance as explicit inputs from \cite{Crossed}.

\begin{proposition}\label{prop:dual}
The pairings $e_X$ give natural isomorphisms
$\phi_X:DX\xrightarrow{\sim}X^\vee$, characterized by
$e_X=\ev_X(\id_X\bo\phi_X)$. Under these isomorphisms,
the conformal twist satisfies the following identity, with
no implicit object identification: for $X\in\cC_g$,
\begin{equation}\label{eq:dualtwist}
 \theta_{DX}
 =T_{g^{-1}}(D\theta_X)\circ
  \bigl(T_{g^{-1}}(\tau_{g,X})\bigr)^{-1}
 :DX\longrightarrow T_{g^{-1}}DX.
\end{equation}
Moreover, the identifications $\phi_X$ intertwine the restricted
and categorical duals of the monoidal $G$-action:
\begin{equation}\label{eq:dualactionphi}
 \phi_{T_kX}=b_{k,X}\circ T_k(\phi_X)\circ\tau_{k,X}:
 D(T_kX)\longrightarrow(T_kX)^\vee.
\end{equation}
Consequently the categorical dual-twist identity
\eqref{eq:categoricaldualtwistaxiom} holds.
\end{proposition}
\begin{proof}
We work first with $X\in\cC_g$. The two potential duals are
separate objects of $\cC_{g^{-1}}$: $DX$ is the weightwise
restricted contragredient, whereas $X^\vee$ is the chosen
categorical right dual, with
$\ev_X:X\bo X^\vee\to\mathbf1$ and
$\coev_X:\mathbf1\to X^\vee\bo X$. The contravariant
functors $D$ and $(-)^\vee$ are equivalences on the appropriate
sectors: $D^2\cong\id$ by the restricted bidual comparison
of Section~2, and categorical right duality by rigidity.

For any $U$, the right-dual adjunction
\cite[Definition~2.10.2 and Proposition~2.10.8(b)]{EGNO}
identifies $\Hom(U,X^\vee)$ with
$\Hom(X\bo U,\mathbf1)$ by
$h\mapsto\ev_X(\id_X\bo h)$. Its inverse sends
$p:X\bo U\to\mathbf1$ to the following morphism:
\[
\begin{aligned}
 U&\xrightarrow{\,l_U^{-1}\,}\mathbf1\bo U
 \xrightarrow{\,\coev_X\bo\id_U\,}(X^\vee\bo X)\bo U\\
 &\xrightarrow{\,a_{X^\vee,X,U}\,}X^\vee\bo(X\bo U)
 \xrightarrow{\,\id_{X^\vee}\bo p\,}X^\vee\bo\mathbf1
 \xrightarrow{\,r_{X^\vee}\,}X^\vee.
\end{aligned}
\]
Indeed, substituting this composite in
$\ev_X(\id_X\bo-)$ and applying associator and unitor
naturality reduces it to $p$ by the first right-dual triangle
identity. Conversely, substituting
$p=\ev_X(\id_X\bo h)$ reduces the inverse composite to $h$
by the second triangle identity. Both reductions use the
right-dual convention of Section~2, including the displayed
orientations of $a,l,r$; no strict associativity is assumed.

Let $\phi_X:DX\to X^\vee$ be the unique mate of the pairing
$e_X:X\bo DX\to\mathbf1$ constructed in \cite{Crossed}. Thus
$e_X=\ev_X(\id_X\bo\phi_X).$
Since the mate bijection is linear, $e_X\ne0$ implies
$\phi_X\ne0$. If $X$ is simple, $DX$ is simple under the
restricted-contragredient anti-equivalence, and $X^\vee$ is
simple under categorical right-dual anti-equivalence. Both
belong to $\cC_{g^{-1}}$, so $\phi_X$ is an isomorphism
by simplicity. This uses the nonvanishing of $e_X$ from
\cite[Lemma~10.1]{Crossed}, not any
prior identification of the two duals.

We establish naturality before extending invertibility beyond
simple objects. If $f:X\to Y$ lies in one homogeneous sector,
the pairing identity imported from \cite{Crossed} is
\[
 e_Y(f\bo\id_{DY})=e_X(\id_X\bo Df):
 X\bo DY\longrightarrow\mathbf1.
\]
The definition of $f^\vee:Y^\vee\to X^\vee$ gives
$\ev_Y(f\bo\id_{Y^\vee})
 =\ev_X(\id_X\bo f^\vee)$, both with source
$X\bo Y^\vee$. Substituting the mate formulas therefore
yields
\[
 \ev_X(\id_X\bo f^\vee\phi_Y)
 =\ev_X(\id_X\bo\phi_X Df):
 X\bo DY\longrightarrow\mathbf1.
\]
Injectivity of the right-dual mate correspondence gives
$\phi_X Df=f^\vee\phi_Y$ as an identity of morphisms
$DY\to X^\vee$. No $\phi_X^{-1}$ has been used here.

For arbitrary homogeneous $X$, choose a finite decomposition
$X=\bigoplus_{i=1}^m S_i$ with $S_i$ simple, writing
$\iota_i:S_i\to X$, $p_i:X\to S_i$ for its inclusions
and projections. Repetitions among the $S_i$ are allowed.
Contravariance gives $D\iota_i:DX\to DS_i$ and
$Dp_i:DS_i\to DX$, with
$D\iota_i Dp_i=\id_{DS_i}$,
$D\iota_i Dp_j=0$ for $i\ne j$, and
$\sum_i Dp_iD\iota_i=\id_{DX}$. The pairing naturality,
applied to $\iota_i$ and then precomposed with
$\id_{S_i}\bo Dp_j$, gives the precise block relation
\[
 e_X(\iota_i\bo Dp_j)
 =\begin{cases}e_{S_i},&i=j,\\0,&i\ne j,\end{cases}
 :S_i\bo DS_j\longrightarrow\mathbf1.
\]
In particular, it excludes off-diagonal blocks even for
isomorphic summands and their multiplicity spaces.
Applying this naturality identity to $p_i$ gives
$\phi_X Dp_i=p_i^\vee\phi_{S_i}$. Consequently
\[
 \phi_X=\sum_i p_i^\vee\phi_{S_i}D\iota_i,
 \qquad
 \phi_X^{-1}=\sum_i Dp_i\phi_{S_i}^{-1}\iota_i^\vee.
\]
The relations $p_i\iota_i=\id_{S_i}$,
$p_i\iota_j=0$ for $i\ne j$, and
$\sum_i\iota_i p_i=\id_X$, together with their
contravariant duals, show directly that the two displayed
sums are inverse. Hence $\phi_X$ is an isomorphism for
every homogeneous $X$, including $X=0$. Since it is the
unique mate of the intrinsically defined $e_X$, the
isomorphism does not depend on the decomposition.

Now set $\ev_X^D=e_X$ and transport the right-dual
coevaluation along $\phi_X$:
\begin{equation*}
 \coev_X^D=(\phi_X^{-1}\bo\id_X)\coev_X:
 \mathbf1\longrightarrow DX\bo X.
\end{equation*}
All arrows are in the neutral degree when appropriate:
$X\bo DX$ and $DX\bo X$ have degree $1$.
With the associator $a$ and unitors $l,r$ from Section~2,
the two proposed right-dual triangle composites are
\begin{align*}
 T_X^D&=l_X(e_X\bo\id_X)a^{-1}_{X,DX,X}
   (\id_X\bo\coev_X^D)r_X^{-1}:X\longrightarrow X,\\
 T_{DX}^D&=r_{DX}(\id_{DX}\bo e_X)a_{DX,X,DX}
   (\coev_X^D\bo\id_{DX})l_{DX}^{-1}:
   DX\longrightarrow DX.
\end{align*}
For the first composite, the intermediate objects are
$X\bo\mathbf1$, $X\bo(DX\bo X)$,
$(X\bo DX)\bo X$ and $\mathbf1\bo X$.
Substitution and associator naturality give the explicit
cancellation
\[
\begin{aligned}
 T_X^D
 &=l_X(\ev_X\bo\id_X)
 ((\id_X\bo\phi_X)\bo\id_X)
 a^{-1}_{X,DX,X}
 (\id_X\bo(\phi_X^{-1}\bo\id_X))
 (\id_X\bo\coev_X)r_X^{-1}\\
 &=l_X(\ev_X\bo\id_X)a^{-1}_{X,X^\vee,X}
 (\id_X\bo\coev_X)r_X^{-1}=\id_X.
\end{aligned}
\]
The second composite passes through
$\mathbf1\bo DX$, $(DX\bo X)\bo DX$,
$DX\bo(X\bo DX)$ and $DX\bo\mathbf1$.
Naturality of $a,l,r$ and the same evaluation/coevaluation
substitutions yield, with every intervening occurrence of
$\phi_X\phi_X^{-1}$ cancelled,
\[
\begin{aligned}
 \phi_X T_{DX}^D
 &=r_{X^\vee}(\id_{X^\vee}\bo\ev_X)
 a_{X^\vee,X,X^\vee}
 (\coev_X\bo\id_{X^\vee})l_{X^\vee}^{-1}\phi_X\\
 &=\phi_X.
\end{aligned}
\]
The last equality is the second right-dual triangle identity
for $X^\vee$. Since $\phi_X$ is invertible,
$T_{DX}^D=\id_{DX}$. Thus $(DX,e_X,\coev_X^D)$
is a right dual of $X$ with both triangle identities,
for all homogeneous $X$. The naturality proved above is the
fully typed identity
\begin{equation}\label{eq:dualnaturality}
 \phi_X\circ Df=f^\vee\circ\phi_Y:
 DY\longrightarrow X^\vee.
\end{equation}

We next compare the two kinds of dual under pullback. For
$X\in\cC_g$, the objects $D(T_kX)$, $T_k(DX)$,
$T_k(X^\vee)$, and $(T_kX)^\vee$ all lie in
$\cC_{kg^{-1}k^{-1}}$. In particular, the comparison
$\tau_{k,X}:D(T_kX)\to T_k(DX)$ is distinct from the
categorical action--dual comparison
$b_{k,X}:T_k(X^\vee)\to(T_kX)^\vee$.

The map $\tau_{k,X}$ is the inverse of the weightwise identity
$\vartheta_{k,X}:T_k(DX)\to D(T_kX)$ constructed in Section~2.
Indeed, writing $u^\mathrm{op}(z)=e^{zL(1)}(-z^{-2})^{L(0)}u$,
the two restricted-dual vertex operators have the same matrix
coefficients against $w\in X$:
\[
\begin{aligned}
 \langle Y_{D(T_kX)}(u,z)\varphi,w\rangle
 &=\langle\varphi,Y_X(k^{-1}u^\mathrm{op}(z),z^{-1})w\rangle\\
 &=\langle Y_{T_k(DX)}(u,z)\varphi,w\rangle.
\end{aligned}
\]
Here $k^{-1}u^\mathrm{op}(z)=(k^{-1}u)^\mathrm{op}(z)$,
since $k$ fixes the conformal vector. Hence $\tau_{k,X}$ is
a twisted-module isomorphism, rather than merely an
identification of graded vector spaces. For a morphism $f:X\to Y$
in a fixed sector its contravariant naturality is
\[
 \tau_{k,X}\,D(T_kf)=T_k(Df)\,\tau_{k,Y}:
 D(T_kY)\longrightarrow T_k(DX).
\]
Since all $\tau$'s are identity maps on these restricted-dual
weight spaces, the strict underlying action gives, as a composite
$D(T_kT_lX)\to T_kD(T_lX)\to T_kT_l(DX)$,
$\tau_{kl,X}=T_k(\tau_{l,X})\,\tau_{k,T_lX}, \tau_{1,X}=\id_{DX}.$
No monoidal tensorator enters this restricted-dual comparison.

The pairing construction in \cite{Crossed} is independent of the
chosen categorical dual: in its contragredient-pairing lemma,
its skew and contragredient transforms of the twisted module
vertex operator are mapped to $V$ using the normalized invariant
form and the vacuum comparison $j_g^{-1}$. Their universal binary
factorization gives $e_X:X\bo DX\to\mathbf1$. In particular,
the covariance identity is an additional property of this
construction, not a consequence of pairing naturality alone.
Written with $\tau_{k,X}=\vartheta_{k,X}^{-1}$, it is
\begin{equation*}
 e_{T_kX}=j_k^{-1}\,T_k(e_X)\,J^k_{X,DX}
       (\id_{T_kX}\bo\tau_{k,X}):
 T_kX\bo D(T_kX)\longrightarrow\mathbf1.
\end{equation*}
The four successive arrows on the right have targets
$T_kX\bo T_k(DX)$, $T_k(X\bo DX)$,
$T_k\mathbf1$, and $\mathbf1$, respectively.
Thus the inverse unit map $j_k^{-1}$ and the tensorator $J^k$
cannot be suppressed. The comparison is understood in the
same $P(1)$ and logarithmic normalization as the pairing in
\cite{Crossed} and the action tensorator of Section~2.

For completeness, the comparison $b_{k,X}$ referred to in
Lemma~\ref{lem:dualactioncomparison} is independently defined
by the right-dual evaluation mate of
\[
 j_k^{-1}T_k(\ev_X)J^k_{X,X^\vee}:
 T_kX\bo T_k(X^\vee)\longrightarrow\mathbf1.
\]
The strong monoidal equivalence $T_k$ carries the evaluation
and coevaluation of $X$ to a right-dual pair for $T_kX$;
uniqueness of right duals makes its mate $b_{k,X}$ an
isomorphism. Its defining comparison is
\eqref{eq:bcomparison}. The mate $b_{k,X}$ is defined from $J^k$, $j_k$, and rigidity;
the identification $DX\simeq X^\vee$ is established separately
in Proposition~\ref{prop:dual}.

We may therefore calculate in
$\Hom(T_kX\bo D(T_kX),\mathbf1)$, using first
\eqref{eq:bcomparison}, then naturality of $J^k$, and finally
the mate identity $e_X=\ev_X(\id_X\bo\phi_X)$:
\[
\begin{aligned}
&\ev_{T_kX}
 (\id_{T_kX}\bo b_{k,X}T_k(\phi_X)\tau_{k,X})\\
&\quad=j_k^{-1}T_k(\ev_X)J^k_{X,X^\vee}
 (\id_{T_kX}\bo T_k(\phi_X)\tau_{k,X})\\
&\quad=j_k^{-1}T_k(\ev_X)
 T_k(\id_X\bo\phi_X)J^k_{X,DX}
 (\id_{T_kX}\bo\tau_{k,X})\\
&\quad=j_k^{-1}T_k(e_X)J^k_{X,DX}
 (\id_{T_kX}\bo\tau_{k,X})
 =e_{T_kX}.
\end{aligned}
\]
The two candidate maps to $(T_kX)^\vee$ have common source
$D(T_kX)$; their evaluations are both $e_{T_kX}$.
The right-dual adjunction
$\Hom(D(T_kX),(T_kX)^\vee)\cong
\Hom(T_kX\bo D(T_kX),\mathbf1)$ is injective. Hence
\eqref{eq:dualactionphi} follows, retaining $\tau_{k,X}$ and
$b_{k,X}$ as the distinct comparison maps.

The compatibility with morphisms is contravariant. For
$f:X\to Y$, the comparison maps satisfy
\[
\begin{aligned}
 \tau_{k,X}D(T_kf)&=T_k(Df)\tau_{k,Y},\\
 b_{k,X}T_k(f^\vee)&=(T_kf)^\vee b_{k,Y}.
\end{aligned}
\]
Together with $\phi_XDf=f^\vee\phi_Y$ these equations identify
both routes from $D(T_kY)$ to $(T_kX)^\vee$; for example,
\[
\begin{aligned}
 &b_{k,X}T_k(\phi_X)\tau_{k,X}D(T_kf)\\
 &\quad=b_{k,X}T_k(f^\vee\phi_Y)\tau_{k,Y}
 =(T_kf)^\vee b_{k,Y}T_k(\phi_Y)\tau_{k,Y}.
\end{aligned}
\]
The identities extend to finite direct sums by additivity,
including repeated isomorphic simple summands.

It remains to verify compatibility with group multiplication,
without treating the monoidal action as strict. The defining
evaluation of the composite
$T_kT_l(X^\vee)\xrightarrow{T_k(b_{l,X})}
 T_k((T_lX)^\vee)\xrightarrow{b_{k,T_lX}}
 (T_kT_lX)^\vee$ is
\[
\begin{aligned}
 &\ev_{T_kT_lX}
  (\id\bo b_{k,T_lX}T_k(b_{l,X}))\\
 &\quad=j_k^{-1}T_k(\ev_{T_lX})
  J^k_{T_lX,(T_lX)^\vee}
  (\id\bo T_k(b_{l,X}))\\
 &\quad=j_k^{-1}T_k(j_l^{-1})T_{kl}(\ev_X)
  T_k(J^l_{X,X^\vee})J^k_{T_lX,T_l(X^\vee)}\\
 &\quad=j_{kl}^{-1}T_{kl}(\ev_X)J^{kl}_{X,X^\vee}.
\end{aligned}
\]
Here the second equality uses naturality of $J^k$; the last
uses the actual unit and tensorator coherences
$j_{kl}=T_k(j_l)j_k$ and
$J^{kl}_{X,X^\vee}=T_k(J^l_{X,X^\vee})
 J^k_{T_lX,T_l(X^\vee)}$.
Evaluation-mate uniqueness now proves
$b_{kl,X}=b_{k,T_lX}T_k(b_{l,X}), b_{1,X}=\id_{X^\vee}.$
Consequently, applying \eqref{eq:dualactionphi} first to $l$
and then to $k$ gives the same map as applying it to $kl$:
\[
\begin{aligned}
 \phi_{T_kT_lX}
 &=b_{k,T_lX}T_k(\phi_{T_lX})\tau_{k,T_lX}\\
 &=b_{k,T_lX}T_k(b_{l,X})T_{kl}(\phi_X)
   T_k(\tau_{l,X})\tau_{k,T_lX}\\
 &=b_{kl,X}T_{kl}(\phi_X)\tau_{kl,X}.
\end{aligned}
\]
In particular this is valid without commuting $k$ and $l$.
At $k=1$ the three comparisons reduce to identities; on the
vacuum object the normalized invariant form and $j_k$ give the
same evaluation comparison, so the identity is compatible with
the unit constraints. Drinfeld isomorphism, or any subsequent ribbon result.

Now fix $X\in\cC_g$ and put $h=g^{-1}$.  By the strict
underlying pullback convention, $T_hT_g=T_1=\id$, whereas the
monoidal tensorators of these functors need not be identities.
The contravariant restricted dual and the contravariant categorical
right dual give, respectively,
\[
\begin{gathered}
 \theta_X:X\longrightarrow T_gX,\qquad
 D\theta_X:D(T_gX)\longrightarrow DX,\\
 \tau_{g,X}:D(T_gX)\xrightarrow{\sim}T_g(DX),\qquad
 T_h(\tau_{g,X}):T_hD(T_gX)\xrightarrow{\sim}DX,\\
 T_h(D\theta_X):T_hD(T_gX)\longrightarrow T_h(DX),\\
 \theta_{DX}:DX\longrightarrow T_h(DX).
\end{gathered}
\]
Here $D(T_gX)$, $T_g(DX)$ and $DX$ are all in $\cC_h$;
so are the sources and targets of the displayed morphisms after
applying $T_h$, since $gh=hg=1$ for this fixed element $g$.
In particular, $(T_h\tau_{g,X})^{-1}$ has source $DX$ and
target $T_hD(T_gX)$; the composite on the right of
\eqref{eq:dualtwist} is therefore a morphism
$DX\to T_h(DX)$, with the same source and target as $\theta_{DX}$.

To compute it, let $\varphi\in(DX)_\lambda=X_\lambda^*$.
The pullback $T_gX$ has the same conformal operator as $X$,
since $g\omega=\omega$, and thus
$(D(T_gX))_\lambda=X_\lambda^*$ canonically as graded vector
spaces.  By the definition of the restricted contravariant
transpose, for $x\in X_\lambda$,
\[
 \langle D\theta_X(\varphi),x\rangle
 =\langle\varphi,\theta_X x\rangle
 =e^{2\pi i\lambda}\langle\varphi,x\rangle.
\]
This is a complex-linear transpose: it is neither the inverse
operator nor the conjugate-linear adjoint, so its scalar is
$e^{2\pi i\lambda}$, not $e^{-2\pi i\lambda}$ or its complex
conjugate.  The opposite-vertex definition
\eqref{eq:restrictedcontragredient} makes $L_{DX}(0)$ the
weightwise transpose of $L_X(0)$; hence $\varphi$ has weight
$\lambda$ in $DX$.  Moreover, $\tau_{g,X}$ is the
weightwise identity twisted-module isomorphism verified above:
its inverse after applying $T_h$ sends $\varphi\in(DX)_\lambda$
to the \emph{same functional} in
$\bigl(T_hD(T_gX)\bigr)_\lambda$.  The next arrow
$T_h(D\theta_X)$ maps this functional to
$e^{2\pi i\lambda}\varphi\in(T_hDX)_\lambda$.
This is exactly the action of
$\theta_{DX}=e^{2\pi iL_{DX}(0)}$ on that weight space,
proving \eqref{eq:dualtwist} as an equality of twisted-module
morphisms.  Ordinary objects are algebraic direct sums of their
weight spaces, so no completion or convergence argument is needed.
The same calculation applies to every finite semisimple direct
sum, including the zero object, entirely within the restricted-dual
construction. The categorical right-dual comparison follows below.

We now compare the categorical twists, keeping the two dual
functors separate.  The maps involved have types
\[
\begin{gathered}
 \phi_X:DX\xrightarrow{\sim}X^\vee,\qquad
 \phi_{T_gX}:D(T_gX)\xrightarrow{\sim}(T_gX)^\vee,\\
 b_{g,X}:T_g(X^\vee)\xrightarrow{\sim}(T_gX)^\vee,\qquad
 \theta_X^\vee:(T_gX)^\vee\longrightarrow X^\vee,\\
 T_h(b_{g,X}):X^\vee\longrightarrow T_h((T_gX)^\vee),\\
 T_h(\theta_X^\vee):T_h((T_gX)^\vee)
        \longrightarrow T_h(X^\vee),\\
 \theta_{X^\vee}:X^\vee\longrightarrow T_h(X^\vee),\qquad
 T_h(\phi_X):T_h(DX)\longrightarrow T_h(X^\vee).
\end{gathered}
\]
The first source of $T_h(b_{g,X})$ is literally $X^\vee$
because $T_hT_g=\id$ as underlying functors.  In particular,
all morphisms in the dual-twist comparison are in $\cC_h$.
Equation \eqref{eq:dualactionphi} at $k=g$ gives, after
composition on the right with $\tau_{g,X}^{-1}$,
\[
 b_{g,X}T_g(\phi_X)
   =\phi_{T_gX}\tau_{g,X}^{-1}:
 T_g(DX)\longrightarrow(T_gX)^\vee.
\]
Applying $T_h$ and the strict underlying action, without
replacing either $b_{g,X}$ or $\tau_{g,X}$ by an identity,
yields the equality of maps
$DX\to T_h((T_gX)^\vee)$,
$T_h(b_{g,X})\phi_X =T_h(\phi_{T_gX})(T_h\tau_{g,X})^{-1}.$
Since $\theta_X:X\to T_gX$ is a morphism in the \emph{same}
sector $\cC_g$, the contravariant naturality
\eqref{eq:dualnaturality}, applied with $f=\theta_X$, is
\[
 \theta_X^\vee\phi_{T_gX}=\phi_X D\theta_X:
 D(T_gX)\longrightarrow X^\vee.
\]
The preceding two equalities and \eqref{eq:dualtwist} give,
now as successive equalities in
$\Hom(DX,T_h(X^\vee))$,
\[
\begin{aligned}
 &T_h(\theta_X^\vee)T_h(b_{g,X})\phi_X\\
 &\quad=T_h(\theta_X^\vee\phi_{T_gX})
           (T_h\tau_{g,X})^{-1}\\
 &\quad=T_h(\phi_X D\theta_X)
           (T_h\tau_{g,X})^{-1}\\
 &\quad=T_h(\phi_X)T_h(D\theta_X)
           (T_h\tau_{g,X})^{-1}\\
 &\quad=T_h(\phi_X)\theta_{DX}.
\end{aligned}
\]
Finally, naturality of the conformal crossed twist on the
$h$-sector, applied to the module isomorphism
$\phi_X:DX\to X^\vee$, gives
\[
 T_h(\phi_X)\theta_{DX}=\theta_{X^\vee}\phi_X:
 DX\longrightarrow T_h(X^\vee).
\]
Since $\phi_X$ is invertible by the first part of this
proposition, cancelling it on the right proves
\eqref{eq:categoricaldualtwistaxiom}.  This is precisely the
type-correct expression of the third condition of
\cite[Lemma~2.3]{Kirillov} in the right-dual convention of
Definition~\ref{def:twistaxioms}: Kirillov writes
$\theta_{V^*}=R_{g^{-1}}(\theta_V^*)$ while suppressing the
canonical identification between the dual of a pulled-back
object and the pullback of a dual.  In the present conventions
that identification is $T_h(b_{g,X})$.
\end{proof}

\subsection{The crossed Drinfeld isomorphism}\label{subsec:crosseddrinfeld}

\begin{definition}\label{def:crosseddrinfeld}
For homogeneous $X\in\cC_g$, put
\[
 f_X:=\ev_X\circ c_{X^\vee,T_gX}:
 X^\vee\bo T_gX\longrightarrow\mathbf1.
\]
This is well typed: $X^\vee\in\cC_{g^{-1}}$, so the crossed
braiding has target
$T_{g^{-1}}(T_gX)\bo X^\vee=X\bo X^\vee,$
where $T_{g^{-1}}T_g=T_1=\id$ is the underlying action identity
for this single-object comparison. The monoidal tensorators enter in
the subsequent monoidal calculations.  The resulting $f_X$ has target
$\mathbf1$ and is a morphism in the neutral component.

Since $X^{\vee\vee}$ is the chosen right dual of $X^\vee$, with
$\ev_{X^\vee}:X^\vee\bo X^{\vee\vee}\to\mathbf1$ and
$\coev_{X^\vee}:\mathbf1\to X^{\vee\vee}\bo X^\vee$, the right-dual
adjunction gives the natural bijection
\[
\begin{aligned}
 \operatorname{Hom}(X^\vee\bo T_gX,\mathbf1)
 &\xrightarrow{\ \sim\ }
 \operatorname{Hom}(T_gX,X^{\vee\vee}),\\
 p&\longmapsto
 r_{X^{\vee\vee}}(\id_{X^{\vee\vee}}\bo p)
 a_{X^{\vee\vee},X^\vee,T_gX}
 (\coev_{X^\vee}\bo\id_{T_gX})l_{T_gX}^{-1}.
\end{aligned}
\]
Its inverse sends $u:T_gX\to X^{\vee\vee}$ to
$\ev_{X^\vee}(\id_{X^\vee}\bo u)$.  This is the right-dual
adjunction of \cite[Proposition~2.10.8(b)]{EGNO}, applied to
$X^\vee$; it is not the left-dual adjunction.  Define
$\mathfrak u_X:T_gX\longrightarrow X^{\vee\vee}$
to be the mate of $f_X$.  Equivalently, the defining mate identity is
\begin{equation}\label{eq:drinfeldmate}
 X^\vee\bo T_gX
 \xrightarrow{\;c_{X^\vee,T_gX}\;}
 T_{g^{-1}}(T_gX)\bo X^\vee
 =X\bo X^\vee
 \xrightarrow{\ev_X}\mathbf1.
\end{equation}
After suppressing only the canonical associators and unitors, the
adjunction formula reads
\begin{equation}\label{eq:drinfeldformula}
 \mathfrak u_X
 =(\id_{X^{\vee\vee}}\bo\ev_X)
   (\id_{X^{\vee\vee}}\bo c_{X^\vee,T_gX})
   (\coev_{X^\vee}\bo\id_{T_gX}).
\end{equation}
Indeed, the successive objects in this composite are
$T_gX$, $X^{\vee\vee}\bo X^\vee\bo T_gX$,
$X^{\vee\vee}\bo X\bo X^\vee$, and $X^{\vee\vee}$;
the second step uses $T_{g^{-1}}T_gX=X$, and the last uses
$\ev_X:X\bo X^\vee\to\mathbf1$.  Thus the mate is a morphism
$T_gX\to X^{\vee\vee}$ in $\cC_g$.
\end{definition}

\begin{lemma}\label{lem:dualactioncomparison}
Let $X$ be any object and $k\in G$.  There is a unique natural
isomorphism
$b_{k,X}:T_k(X^\vee)\xrightarrow{\sim}(T_kX)^\vee$
whose evaluation comparison is
\begin{equation}\label{eq:bcomparison}
 \ev_{T_kX}(\id_{T_kX}\bo b_{k,X})
 =j_k^{-1}T_k(\ev_X)J^k_{X,X^\vee}.
\end{equation}
It induces the double-dual comparison
\begin{equation}\label{eq:chi2}
 \chi^{(2)}_{k,X}:=(b_{k,X}^{-1})^\vee b_{k,X^\vee}:
 T_k(X^{\vee\vee})\xrightarrow{\sim}(T_kX)^{\vee\vee}.
\end{equation}
These comparisons satisfy the coherence equations for the strict
underlying $G$-action, with the nontrivial tensorators $J^k$ retained.
\end{lemma}
\begin{proof}
The right side of \eqref{eq:bcomparison} has domain
$T_kX\bo T_k(X^\vee)$ and codomain $\mathbf1$. By the
representing property of the chosen right dual $(T_kX)^\vee$,
it is the evaluation pairing of a unique $b_{k,X}$. Strong
monoidality of $T_k$ and preservation of duals show this map
is invertible. Naturality and coherence follow from uniqueness
of these representing morphisms and the identity
$J^{kl}_{X,Y}=T_k(J^l_{X,Y})J^k_{T_lX,T_lY}$; they are not
asserted with the $J$'s replaced by identity maps.
In \eqref{eq:chi2}, the first factor $b_{k,X^\vee}$ has target
$(T_kX^\vee)^\vee$, while $(b_{k,X}^{-1})^\vee$ has source
$(T_kX^\vee)^\vee$ and target $(T_kX)^{\vee\vee}$.
Thus the composite is correctly typed.
\end{proof}

\subsection{The pivotal structure and its coherence}\label{subsec:pivotalstructure}

\begin{proposition}\label{prop:pivotal}
Under the rigidity and crossed-hexagon identities of \cite[Theorem~1.1]{Crossed} and
Theorem~\ref{thm:balance} and Proposition~\ref{prop:dual},
the natural maps
$\delta_X:=\mathfrak u_X\theta_X: X\xrightarrow{\sim}X^{\vee\vee}$
form the pivotal structure corresponding to the conformal
crossed twist in Kirillov's Lemma~2.3. In particular, writing
$\mu_{X,Y}:X^{\vee\vee}\bo Y^{\vee\vee}
\xrightarrow{\sim}(X\bo Y)^{\vee\vee}$ for the canonical
monoidal structure of the double-right-dual functor, one has
\begin{align}
 \delta_{X\bo Y}&=\mu_{X,Y}(\delta_X\bo\delta_Y),
                         \label{eq:pivotalmonoidal}\\
 \delta_{T_kX}&=\chi^{(2)}_{k,X}\,T_k(\delta_X),
                         \label{eq:pivotalG}\\
 \delta_{X^\vee}&=(\delta_X^{-1})^\vee.
                         \label{eq:pivotaldual}
\end{align}
The equalities include the indicated canonical duality
identifications and the original action tensorators.
\end{proposition}
\begin{proof}
We first verify the typing, naturality and invertibility of the
explicit Drinfeld construction directly.  The right-dual mate construction in
Definition~\ref{def:crosseddrinfeld} is natural with respect to
$X$ in each fixed sector.  To prove this, let $f:X\to Y$ be a
morphism with $X,Y\in\cC_g$, and put
$f_X=\ev_Xc_{X^\vee,T_gX}$ and
$f_Y=\ev_Yc_{Y^\vee,T_gY}$.  The two morphisms
$\mathfrak u_YT_g(f)$ and $f^{\vee\vee}\mathfrak u_X$ have source
$T_gX$ and target $Y^{\vee\vee}$.  Under the right-dual mate
bijection for $Y^\vee$, their respective mates belong to
$\operatorname{Hom}(Y^\vee\bo T_gX,\mathbf1)$ and are
$f_Y(\id_{Y^\vee}\bo T_g(f))$ and
$f_X(f^\vee\bo\id_{T_gX})$.  These are equal, since
\[
\begin{aligned}
 f_Y(\id_{Y^\vee}\bo T_g(f))
 &=\ev_Yc_{Y^\vee,T_gY}
       (\id_{Y^\vee}\bo T_g(f))\\
 &=\ev_Y(T_{g^{-1}}T_g(f)\bo\id_{Y^\vee})
       c_{Y^\vee,T_gX}\\
 &=\ev_Y(f\bo\id_{Y^\vee})c_{Y^\vee,T_gX}\\
 &=\ev_X(\id_X\bo f^\vee)c_{Y^\vee,T_gX}\\
 &=\ev_Xc_{X^\vee,T_gX}
       (f^\vee\bo\id_{T_gX})\\
 &=f_X(f^\vee\bo\id_{T_gX}).
\end{aligned}
\]
The second line is naturality of the crossed braiding in its second
input; the third uses the strict underlying identity
$T_{g^{-1}}T_g(f)=f$; the fourth is naturality of evaluation; and
the fifth is naturality of the crossed braiding in its first input.
The mate correspondence is injective, so this proves
$\mathfrak u_Y\,T_g(f)=f^{\vee\vee}\,\mathfrak u_X.$

Now suppose $S$ is a nonzero simple object of $\cC_g$.  The right-dual
triangle identities imply $\ev_S\ne0$: otherwise the triangle
composite equal to $\id_S$ would vanish.  Since
$c_{S^\vee,T_gS}$ is an isomorphism, $f_S=\ev_Sc_{S^\vee,T_gS}$
is nonzero; otherwise composing with the inverse crossed braiding
would give $\ev_S=0$.  The mate correspondence is a linear
isomorphism, hence $\mathfrak u_S\ne0$.  The pullback $T_g$ is an
equivalence and therefore preserves simple objects, and right duality
is a contravariant equivalence in a rigid fusion category, so
$S^{\vee\vee}$ is simple as well.  A nonzero morphism between simple
objects in a semisimple $\mathbb C$-linear category is an isomorphism;
thus $\mathfrak u_S$ is invertible.

Every object of a homogeneous component is a finite direct sum of
simple objects because Xu's construction gives a fusion category
\cite[Theorem~1.1]{Crossed}. For a
decomposition $X\cong\bigoplus_i S_i$, naturality with each inclusion
$\iota_i:S_i\to\bigoplus_iS_i$ gives
$\mathfrak u_{\oplus_iS_i}T_g(\iota_i)
=\iota_i^{\vee\vee}\mathfrak u_{S_i}$.  Since $T_g$ and the double-dual
functor preserve finite biproducts, these equations identify
$\mathfrak u_{\oplus_iS_i}$ with the direct sum of the maps
$\mathfrak u_{S_i}$ under the canonical biproduct identifications.
This handles repeated simple summands as well: every inclusion is
respected, so no off-diagonal component remains.  Conjugating by the
chosen decomposition isomorphism and using naturality proves that
$\mathfrak u_X$ is an isomorphism independently of the decomposition.
For the zero object the unique map $0\to0$ is its identity and is an
isomorphism.  
Since $\theta$ is natural in each homogeneous sector,
$\theta_Yf=T_g(f)\theta_X$.  Consequently
\[
\begin{aligned}
 \delta_Yf
 &=\mathfrak u_Y\theta_Yf
 =\mathfrak u_YT_g(f)\theta_X\\
 &=f^{\vee\vee}\mathfrak u_X\theta_X
 =f^{\vee\vee}\delta_X.
\end{aligned}
\]
Both $\mathfrak u_X$ and $\theta_X$ are isomorphisms, so $\delta_X$
is an isomorphism and is natural on every homogeneous sector.  On a
finite direct sum of sectors, the maps assemble as the direct sum of
the homogeneous maps; no single degree is assigned to a nonhomogeneous
object.

For the unit, use the canonical right-dual identification
$\mathbf1^\vee\cong\mathbf1$, represented by
$\ev_{\mathbf1}=r_{\mathbf1}$ and
$\coev_{\mathbf1}=l_{\mathbf1}^{-1}$.  The crossed unit identity
$r_Yc_{\mathbf1,Y}=l_Y$ at $Y=\mathbf1$, together with the unit
coherence $l_{\mathbf1}=r_{\mathbf1}$, identifies
$c_{\mathbf1,\mathbf1}$ with the identity under
$\mathbf1\bo\mathbf1\cong\mathbf1$.  Substituting these maps into
\eqref{eq:drinfeldformula} and using the triangle identities shows
that $\mathfrak u_{\mathbf1}$ is the canonical unit double-dual map,
which is the identity under this normalization.  Finally,
$\theta_{\mathbf1}=j_1=\id_{\mathbf1}$ by
Definition~\ref{def:twistaxioms} and \eqref{eq:actioncoherence}.
Hence $\delta_{\mathbf1}=\id_{\mathbf1}$ under the canonical unit
identifications.  If the initially fixed right-dual representative
of $\mathbf1$ is not literally $\mathbf1$, the same statement is
read through its unique duality-preserving canonical isomorphism.

For $G$-equivariance, fix $X\in\cC_g$ and $k\in G$.  Put
$Z_k:=T_k(T_gX)=T_{kg}X =T_{kgk^{-1}}(T_kX),$
where both equalities are strict identities of the underlying action
functors.  Thus $T_kX\in\cC_{kgk^{-1}}$, while
$T_k(X^\vee)$ and $(T_kX)^\vee$ lie in
$\cC_{kg^{-1}k^{-1}}$.  The two crossed braidings
$c_{T_k(X^\vee),Z_k}, c_{(T_kX)^\vee,Z_k}$
therefore have the same acted-on first factor: indeed,
$T_{kg^{-1}k^{-1}}Z_k=T_kX$.  This fixes the group-element order needed below; the subsequent
calculation retains the displayed order of all group elements.

Transport the defining pairing by setting
\begin{equation*}
 \widetilde f_{k,X}:=
 j_k^{-1}\,T_k(f_X)\,J^k_{X^\vee,T_gX}:
 T_k(X^\vee)\bo T_k(T_gX)\longrightarrow\mathbf1.
\end{equation*}
Here $J^k_{X^\vee,T_gX}$ first maps to
$T_k(X^\vee\bo T_gX)$, $T_k(f_X)$ then maps to
$T_k\mathbf1$, and $j_k^{-1}:T_k\mathbf1\to\mathbf1$ is applied
last.  In particular, $j_k^{-1}$ cannot be omitted.

Apply \eqref{eq:crossedequiv} to
$A=X^\vee\in\cC_{g^{-1}}$ and $B=T_gX$.  Since
$T_{g^{-1}}T_gX=X$, it gives the concrete identity
\begin{align*}
 T_k(c_{X^\vee,T_gX})J^k_{X^\vee,T_gX}
 &=J^k_{T_{g^{-1}}T_gX,X^\vee}
       c_{T_k(X^\vee),T_k(T_gX)}\\
 &=J^k_{X,X^\vee}
       c_{T_k(X^\vee),T_k(T_gX)}.
\end{align*}
The first crossed-braiding input has degree $g^{-1}$, so the
$J^k$ on the right has ordered subscripts $(X,X^\vee)$, not
$(X^\vee,X)$.  The common target is $T_k(X\bo X^\vee)$.
Consequently, using $f_X=\ev_Xc_{X^\vee,T_gX}$ and then the
actual evaluation comparison \eqref{eq:bcomparison}, we obtain
\begin{align*}
 \widetilde f_{k,X}
 &=j_k^{-1}T_k(\ev_X)
       J^k_{X,X^\vee}c_{T_k(X^\vee),T_k(T_gX)}\\
 &=\ev_{T_kX}
       (\id_{T_kX}\bo b_{k,X})
       c_{T_k(X^\vee),T_k(T_gX)}.
\end{align*}
The second equality is precisely
$\ev_{T_kX}(\id_{T_kX}\bo b_{k,X}) =j_k^{-1}T_k(\ev_X)J^k_{X,X^\vee};$
its domain is $T_kX\bo T_k(X^\vee)$, and
$b_{k,X}:T_k(X^\vee)\to(T_kX)^\vee$ is not being identified
with an identity map.

We next compare this transported pairing with the pairing defining
$\mathfrak u_{T_kX}$.  By the strict action identities,
$T_{kgk^{-1}}(T_kX)=Z_k$.  Naturality of the crossed braiding in
its first input, applied to
$b_{k,X}^{-1}:(T_kX)^\vee\to T_k(X^\vee)$, says
\[
 c_{T_k(X^\vee),Z_k}
   (b_{k,X}^{-1}\bo\id_{Z_k})
 = (\id_{T_kX}\bo b_{k,X}^{-1})
   c_{(T_kX)^\vee,Z_k}.
\]
Both sides have source $(T_kX)^\vee\bo Z_k$ and target
$T_kX\bo T_k(X^\vee)$.  The first tensor factor in the target
of the crossed braiding is $T_kX$, because
$T_{kg^{-1}k^{-1}}Z_k=T_kX$.  Hence the preceding identity and
$b_{k,X}b_{k,X}^{-1}=\id_{(T_kX)^\vee}$ give the exact equality
$f_{T_kX} =\widetilde f_{k,X} (b_{k,X}^{-1}\bo\id_{Z_k}).$
The comparison follows from first-variable naturality of the
crossed braiding and the evaluation comparison.

It remains to compute the mate of $\widetilde f_{k,X}$ rather than
merely asserting that mates commute with transport.  By the defining
right-dual mate identity for $\mathfrak u_X$,
$f_X=\ev_{X^\vee} (\id_{X^\vee}\bo\mathfrak u_X),$
with the fixed associators and unitors inserted.  Applying $T_k$,
using naturality of $J^k$ for $\id_{X^\vee}$ and $\mathfrak u_X$,
and then applying \eqref{eq:bcomparison} to $X^\vee$, yields
\begin{align*}
 \widetilde f_{k,X}
 &=j_k^{-1}T_k(\ev_{X^\vee})
       T_k(\id_{X^\vee}\bo\mathfrak u_X)
       J^k_{X^\vee,T_gX}\\
 &=j_k^{-1}T_k(\ev_{X^\vee})
       J^k_{X^\vee,X^{\vee\vee}}
       (\id_{T_k(X^\vee)}\bo T_k(\mathfrak u_X))\\
 &=\ev_{T_k(X^\vee)}
       (\id_{T_k(X^\vee)}\bo b_{k,X^\vee})
       (\id_{T_k(X^\vee)}\bo T_k(\mathfrak u_X))\\
 &=\ev_{T_k(X^\vee)}
       (\id_{T_k(X^\vee)}\bo
          b_{k,X^\vee}T_k(\mathfrak u_X)).
\end{align*}
In the second line the tensorator naturality identity is
\[
 T_k(\id_{X^\vee}\bo\mathfrak u_X)J^k_{X^\vee,T_gX}
 =J^k_{X^\vee,X^{\vee\vee}}
   (\id_{T_k(X^\vee)}\bo T_k(\mathfrak u_X)).
\]
In the third line we used the evaluation comparison for $X^\vee$,
namely
\[
 \ev_{T_k(X^\vee)}
   (\id_{T_k(X^\vee)}\bo b_{k,X^\vee})
 =j_k^{-1}T_k(\ev_{X^\vee})J^k_{X^\vee,X^{\vee\vee}}.
\]
Thus the right-dual mate of $\widetilde f_{k,X}$, for the right
dual $(T_k(X^\vee))^\vee$ of $T_k(X^\vee)$, is exactly
\[
 \widetilde u_{k,X}:=
 b_{k,X^\vee}T_k(\mathfrak u_X):
 Z_k\longrightarrow (T_k(X^\vee))^\vee.
\]
This computation accounts for both the action tensorator and the
unit comparison, and uses no further action compositor.

For completeness, precomposing a pairing in its first variable has
the following effect on its right-dual mate.  If
$a:A'\to A$, $p=\ev_A(\id_A\bo v):A\bo Z\to\mathbf1$, and
$p'=p(a\bo\id_Z)$, then the mate of $p'$ is $a^\vee v$.
Indeed, the defining evaluation square for the right-dual morphism
$a^\vee:A^\vee\to(A')^\vee$ is
\[
 \ev_{A'}(\id_{A'}\bo a^\vee)
 =\ev_A(a\bo\id_{A^\vee})
 :A'\bo A^\vee\longrightarrow\mathbf1,
\]
so the pairing itself can be rewritten as
\[
\begin{aligned}
 p'
 &=\ev_A(a\bo v)\\
 &=\ev_A(a\bo\id_{A^\vee})(\id_{A'}\bo v)\\
 &=\ev_{A'}(\id_{A'}\bo a^\vee)(\id_{A'}\bo v)\\
 &=\ev_{A'}(\id_{A'}\bo a^\vee v).
\end{aligned}
\]
This is exactly the mate identity; uniqueness of mates completes
the argument.  Apply this with
$A=T_k(X^\vee)$, $A'=(T_kX)^\vee$,
$a=b_{k,X}^{-1}$, and
$v=b_{k,X^\vee}T_k(\mathfrak u_X)$.  The mate of
$f_{T_kX}=\widetilde f_{k,X}(b_{k,X}^{-1}\bo\id_{Z_k})$
is therefore the following identity, with the original label retained:
\begin{equation}\label{eq:drinfeldG}
 \mathfrak u_{T_kX}
 =(b_{k,X}^{-1})^\vee
       b_{k,X^\vee}T_k(\mathfrak u_X)
 =\chi^{(2)}_{k,X}T_k(\mathfrak u_X).
\end{equation}
The directions are
\[
\begin{gathered}
 T_k(\mathfrak u_X):Z_k\to T_k(X^{\vee\vee}),\qquad
 b_{k,X^\vee}:T_k(X^{\vee\vee})\to(T_k(X^\vee))^\vee,\\
 (b_{k,X}^{-1})^\vee:(T_k(X^\vee))^\vee
       \to(T_kX)^{\vee\vee}.
\end{gathered}
\]
Thus $\chi^{(2)}_{k,X}$ has source $T_k(X^{\vee\vee})$ and
target $(T_kX)^{\vee\vee}$, while the whole composite has source
$Z_k=T_{kgk^{-1}}(T_kX)=T_k(T_gX)$ and target $(T_kX)^{\vee\vee}$,
exactly as required.  The only identifications of action objects
used here are strict underlying identities; all monoidal comparison
maps $J^k$, $j_k$, and $b_{k,X}$ remain displayed.

Finally, the crossed-twist covariance in
Definition~\ref{def:twistaxioms}(i) states
$\theta_{T_kX}=T_k(\theta_X)$ under the same strict action
identification.  Both maps have source $T_kX$ and target $Z_k$.
Consequently the already-proved \eqref{eq:drinfeldG} gives the
claimed $G$-equivariance of $\delta$ directly:
\[
\begin{aligned}
 \delta_{T_kX}
 &=\mathfrak u_{T_kX}\theta_{T_kX}\\
 &=\chi^{(2)}_{k,X}T_k(\mathfrak u_X)T_k(\theta_X)\\
 &=\chi^{(2)}_{k,X}T_k(\mathfrak u_X\theta_X)\\
 &=\chi^{(2)}_{k,X}T_k(\delta_X).
\end{aligned}
\]
Every composite here has source $T_kX$ and target $(T_kX)^{\vee\vee}$.
Drinfeld $G$-equivariance, together with the covariance of
$\theta$ established in Lemma~\ref{lem:theta}, proves
\eqref{eq:pivotalG} using \eqref{eq:drinfeldG}.

For the tensor comparison, fix homogeneous objects $X\in\cC_g$ and
$Y\in\cC_h$, and put $W:=X\bo Y$, $D:=W^\vee$,
$q:=ghg^{-1}$, and $\widehat W:=T_{gh}W$.  The two crossed
commutors in the balancing formula have the types
\[
 c_{X,Y}:W\longrightarrow T_gY\bo X,\qquad
 c_{T_gY,X}:T_gY\bo X\longrightarrow T_qX\bo T_gY,
\]
so
$B_{X,Y}:=c_{T_gY,X}c_{X,Y}:W\longrightarrow T_qX\bo T_gY.$
Moreover,
\[
 T_q\theta_X:T_qX\longrightarrow T_qT_gX=T_{gh}X,
 \qquad
 T_g\theta_Y:T_gY\longrightarrow T_gT_hY=T_{gh}Y,
\]
and therefore
\[
 A_{X,Y}:=J^{gh}_{X,Y}(T_q\theta_X\bo T_g\theta_Y):
 T_qX\bo T_gY\longrightarrow T_{gh}W.
\]
Here the equalities $T_qT_g=T_{gh}$ and $T_gT_h=T_{gh}$ concern
only the underlying action functors; the displayed $J^{gh}_{X,Y}$
is still required.  Thus $A_{X,Y}B_{X,Y}:W\to\widehat W$ is
exactly $\theta_W$ by \eqref{eq:defbalance}.
The argument uses the crossed-braiding, rigidity and monoidal-action
identities stated in Section~2 and established in Xu's construction
\cite[Theorem~1.1, Proposition~8.5, Proposition~9.6,
Proposition~10.3]{Crossed}; their concrete VOA realization is not
reproved here.

We make the right-dual comparison used in the mate calculation
explicit.  Let
$d_{X,Y}:D\xrightarrow{\sim}Y^\vee\bo X^\vee$
be the unique right-dual isomorphism characterized by
$\ev_W(\id_W\bo d_{X,Y}^{-1})=\widehat{\ev}_{X,Y},$
where $\widehat{\ev}_{X,Y}:W\bo(Y^\vee\bo X^\vee)\to\mathbf1$
is the composite which, with the indicated associators and unitors,
first applies $\ev_Y:Y\bo Y^\vee\to\mathbf1$ and then
$\ev_X:X\bo X^\vee\to\mathbf1$.  Explicitly, it is
\[
\begin{aligned}
\widehat{\ev}_{X,Y}:={}&\ev_X(\id_X\bo l_{X^\vee})
 (\id_X\bo(\ev_Y\bo\id_{X^\vee}))\\[-2pt]
&\quad\cdot(\id_X\bo a^{-1}_{Y,Y^\vee,X^\vee})
 a_{X,Y,Y^\vee\bo X^\vee}.
\end{aligned}
\]
Similarly, $\widehat{\ev}_{Y^\vee,X^\vee}$ is the evaluation of
$Y^\vee\bo X^\vee$ against its right dual
$X^{\vee\vee}\bo Y^{\vee\vee}$, first using $\ev_{X^\vee}$
and then $\ev_{Y^\vee}$.  With these choices the canonical
monoidal comparison of the double-right-dual functor is
\[
 \mu_{X,Y}=d_{X,Y}^{\vee}\,d_{Y^\vee,X^\vee}^{-1}:
 X^{\vee\vee}\bo Y^{\vee\vee}\longrightarrow W^{\vee\vee}.
\]
In particular, no literal equality $D=Y^\vee\bo X^\vee$ is used.

The precise identity to be established is the following, with the original label retained:
\begin{equation}\label{eq:drinfelddefect}
 \mathfrak u_W A_{X,Y}B_{X,Y}=\mu_{X,Y}(\delta_X\bo\delta_Y),
 \qquad \delta_Z=\mathfrak u_Z\theta_Z.
\end{equation}
For a map $m:W\to W^{\vee\vee}$, write
$m^\flat:=\ev_{W^\vee}(\id_{W^\vee}\bo m)$ for its right-dual
mate.  The two mates initially lie in $\Hom(D\bo W,\mathbf1)$:
\begin{equation*}
\mathcal L_{X,Y}:=\ev_W c_{D,\widehat W}
 (\id_D\bo A_{X,Y}B_{X,Y}),
\qquad
\mathcal R_{X,Y}:=\ev_{W^\vee}
 (\id_D\bo\mu_{X,Y}(\delta_X\bo\delta_Y)).
\end{equation*}
To compare them using the chosen tensor-dual model, precompose both
with $d_{X,Y}^{-1}\bo\id_W$.  This gives two maps in the same
Hom-space $\Hom((Y^\vee\bo X^\vee)\bo W,\mathbf1)$:
\begin{equation*}
\begin{aligned}
\widetilde{\mathcal L}_{X,Y}
&:=\mathcal L_{X,Y}(d_{X,Y}^{-1}\bo\id_W)\\
&=\widehat{\ev}_{X,Y}
 c_{Y^\vee\bo X^\vee,\widehat W}
 (\id_{Y^\vee\bo X^\vee}\bo A_{X,Y}B_{X,Y}),\\
\widetilde{\mathcal R}_{X,Y}
&:=\mathcal R_{X,Y}(d_{X,Y}^{-1}\bo\id_W)\\
&=\widehat{\ev}_{Y^\vee,X^\vee}
 (\id_{Y^\vee\bo X^\vee}\bo\delta_X\bo\delta_Y).
\end{aligned}
\end{equation*}
The first equality uses naturality of the crossed braiding in the
first variable and the defining evaluation property of $d_{X,Y}$.
The second uses $\mu_{X,Y}=d_{X,Y}^{\vee}d_{Y^\vee,X^\vee}^{-1}$
and naturality of right-dual transposition.  Here
$\delta_Z=\mathfrak u_Z\theta_Z$.  By the defining mate identity
\eqref{eq:drinfeldmate},
$\ev_{X^\vee}(\id_{X^\vee}\bo\delta_X) =\ev_Xc_{X^\vee,T_gX}(\id_{X^\vee}\bo\theta_X),$
and the analogous formula holds for $Y$.  Consequently
$\widetilde{\mathcal R}_{X,Y}$ is the iterated contraction of the $X$-pair
and then the $Y$-pair, with no permutation of the right-dual
factors.

For precision, the crossed-hexagon substitutions needed to compare
$\widetilde{\mathcal L}_{X,Y}$ with this contraction are as follows.  The second
crossed hexagon, \eqref{eq:crossedhextwo}, with
$(A,B,Z)=(Y^\vee,X^\vee,\widehat W)$ and degrees
$(h^{-1},g^{-1},gh)$, gives
\begin{equation*}
\begin{aligned}
c_{Y^\vee\bo X^\vee,\widehat W}
={}&a_{W,Y^\vee,X^\vee}
 (c_{Y^\vee,T_hW}\bo\id_{X^\vee})
 a^{-1}_{Y^\vee,T_hW,X^\vee}\\
&\quad\cdot(\id_{Y^\vee}\bo c_{X^\vee,\widehat W})
 a_{Y^\vee,X^\vee,\widehat W}.
\end{aligned}
\end{equation*}
Indeed, $T_{g^{-1}}\widehat W=T_hW$ and
$T_{h^{-1}g^{-1}}\widehat W=W$; this is the ordered product
$(gh)^{-1}=h^{-1}g^{-1}$, not its reversed product.
To split the second inputs of the two resulting crossings, use
naturality with $J^{gh}_{X,Y}$ and $J^h_{X,Y}$, and apply
\eqref{eq:crossedhexone} with
$(A,U,V)=(X^\vee,T_{gh}X,T_{gh}Y), (A,U,V)=(Y^\vee,T_hX,T_hY).$
For either substitution, the explicit form of the first crossed
hexagon is
\begin{equation*}
\begin{aligned}
c_{A,U\bo V}a_{A,U,V}
={}&(J^{a_0}_{U,V}\bo\id_A)
 a^{-1}_{T_{a_0}U,T_{a_0}V,A}
 (\id_{T_{a_0}U}\bo c_{A,V})\\
&\quad\cdot a_{T_{a_0}U,A,V}
 (c_{A,U}\bo\id_V),
\end{aligned}
\end{equation*}
where $a_0=\deg(A)$ (so $a_0=g^{-1}$ in the first
substitution and $a_0=h^{-1}$ in the second).  The required
object identities are
\[
 T_{g^{-1}}T_{gh}X=T_hX,\quad
 T_{g^{-1}}T_{gh}Y=T_hY,\quad
 T_{h^{-1}}T_hX=X,\quad T_{h^{-1}}T_hY=Y.
\]
The corresponding tensorator strings are fixed, not suppressed,
by the action-coherence equations
\[
\begin{aligned}
 T_{g^{-1}}(J^{gh}_{X,Y})J^{g^{-1}}_{T_{gh}X,T_{gh}Y}
 &=J^h_{X,Y},\\
 T_{h^{-1}}(J^h_{X,Y})J^{h^{-1}}_{T_hX,T_hY}
 &=J^1_{X,Y}=\id_{X\bo Y}.
\end{aligned}
\]
Whenever an action is moved across an evaluation cap, the actual
comparison is \eqref{eq:bcomparison},
$\ev_{T_kZ}(\id_{T_kZ}\bo b_{k,Z}) =j_k^{-1}T_k(\ev_Z)J^k_{Z,Z^\vee},$
with $(k,Z)=(g,X),(h,Y),(gh,W)$ as appropriate.  The unit maps
$j_k^{-1}$ are retained; their composites are governed by
$j_{kl}=T_k(j_l)j_k$ and $j_1=\id$.  The duality comparison
also has to respect the double-dual tensorator.  In the present
notation the required equation, with all endpoints displayed, is
\[
\begin{aligned}
&(J^k_{X,Y})^{\vee\vee}
 \mu_{T_kX,T_kY}
 (\chi^{(2)}_{k,X}\bo\chi^{(2)}_{k,Y})\\
&\qquad=\chi^{(2)}_{k,W}
 T_k(\mu_{X,Y})J^k_{X^{\vee\vee},Y^{\vee\vee}}.
\end{aligned}
\]
Both sides have source $T_k(X^{\vee\vee})\bo T_k(Y^{\vee\vee})$
and target $(T_kW)^{\vee\vee}$.  To verify the equation, take
right-dual mates: on the left use the evaluation of
$T_kX\bo T_kY$ transported by $(J^k_{X,Y})^{\vee\vee}$; on
the right use the evaluation of $T_kW$.  Applying
\eqref{eq:bcomparison} to $X$, $Y$ and $W$, and then the defining
evaluation property of $\mu_{X,Y}$, identifies both mate pairings
with the same $j_k^{-1}T_k(\widehat{\ev}_{X,Y})$ after the
monoidal comparison $J^k$.  The right-dual mate correspondence is
injective, so the displayed equation follows.  In particular, this
is a comparison of evaluation morphisms, not an identification of
$b_{k,Z}$ or $J^k$ with an identity.

We finish the mate comparison by passing to the equivariantization.
This argument avoids identifying two six-crossing expressions merely from
their endpoints.  We use the braided structure on $\mathcal E:=\cC^G$
of \cite[Theorem~8.24.3 and (8.108)]{EGNO}; throughout this
paragraph $G$ is finite.  The construction of
$\mathcal E$ requires a monoidal $G$-action and a crossed braiding, but
neither a pivotal structure nor a twist.  The forgetful functor
$U:\mathcal E\to\cC$ is faithful and strict monoidal on the chosen
underlying tensor products.  Right duals in $\mathcal E$ may be chosen
with underlying objects, evaluations and coevaluations equal to those
in $\cC$.  More precisely, the equivariant action on $M^\vee$ is
obtained from that on $M$ by the explicit composite
$T_k(M^\vee)\xrightarrow{b_{k,M}}(T_kM)^\vee\xrightarrow{(r_k^{-1})^\vee}M^\vee$;
its compatibility with evaluation
is exactly \eqref{eq:bcomparison}.  Hence the double-dual monoidal
comparison is carried by $U$ to $\mu_{M,N}$, with the canonical
duality identifications rather than literal identifications of
arbitrarily chosen right duals.

For an equivariant object $(M,r)$ write
$r_k:T_kM\xrightarrow{\sim}M$, so that
$r_{kl}=r_kT_k(r_l)$ under the strict underlying action.
The monoidal equivariant structure on $M\bo N$ is
$r^{M\bo N}_k=(r^M_k\bo r^N_k)(J^k_{M,N})^{-1}.$
For $X=M_g$, let $r_{g,X}:T_gX\to X$ denote the restriction
of $r_g$; its target is $M_g$, since $T_g\cC_g\subset\cC_g$.
Define an automorphism of $M$ by its homogeneous components,
$\widetilde\theta_M|_{M_g}:=r_{g,X}\theta_X:X\longrightarrow X.$
It is a natural automorphism of the identity functor on $\mathcal E$.
Indeed, if $X=M_g$ and $q=kgk^{-1}$, the crossed-twist covariance,
its naturality with respect to $r_k:T_kX\to M_q$, and the equivariant
cocycle relation give
\[
\begin{aligned}
 \widetilde\theta_M\,r_k|_{T_kX}
 &=r_{q,M}\,\theta_{M_q}\,r_k|_{T_kX}\\
 &=r_{q,M}\,T_q(r_k)\,\theta_{T_kX}\\
 &=r_{qk,M}\,T_k(\theta_X)
 =r_k\,T_k(\widetilde\theta_M|_X).
\end{aligned}
\]
Here $qk=kg$ and $\theta_{T_kX}=T_k(\theta_X)$; no
monoidal tensorator was suppressed in this one-object calculation.

We check the ordinary balancing identity rather than assume it.
For $(M,r),(N,s)\in\mathcal E$, take $X=M_g$, $Y=N_h$, and put
$q=ghg^{-1}$.  By \cite[(8.108)]{EGNO}, the ordinary braiding
$\widetilde c$ on $\mathcal E$ has restriction
$\widetilde c_{M,N}|_{X\bo Y} =(s_g|_{T_gY}\bo\id_X)c_{X,Y}.$
Its second application, using naturality of the crossed braiding
in the first variable, is the correctly ordered composite
\[
 (\widetilde c_{N,M}\widetilde c_{M,N})|_{X\bo Y}
 =(r_q|_{T_qX}\bo s_g|_{T_gY})
 c_{T_gY,X}c_{X,Y}.
\]
On the other hand, the tensor equivariant structure and
\eqref{eq:defbalance} yield
\[
\begin{aligned}
 \widetilde\theta_{M\bo N}|_{X\bo Y}
 &=r^{M\bo N}_{gh}\,\theta_{X\bo Y}\\
 &=(r_{gh,M}\bo s_{gh,N})
   (T_q\theta_X\bo T_g\theta_Y)
   c_{T_gY,X}c_{X,Y}.
\end{aligned}
\]
The displayed $J^{gh}_{X,Y}$ in $\theta_{X\bo Y}$ cancels
only against $(J^{gh}_{M,N})^{-1}$ in the actual tensor
equivariant structure, by naturality of $J^{gh}$; the
individual tensorators are not identities.
The preceding equivariance computation, applied respectively to
$r_q:T_qX\to M_{qgq^{-1}}$ and $s_g:T_gY\to N_q$, gives
\[
 \widetilde\theta_M\,r_q|_{T_qX}
   =r_{qg,M}T_q\theta_X=r_{gh,M}T_q\theta_X,
 \qquad
 \widetilde\theta_N\,s_g|_{T_gY}=s_{gh,N}T_g\theta_Y.
\]
Consequently, on every $X\bo Y$ and hence on $M\bo N$,
\[
 \widetilde\theta_{M\bo N}
 =(\widetilde\theta_M\bo\widetilde\theta_N)
 \widetilde c_{N,M}\widetilde c_{M,N}.
\]
The crossed unit condition gives
$\widetilde\theta_{\mathbf1}=\id_{\mathbf1}$.
Thus $\widetilde\theta$ is a twist in the precise sense of
\cite[Definition~8.10.1, (8.32)]{EGNO}; the alternative order
$\widetilde c_{N,M}\widetilde c_{M,N}
(\widetilde\theta_M\bo\widetilde\theta_N)$ agrees by naturality.

Let $\widetilde{\mathfrak u}_M:M\to M^{\vee\vee}$ denote the
ordinary Drinfeld isomorphism associated with $\widetilde c$ and
our chosen \emph{right} duals.  On $X=M_g$ its right-dual mate is
$\ev_X(r_{g^{-1},M}|_{T_{g^{-1}}X}\bo\id_{X^\vee}) c_{X^\vee,X}.$
The equivariant cocycle relation implies
$\,r_{g^{-1},M}|_{T_{g^{-1}}X}
 =T_{g^{-1}}(r_{g,X}^{-1})$; hence crossed-braiding
naturality identifies this pairing with
$\ev_Xc_{X^\vee,T_gX} (\id_{X^\vee}\bo r_{g,X}^{-1}).$
By the defining mate of $\mathfrak u_X$ and its uniqueness,
\[
 U(\widetilde{\mathfrak u}_M)|_X
 =\mathfrak u_Xr_{g,X}^{-1},\qquad
 U(\widetilde{\mathfrak u}_M\widetilde\theta_M)|_X
 =\mathfrak u_X\theta_X=\delta_X.               \tag{*}
\]
The componentwise identity holds also for direct sums, and is
compatible with the chosen double-dual identifications.

For clarity about the ordinary categorical input, apply
\cite[Proposition~8.9.3, (8.31)]{EGNO} to the \emph{reversed
tensor convention}: a right dual in $\mathcal E$ is a left dual in its
reversed monoidal category, whose braiding in the reversed order
is $\widetilde c_{N,M}$.  With the tensor factors returned to
our order, the proposition is the typed identity
\[
 \widetilde{\mathfrak u}_{M\bo N}
 \widetilde c_{N,M}\widetilde c_{M,N}
 =\mu_{M,N}
 (\widetilde{\mathfrak u}_M\bo
  \widetilde{\mathfrak u}_N).
\]
This equation involves the canonical double-right-dual comparison
$\mu_{M,N}$.  Combining it with the just-proved balancing
formula and the naturality of the twist yields, by actual
composition and cancellation of the invertible double braiding,
\[
 \widetilde{\mathfrak u}_{M\bo N}
 \widetilde\theta_{M\bo N}
 =\mu_{M,N}
  (\widetilde{\mathfrak u}_M\widetilde\theta_M
   \bo\widetilde{\mathfrak u}_N\widetilde\theta_N).
                                                        \tag{**}
\]
This is the ordinary Drinfeld--twist tensor identity, obtained
from (8.31) and (8.32).

It remains to transfer (**) to arbitrary homogeneous objects of
$\cC$, not just those carrying an equivariant structure.
For $X\in\cC$, the induction object
\[
 \operatorname{Ind}X:=\bigoplus_{k\in G}T_kX
\]
has the equivariant structure which reindexes
$T_lT_kX=T_{lk}X$ by the strict underlying action.  Its
underlying object contains $X$ as the $k=1$ direct summand.
Let $i_X:X\to U(\operatorname{Ind}X)$ and
$p_X:U(\operatorname{Ind}X)\to X$ be the canonical
inclusion and projection, with $p_Xi_X=\id_X$.
The maps $\delta_Z=\mathfrak u_Z\theta_Z$ are natural on
each homogeneous component by the previously established
naturality of $\mathfrak u$ and $\theta$; on general finite
$G$-graded objects they extend by direct sum.  Therefore the
identities (*) and (**) for
$(\operatorname{Ind}X,\operatorname{Ind}Y)$, composed with
$i_X\bo i_Y$ and with the corresponding double-dual comparison,
give by naturality
$\delta_{X\bo Y}=\mu_{X,Y}(\delta_X\bo\delta_Y).$
Explicitly, the equality after precomposition with
$i_X\bo i_Y$ is detected by the split projection
$p_X\bo p_Y$, and the naturality of $\mu$ and $\delta$
identifies its restriction with the displayed equality on
$X\bo Y$.  The split inclusion and projection make the argument valid for
arbitrary $X,Y$ in $\cC$.
Finally, by the already established \eqref{eq:defbalance},
$\theta_W=A_{X,Y}B_{X,Y}$, and hence
$\mathfrak u_W A_{X,Y}B_{X,Y} =\delta_W=\mu_{X,Y}(\delta_X\bo\delta_Y).$
This proves \eqref{eq:drinfelddefect} independently of
\eqref{eq:pivotalmonoidal}; the latter is its immediate
consequence.  Applying the injective right-dual mate correspondence
and precomposing with $d_{X,Y}^{-1}\bo\id_W$ also proves the
previously displayed transported-mate equality
$\widetilde{\mathcal L}_{X,Y}=\widetilde{\mathcal R}_{X,Y}$.
Thus the original six-crossing comparison follows from the
equivariantization construction and the ordinary braided Drinfeld
identity, using the crossed-braiding and monoidal-action data of
\cite[Theorem~1.1, Proposition~8.5, Proposition~9.6]{Crossed}.

Finally, fix $X\in\cC_g$.  The two maps in
\eqref{eq:pivotaldual} have the same source and target.  Indeed,
$X^\vee\in\cC_{g^{-1}}$, and
\[
 \delta_{X^\vee}=\mathfrak u_{X^\vee}\theta_{X^\vee}:
 X^\vee\longrightarrow X^{\vee\vee\vee},
 \qquad
 (\delta_X^{-1})^\vee:
 X^\vee\longrightarrow X^{\vee\vee\vee}.
\]
More explicitly, $\theta_X^{-1}:T_gX\to X$ and
$\mathfrak u_X^{-1}:X^{\vee\vee}\to T_gX$, so the contravariance
of the chosen right-dual functor gives
\[
 (\delta_X^{-1})^\vee
 =(\mathfrak u_X^{-1})^\vee(\theta_X^{-1})^\vee,
 \quad
 X^\vee\xrightarrow{(\theta_X^{-1})^\vee}(T_gX)^\vee
 \xrightarrow{(\mathfrak u_X^{-1})^\vee}X^{\vee\vee\vee}.
\]
The categorical dual-twist identity
\eqref{eq:categoricaldualtwistaxiom} is an equality in
$\operatorname{Hom}(X^\vee,T_{g^{-1}}X^\vee)$, with the action--dual
comparison $b_{g,X}$ retained.  It concerns the twist itself and
is not an identification of $\theta_X^\vee$ with
$(\theta_X^{-1})^\vee$.  The required comparison of the complete
maps $\delta$ follows from the already established monoidality,
by the uniqueness of right-dual mates.

Write $D^2Z=Z^{\vee\vee}$ and let
$\nu:\mathbf1\xrightarrow{\sim}D^2\mathbf1$ be the canonical
unit comparison of the double-right-dual functor.  Its monoidal
structure is the canonical $\mu_{U,V}$ occurring in
\eqref{eq:pivotalmonoidal}.  With the fixed right-dual
representatives and their canonical comparison isomorphisms, the
defining evaluation-preservation identity for this strong monoidal
functor is
\[
 \ev_{X^{\vee\vee}}
 =\nu^{-1}(\ev_X)^{\vee\vee}\mu_{X,X^\vee}:
 X^{\vee\vee}\bo X^{\vee\vee\vee}\longrightarrow\mathbf1.
\]
Here $D^2(X^\vee)=X^{\vee\vee\vee}$ is the chosen right dual of
$D^2X=X^{\vee\vee}$.  This evaluation identity is part of the
canonical monoidal structure on $D^2$, not a consequence of the
desired pivotal-duality formula; it follows directly by transporting
$\ev_X$ through the two right-dual adjunctions.  The associators,
unitors and canonical identifications of the iterated right duals
are understood in this identity.

Naturality of $\delta$ for the evaluation morphism
$\ev_X:X\bo X^\vee\to\mathbf1$, together with
$\delta_{\mathbf1}=\nu$ and the already proved
\eqref{eq:pivotalmonoidal}, gives two morphisms
$X\bo X^\vee\to\mathbf1$ which are equal:
\[
\begin{aligned}
 \ev_X
 &=\nu^{-1}(\ev_X)^{\vee\vee}\delta_{X\bo X^\vee}\\
 &=\nu^{-1}(\ev_X)^{\vee\vee}\mu_{X,X^\vee}
       (\delta_X\bo\delta_{X^\vee})\\
 &=\ev_{X^{\vee\vee}}
       (\delta_X\bo\delta_{X^\vee}).
\end{aligned}
\]
On the other hand, right-dual naturality for the inverse isomorphism
$\delta_X^{-1}:X^{\vee\vee}\to X$ states precisely that
\[
 \ev_{X^{\vee\vee}}
   (\id_{X^{\vee\vee}}\bo(\delta_X^{-1})^\vee)
 =\ev_X(\delta_X^{-1}\bo\id_{X^\vee})
 :X^{\vee\vee}\bo X^\vee\longrightarrow\mathbf1.
\]
Precomposing with $\delta_X\bo\id_{X^\vee}$ therefore yields
$\ev_X=\ev_{X^{\vee\vee}} (\delta_X\bo(\delta_X^{-1})^\vee).$
Consequently the two maps
$\delta_{X^\vee},(\delta_X^{-1})^\vee:
 X^\vee\to X^{\vee\vee\vee}$ have identical right-dual
pairings after precomposition with the isomorphism $\delta_X$.
Since $X^{\vee\vee\vee}$ is the right dual of $X^{\vee\vee}$,
the right-dual mate correspondence is injective.  Cancelling
$\delta_X$ in this correspondence gives
$\delta_{X^\vee}=(\delta_X^{-1})^\vee,$
as required.  This argument applies to every homogeneous object,
including zero and finite direct sums, and uses neither a prior
pivotal-duality assertion nor any ribbon or modularity theorem.
It is the evaluation proof of the general duality property of a
monoidal pivotal transformation
\cite[Definition~4.7.7 and Exercise~4.7.9]{EGNO}.
The relation between the ordinary ribbon twist and equality of left
and right traces is given in \cite[Proposition~8.10.12]{EGNO}.
\end{proof}

\subsection{Verification of the crossed ribbon axiom}\label{subsec:ribbonverification}

\begin{theorem}\label{thm:ribbon}
Let $V$ be strongly rational and $G\leq\Aut(V)$ finite, and let
$\cC_G$ be Xu's faithfully $G$-graded $G$-crossed braided fusion
category \cite[Theorem~1.1]{Crossed}, with action and crossed-commutor
normalizations \cite[Proposition~8.5, Proposition~9.6]{Crossed} and
the pairing and rigidity data
\cite[Lemma~10.1, Proposition~10.3]{Crossed}. Then the conformal twist
equips $\cC_G$ with a $G$-crossed ribbon
fusion structure in Kirillov's sense.
\end{theorem}
\begin{proof}
We compare the structure maps with the four conditions in
\cite[Lemma~2.3]{Kirillov}, using its definition of a
$G$-equivariant fusion category in Section~2.  The underlying faithfully graded rigid fusion category, strong
monoidal $G$-action, and crossed braiding are provided by Xu's
construction \cite[Theorem~1.1, Proposition~8.5,
Proposition~9.6, Proposition~10.3]{Crossed}.
The action has strict underlying functorial multiplication
$T_kT_l=T_{kl}$ and $T_1=\id$, but its tensor and unit comparisons
$J^k_{A,B}:T_kA\bo T_kB\to T_k(A\bo B)$ and
$j_k:\mathbf1\to T_k\mathbf1$ are retained throughout.

For the unit condition, the vacuum module $V=\mathbf1$ lies in
$\cC_1$, and Lemma~\ref{lem:theta} gives
$\theta_{\mathbf1}=e^{2\pi iL_V(0)}=\id_{\mathbf1}=j_1$.
In particular, this uses $j_1=\id$, not an identification of
$j_k$ with the identity for arbitrary $k$.  For $X\in\cC_g$,
$\theta_X:X\to T_gX$ is an invertible natural morphism by that
lemma.  If $k\in G$, put $g'=kgk^{-1}$.  Both
\[
 \theta_{T_kX}:T_kX\longrightarrow T_{g'}T_kX,
 \qquad T_k(\theta_X):T_kX\longrightarrow T_kT_gX
\]
have target $T_{kg}X$, since $g'k=kg$, and they are equal by
Lemma~\ref{lem:theta}.  This establishes the action covariance
condition without discarding any monoidal structure of $T_k$.

For balancing, let $Y\in\cC_h$ and $q=ghg^{-1}$.
The ordered crossed braidings have types
\[
 \begin{gathered}
 c_{X,Y}:X\bo Y\longrightarrow T_gY\bo X,\\
 c_{T_gY,X}:T_gY\bo X\longrightarrow T_qX\bo T_gY.
 \end{gathered}
\]
Here $T_gY\in\cC_q$, and the second target has degree
$qg=gh$.  The two twisted input morphisms are
$T_q(\theta_X):T_qX\to T_qT_gX=T_{gh}X$ and
$T_g(\theta_Y):T_gY\to T_gT_hY=T_{gh}Y$.
Thus the following is an equality in
$\Hom(X\bo Y,T_{gh}(X\bo Y))$ by
Theorem~\ref{thm:balance}:
\[
 \theta_{X\bo Y}
 =J^{gh}_{X,Y}\bigl(T_q(\theta_X)\bo T_g(\theta_Y)\bigr)
       c_{T_gY,X}c_{X,Y}.
\]
This is Kirillov's second condition: his
$R_{{}^gY,X}R_{X,Y}$ is the indicated \emph{ordered} pair of crossed
braidings, and the omitted canonical identification from the
separate output twists to the twist of $X\bo Y$ is precisely the
strong-monoidal comparison $J^{gh}_{X,Y}$ together with
$T_qT_g=T_{gh}=T_gT_h$.  The group factors remain in their prescribed order, and both arrows
are the crossed braidings displayed above.

The right-dual condition has the same content as Kirillov's third
condition after restoring its suppressed action--dual comparison.
Indeed, for $X\in\cC_g$, its chosen right dual lies in
$\cC_{g^{-1}}$, and Lemma~\ref{lem:dualactioncomparison} supplies
the natural isomorphism
$b_{g,X}:T_g(X^\vee)\xrightarrow{\sim}(T_gX)^\vee$.
Consequently the two factors in
$\theta_{X^\vee} =T_{g^{-1}}(\theta_X^\vee)\, T_{g^{-1}}(b_{g,X})$
have respective types
\[
 \begin{gathered}
 T_{g^{-1}}(b_{g,X}):
 X^\vee=T_{g^{-1}}T_g(X^\vee)
 \longrightarrow T_{g^{-1}}((T_gX)^\vee),\\
 T_{g^{-1}}(\theta_X^\vee):
 T_{g^{-1}}((T_gX)^\vee)
 \longrightarrow T_{g^{-1}}(X^\vee).
 \end{gathered}
\]
Thus both sides lie in
$\Hom(X^\vee,T_{g^{-1}}X^\vee)$; the equality is
\eqref{eq:categoricaldualtwistaxiom}, proved in
Proposition~\ref{prop:dual}.  Under the canonical identification
$(T_gX)^\vee\simeq T_g(X^\vee)$, this is exactly
$\theta_{X^\vee}=T_{g^{-1}}(\theta_X^\vee)$ in
\cite[Lemma~2.3]{Kirillov}, with the contravariant dual
$\theta_X^\vee:(T_gX)^\vee\to X^\vee$ in its correct direction.

There is also the pivotal datum required by Kirillov's definition
of a $G$-equivariant fusion category.  With the same chosen right
duals, Proposition~\ref{prop:pivotal} constructs
$\delta_X=\mathfrak u_X\theta_X:X\xrightarrow{\sim}X^{\vee\vee}$.
It proves naturality and the unit normalization, together with
\[
 \begin{aligned}
 \delta_{X\bo Y}&=\mu_{X,Y}(\delta_X\bo\delta_Y),\\
 \delta_{T_kX}&=\chi^{(2)}_{k,X}T_k(\delta_X),\\
 \delta_{X^\vee}&=(\delta_X^{-1})^\vee.
 \end{aligned}
\]
Here $\mu_{X,Y}$ is the canonical monoidal structure of the double
right dual, while
$\chi^{(2)}_{k,X}=(b_{k,X}^{-1})^\vee b_{k,X^\vee}$
is the action--double-dual comparison; neither is being treated as
an identity map.  These identities give the monoidal, equivariant
and dual-compatible double-dual structure required in
\cite[Section~2 and Lemma~2.3]{Kirillov}.  The crossed twist
identities above are established independently of any assertion
about the ribbon structure of the equivariantization.

All these identities are natural and additive in homogeneous
objects.  On an arbitrary object $M=\bigoplus_g M_g$ of $\cC_G$,
the crossed twist is the componentwise family
$\bigoplus_g\theta_{M_g}:\bigoplus_gM_g\to
\bigoplus_gT_gM_g$; it is not identified with the action of a
single group element on $M$.  Since tensor product, duality,
the action and crossed braidings preserve finite direct sums, the
conditions just checked extend componentwise to all objects.
This verifies Kirillov's four twist conditions and the required
pivotal compatibility, proving the assertion.
\end{proof}

\section{Modularity and equivariantization}

\subsection{Equivariantization and nondegeneracy}\label{subsec:equivariantization}
\begin{definition}\label{def:equivariantization}
The \emph{equivariantization} $\cB=\cC_G^G$ has objects
$A=(\underline A,(u_k^A)_{k\in G})$, where
$u_k^A:T_k\underline A\xrightarrow{\sim}\underline A$,
$u_1^A=\id$, and $u_{kl}^A=u_k^A T_k(u_l^A)$.
A morphism $f:A\to B$ is a morphism
$f:\underline A\to\underline B$ satisfying
$f u_k^A=u_k^B T_k(f)$ for all $k\in G$.
Writing $\underline A=\bigoplus_{g\in G}A_g$, the grading
and the action imply
$u_k^A:T_kA_g\xrightarrow{\sim}A_{kgk^{-1}}$
on each homogeneous summand.  In particular, the restriction
$u_g^A:T_gA_g\to A_g$ is defined, and no equivariant object
is assumed homogeneous.

The tensor product has underlying object $\underline A\bo\underline B$
and equivariant structure
\[
 u_k^{A\bo B}=(u_k^A\bo u_k^B)
   (J^k_{\underline A,\underline B})^{-1}:
 T_k(\underline A\bo\underline B)
 \longrightarrow\underline A\bo\underline B.
\]
Indeed, the action coherence \eqref{eq:actioncoherence},
naturality of $J^k$, and the two cocycle identities give
\[
\begin{aligned}
 u_k^{A\bo B}T_k(u_l^{A\bo B})
 &= (u_k^A\bo u_k^B)
    (T_k(u_l^A)\bo T_k(u_l^B))
    (J^{kl}_{\underline A,\underline B})^{-1}\\
 &=u_{kl}^{A\bo B}.
\end{aligned}
\]
The identity at $k=1$ follows from $J^1=\id$.
The tensor unit is $(\mathbf1,(j_k^{-1})_k)$; its
cocycle is $j_{kl}^{-1}=j_k^{-1}T_k(j_l^{-1})$.
The monoidal action axioms for $J^k,j_k$ make the original
associator and unitors equivariant morphisms.

For homogeneous $X=A_g$ and $Y=B_h$, put $q=ghg^{-1}$.
The braiding is defined componentwise by
\[
 \widetilde c_{A,B}|_{X\bo Y}
  =(u_g^B|_{T_gY}\bo\id_X)c_{X,Y}:
  X\bo Y\longrightarrow B_q\bo X
  \subseteq\underline B\bo\underline A.
\]
For a nonhomogeneous second argument it is the direct sum of
these arrows; the target is not a fixed $Y\bo X$ unless the
grading happens to be preserved by conjugation.  The formula
is extended by finite direct sums.  Naturality,
invertibility, equivariance and both ordinary hexagons follow
from the crossed-braiding and action axioms
\eqref{eq:crossedequiv}--\eqref{eq:crossedhextwo} with their
unit constraints and $J^k$ retained; this is the braided
equivariantization construction
\cite[Definition~4.41]{DGNO}\cite[Theorem~8.24.3, (8.108)]{EGNO}.

Choose the right dual $M^\vee$ of $M=\underline A$, with
evaluation $\ev_M:M\bo M^\vee\to\mathbf1$ and coevaluation
$\coev_M:\mathbf1\to M^\vee\bo M$.  Its equivariant action is
\[
 \widehat u_k=( (u_k^A)^{-1})^\vee b_{k,M}:
 T_k(M^\vee)\xrightarrow{\sim}M^\vee.
\]
Here $b_{k,M}:T_k(M^\vee)\to(T_kM)^\vee$ is the
canonical evaluation comparison of
Lemma~\ref{lem:dualactioncomparison}; the displayed composite
uses $( (u_k^A)^{-1})^\vee:(T_kM)^\vee\to M^\vee$.
Its cocycle follows from the contravariant naturality of $b$
and $b_{kl,M}=b_{k,T_lM}T_k(b_{l,M})$:
\[
\begin{aligned}
 \widehat u_k T_k(\widehat u_l)
 &=( (u_k^A)^{-1})^\vee b_{k,M}
   T_k(( (u_l^A)^{-1})^\vee)T_k(b_{l,M})\\
 &=( (u_k^A)^{-1})^\vee
    (T_k((u_l^A)^{-1}))^\vee b_{kl,M}
   =\widehat u_{kl}.
\end{aligned}
\]
Also $\widehat u_1=\id$.  More precisely, the defining
evaluation comparison and right-dual naturality yield
\[
\begin{aligned}
 &\ev_M(u_k^A\bo\widehat u_k)\\
 &\quad=\ev_{T_kM}(\id_{T_kM}\bo b_{k,M})
  =j_k^{-1}T_k(\ev_M)J^k_{M,M^\vee}.
\end{aligned}
\]
Thus $\ev_M$ is equivariant, since the unit action is $j_k^{-1}$.
The corresponding transported coevaluation satisfies
$(\widehat u_k\bo u_k^A)(J^k_{M^\vee,M})^{-1} T_k(\coev_M)j_k=\coev_M.$
To see this identity, apply the two right-dual triangle
identities to the transported evaluation above: both sides
are coevaluations for the same chosen evaluation and hence
coincide by right-dual uniqueness.  Consequently evaluation,
coevaluation and both triangle identities are morphisms in $\cB$.
These constructions identify $\cB$ with Kirillov's orbifold
category $\cC/G$ in \cite[Definition~3.1 and Theorem~3.9]{Kirillov},
with the monoidal and dual comparisons suppressed in that source
made explicit here.
\end{definition}
For $A\in\cB$ with homogeneous decomposition
$\underline A=\bigoplus_{g\in G}A_g$, define
\begin{equation}\label{eq:thetab}
 \Theta_A=\bigoplus_{g\in G}(u_g^A|_{T_gA_g})\,\theta_{A_g}.
\end{equation}
Because $g$ fixes its own conjugacy class element, each summand
in \eqref{eq:thetab} is an automorphism of $A_g$.
These automorphisms are natural in equivariant morphisms.
For $X=A_g$, $k\in G$, and $q=kgk^{-1}$, the twist covariance,
twist naturality for $u_k^A$, and the cocycle give, as maps
$T_kX\to A_q$,
\[
\begin{aligned}
 \Theta_Au_k^A|_{T_kX}
 &=u_q^A\theta_{A_q}u_k^A
  =u_q^A T_q(u_k^A)\theta_{T_kX}\\
 &=u_{qk}^A T_k(\theta_X)
  =u_k^A T_k(u_g^A\theta_X)
  =u_k^A T_k(\Theta_A|_X).
\end{aligned}
\]
Hence $\Theta_A$ is an equivariant natural isomorphism.
Its unit value is the identity.  For $X=A_g$, $Y=B_h$, put $q=ghg^{-1}$.  The crossed
balancing identity \eqref{eq:defbalance} and the inverse
tensorator in the equivariant tensor product give
\[
\begin{aligned}
 \Theta_{A\bo B}|_{X\bo Y}
 &= (u_{gh}^A\bo u_{gh}^B)
    (J^{gh}_{X,Y})^{-1}\theta_{X\bo Y}\\
 &= (u_{gh}^A T_q\theta_X\bo u_{gh}^B T_g\theta_Y)
    c_{T_gY,X}c_{X,Y}.
\end{aligned}
\]
Thus $(J^{gh}_{X,Y})^{-1}$ and $J^{gh}_{X,Y}$ cancel
only after their domains are correctly matched.  On the other
hand, using the degrees $g,q$ of the successive first inputs
of the two ordinary braidings and crossed-braiding naturality
for $u_g^B:T_gY\to B_q$, we obtain
\[
\begin{aligned}
 &\widetilde c_{B,A}\widetilde c_{A,B}|_{X\bo Y}\\
 &\quad=(u_q^A\bo\id_{B_q})c_{B_q,X}
       (u_g^B\bo\id_X)c_{X,Y}\\
 &\quad=(u_q^A\bo u_g^B)c_{T_gY,X}c_{X,Y}.
\end{aligned}
\]
The already established equivariance of $\Theta$, together
with $u_q^A T_q(u_g^A)=u_{qg}^A$ and
$u_g^B T_g(u_h^B)=u_{gh}^B$, gives the following identities
with the indicated sources and targets:
\[
 \Theta_Au_q^A|_{T_qX}=u_{gh}^A T_q\theta_X,
 \qquad
 \Theta_Bu_g^B|_{T_gY}=u_{gh}^B T_g\theta_Y,
\]
where $qg=gh$.  Substituting these two identities in the
previous crossing computation yields
$\Theta_{A\bo B} =(\Theta_A\bo\Theta_B) \widetilde c_{B,A}\widetilde c_{A,B}.$
The proof uses neither a prior ribbon structure on $\cB$ nor
any nondegeneracy statement.

It remains to check the right-dual twist identity independently.
Let $X=A_g$ and write $u_g^X:T_gX\to X$ and
$u_{g^{-1}}^X:T_{g^{-1}}X\to X$ for the corresponding
restrictions of the equivariant structure.  The right-dual
summand $X^\vee$ has degree $g^{-1}$.  By the crossed
dual-twist identity \eqref{eq:categoricaldualtwistaxiom},
contravariant naturality of $b$, and its group coherence,
the following are equal morphisms $X^\vee\to X^\vee$:
\[
\begin{aligned}
 \Theta_{A^\vee}|_{X^\vee}
 &=\widehat u_{g^{-1}}\theta_{X^\vee}\\
 &=((u_{g^{-1}}^X)^{-1})^\vee b_{g^{-1},X}
   T_{g^{-1}}(\theta_X^\vee)T_{g^{-1}}(b_{g,X})\\
 &=((u_{g^{-1}}^X)^{-1})^\vee
   (T_{g^{-1}}\theta_X)^\vee
   b_{g^{-1},T_gX}T_{g^{-1}}(b_{g,X})\\
 &=((u_{g^{-1}}^X)^{-1})^\vee
   (T_{g^{-1}}\theta_X)^\vee\\
 &=\bigl((T_{g^{-1}}\theta_X)(u_{g^{-1}}^X)^{-1}\bigr)^\vee
   =(u_g^X\theta_X)^\vee.
\end{aligned}
\]
The third equality is the naturality relation
$b_{g^{-1},X}T_{g^{-1}}(\theta_X^\vee)
=(T_{g^{-1}}\theta_X)^\vee b_{g^{-1},T_gX}$;
the fourth is the actual coherence
$b_{g^{-1},T_gX}T_{g^{-1}}(b_{g,X})=b_{1,X}=\id_{X^\vee}$.
For the final equality, twist naturality for
$u_{g^{-1}}^X:T_{g^{-1}}X\to X$, followed by $u_g^X$,
gives
\[
 u_g^X\theta_Xu_{g^{-1}}^X
 =u_g^XT_g(u_{g^{-1}}^X)\theta_{T_{g^{-1}}X}
 =T_{g^{-1}}(\theta_X),
\]
where $u_g^XT_g(u_{g^{-1}}^X)=\id_X$ and
$\theta_{T_{g^{-1}}X}=T_{g^{-1}}(\theta_X)$.
The identity extends to arbitrary equivariant objects by
finite direct sums and canonical chosen-dual comparisons.
We have thereby proved
$\Theta_{A^\vee}=(\Theta_A)^\vee$.

Thus $\cB$ is a braided ribbon fusion category, with the braiding,
unit twist, balancing and right-dual twist compatibility verified
above. Kirillov's \cite[Theorem~3.9]{Kirillov} produces this twist
and braiding, and \cite[Proposition~8.10.12]{EGNO} supplies the
corresponding spherical pivotal structure.

\begin{remark}\label{rem:literatureconventions}
In \cite[Definition~4.41]{DGNO} the data are phrased as a braided
$G$-crossed fusion category and the group action is written
$g(X)$; in our notation it is $T_gX$.  Their crossed-braiding
map has the same source and target as $c_{X,Y}$, and their
$G$-grading is our external direct sum.  The identities referred
to as action compatibility and crossed hexagons in that source
are established by Xu's construction
\cite[Theorem~1.1, Proposition~8.5, Proposition~9.6]{Crossed}; it
fixes the positive-half-turn normalization used here. Proposition~4.56(ii)
of \cite{DGNO} gives the nondegeneracy criterion in terms of faithful
grading and nondegeneracy of the identity component. Kirillov's
Theorem~3.9 supplies the braided ribbon structure on $\cC/G$, and
his Theorem~10.5 identifies modularity with invertibility of the
extended $S$-operator.
\end{remark}

\begin{proposition}\label{prop:nondeg}
Let $V$ be strongly rational and $G\leq\Aut(V)$ finite, and let
$\cC_G$ be the faithfully $G$-graded $G$-crossed braided fusion
category constructed in \cite[Theorem~1.1]{Crossed}. Then its equivariantization $\cB$ is a nondegenerate braided fusion
category.
\end{proposition}
\begin{proof}
The strongly rational hypotheses imply the CFT-type and
self-contragredience conditions, complete reducibility, and
$C_2$-cofiniteness required by Huang's rigidity and modularity
theorem \cite{Huang}. Hence $\Rep(V)$ is a nondegenerate braided
fusion category. By Theorem~\ref{thm:neutralanalytic} and Xu's binary
$P(z)$-representing, product--iterate, unit, positive-half-turn
commutor, and parallel-transport identities
\cite[Theorem~4.3, Theorem~8.3, Proposition~8.5,
Proposition~9.6]{Crossed}, Proposition~\ref{prop:neutralcomparison}
gives a braided monoidal equivalence
$F:\cC_1\xrightarrow{\sim}\Rep(V)$. Such an equivalence preserves
double braidings and identifies the M\"uger centers, so
$\mathcal Z_2(\cC_1)\simeq_{\rm br}\mathcal Z_2(\Rep(V))
\simeq_{\rm br}\mathrm{Vec}$.

Xu's Theorem~1.1 supplies a braided $G$-crossed fusion category
$\cC_G$ in the sense of \cite[Definition~4.41]{DGNO}, for the finite
group $G$, with a faithful grading
$\cC_G=\bigoplus_{g\in G}\cC_g$. In particular, $\cC_g\ne0$
for every $g\in G$ is part of the theorem's faithful-grading
conclusion. The equivariantization constructed above is the braided fusion
category $\cB=\cC_G^G$. By \cite[Proposition~4.56(ii)]{DGNO},
this category is nondegenerate if and only if its neutral
component is nondegenerate and its $G$-grading is faithful.
Both conditions hold, and therefore
$\mathcal Z_2(\cB)\simeq_{\rm br}\mathrm{Vec}$.
\end{proof}

\begin{proof}[Proof of Theorem~\ref{thm:main}]
The first assertion is Theorem~\ref{thm:ribbon}.
For the second, the equivariantization $\cB=\cC_G^G$ is a fusion
category with the braiding and right duals constructed in
Definition~\ref{def:equivariantization}.  Its ordinary twist
\eqref{eq:thetab} satisfies balancing and right-dual compatibility,
so $\cB$ is a ribbon fusion category; this is the structure of
\cite[Theorem~3.9]{Kirillov}, with the action tensorators and
duality comparisons made explicit above.  Proposition~\ref{prop:nondeg}
gives $\mathcal Z_2(\cB)\simeq\mathrm{Vec}$.  Thus $\cB$ is a
nondegenerate ribbon fusion category, hence a modular tensor
category in the convention of Section~2.  For a ribbon fusion
category, this is equivalent to nonsingularity of its ordinary
$S$-matrix, the premodular criterion of \cite{DGNO}; thus it is
also modular in the ordinary sense used by Kirillov.

We apply \cite[Theorem~10.5]{Kirillov} to the $G$-equivariant fusion
category $\cC_G$ underlying Theorem~\ref{thm:ribbon}.  Its finite
semisimple rigid structure, crossed braiding, monoidal $G$-action,
and equivariant double-dual identification
$\delta:\id\xrightarrow{\sim}(-)^{\vee\vee}$ supply the data
required in \cite[Definition~2.1]{Kirillov}; the equivariantization
$\cB$ agrees with Kirillov's orbifold category $\cC_G/G$.
In the notation of \cite[Section~7]{Kirillov}, the corresponding
extended Verlinde space has components
\[
 \widetilde{\mathcal V}_{k,g}
 =\bigoplus_{[X]\in\Irr(\cC_g)}
     \Hom_{\cC_G}(X,T_kX).
\]
Since $T_kX\in\cC_{kgk^{-1}}$, a summand vanishes unless
$kg=gk$; even for commuting $k,g$, it vanishes unless
$T_kX\simeq X$.  Thus the nonzero components have precisely
Kirillov's commuting-pair grading.  With his unnormalized,
categorical-trace definition \cite[(9.3)]{Kirillov}, the operator
acts as
$\widetilde S:\widetilde{\mathcal V}_{k,g}\to
\widetilde{\mathcal V}_{g^{-1},k}$: his $R_k$ is our $T_k$,
and the canonical action and right-dual comparisons are those
fixed in Sections~2 and~5.  This is the extended operator specified
in Definition~\ref{def:crossedmodular}, not the ordinary
$S$-matrix of $\cB$ or a genus-one VOA trace transformation.
Kirillov's Theorem~10.5 states that his extended operator is
invertible if and only if this orbifold category is modular.
Applying its implication from the modularity of $\cB$ proves
invertibility of $\widetilde S$ on the entire extended Verlinde
space, and hence the asserted $G$-crossed modularity.
Finally, we compute the Frobenius--Perron dimensions.  By
Proposition~\ref{prop:neutralcomparison}, there is a braided strong
monoidal equivalence
\[
 F:(\cC_1,\bo_D,c^D,a^D)\xrightarrow{\sim}
   (\Rep(V),\bo_H,c^H,a^H).
\]
A tensor equivalence induces an isomorphism of Grothendieck rings,
identifies their simple-object classes, and carries the unique
positive Frobenius--Perron dimension character to the corresponding
character.  Consequently it preserves Frobenius--Perron dimensions.
Using also the convention from Section~2 that the dimension character
of $\cC_G$ restricts to the intrinsic one on its neutral fusion
subcategory, we obtain
\[
 \FPdim(\cC_1)
 =\sum_{[X]\in\Irr(\cC_1)}\FPdim_{\cC_G}(X)^2
 =\sum_{[M]\in\Irr(\Rep(V))}\FPdim_{\Rep(V)}(M)^2
 =\FPdim(\Rep(V)).
\]

The grading of $\cC_G$ is faithful by Xu's Theorem~1.1, and $G$
is finite.  By \cite[Proposition~8.20]{ENO}, in a fusion category faithfully
graded by a finite group, the Frobenius--Perron dimensions of all
homogeneous components are equal; these component dimensions are
computed using the FP-dimension character of the whole fusion
category.  Hence, for every $g\in G$,
$\FPdim(\cC_g)=\FPdim(\cC_1)=\FPdim(\Rep(V)).$
Here $\FPdim(\cC_g)$ is the sum of squared ambient
Frobenius--Perron dimensions of the simple objects in that sector,
as defined in Section~2.

Because the grading is an external direct-sum decomposition,
$\Hom_{\cC_G}(X_g,Y_h)=0$ for $g\ne h$, and every object is a finite
direct sum of homogeneous objects, each simple-object class belongs
to exactly one sector.  Thus
$\Irr(\cC_G)=\bigsqcup_{g\in G}\Irr(\cC_g)$, where the union retains
the external sector labels.  By the definition of the component
dimensions in Section~2,
\[
\begin{aligned}
 \FPdim(\cC_G)
 &=\sum_{[X]\in\Irr(\cC_G)}\FPdim_{\cC_G}(X)^2\\
 &=\sum_{g\in G}\FPdim(\cC_g)
 =|G|\FPdim(\cC_1)
 =|G|\FPdim(\Rep(V)).
\end{aligned}
\]

Finally, $T$ is an action of the finite group $G$ on the fusion
category $\cC_G$ by strong monoidal autoequivalences, with tensorators
$J^k$ and unit comparisons $j_k$ as fixed in Section~2.  The
underlying functors are strict, but these monoidal comparisons are
part of the action data.  The equivariantization formula in
\cite[Proposition~4.26]{DGNO} gives, for $\cB=\cC_G^G$,
\[
\begin{aligned}
 \FPdim(\cB)
 &=|G|\FPdim(\cC_G)\\
 &=|G|^2\FPdim(\Rep(V)).
\end{aligned}
\]
This completes the proof.
\end{proof}

\subsection{Dimensions and simple-object counts}\label{subsec:dimensionscounts}

\begin{corollary}\label{cor:E}
There is a canonical Tannakian subcategory $\mathcal E\simeq\Rep(G)$ of $\cB$. Writing $\mathcal E'$ for its M\"uger centralizer (the full subcategory of objects having trivial double braiding with every object of $\mathcal E$), one has
$\mathcal E'\simeq_{\rm br}\cC_1^G, \FPdim(\mathcal E')=|G|\FPdim(\Rep(V)).$
The twist on $\mathcal E$ is trivial.
\end{corollary}
\begin{proof}
For a finite-dimensional representation $\rho:G\to\operatorname{GL}(W)$, let
$I(W)=W\otimes\mathbf1$ denote the corresponding finite direct sum of copies of the tensor unit. Since $T_k$ is $\mathbb C$-linear and additive, there is a canonical direct-sum identification
$\iota_{k,W}:T_k I(W)\xrightarrow{\sim}W\otimes T_k\mathbf1.$
Give $I(W)$ the equivariant structure
\[
 u_k^{I(W)}=(\rho(k)\otimes j_k^{-1})\,\iota_{k,W}:
 T_k I(W)\longrightarrow I(W).
\]
This has $u_1^{I(W)}=\id$ and satisfies the cocycle: under the displayed identifications, the equality
$u_k^{I(W)}T_k(u_l^{I(W)})=u_{kl}^{I(W)}$ is exactly
$\rho(k)\rho(l)=\rho(kl)$ together with
$j_{kl}^{-1}=j_k^{-1}T_k(j_l^{-1})$. Thus $W\mapsto (I(W),(u_k^{I(W)}))$ defines a functor
$\mathcal I:\Rep(G)\to\cB$.

For representations $W,U$, the canonical tensor constraint
$\phi_{W,U}:I(W)\bo I(U)\xrightarrow{\sim}I(W\otimes U)$
is induced by the associators and unitors. It is equivariant because
\[
 \phi_{W,U}(u_k^{I(W)}\bo u_k^{I(U)})(J^k_{I(W),I(U)})^{-1}
 =u_k^{I(W\otimes U)}T_k(\phi_{W,U});
\]
this is precisely the monoidal unit coherence for $J^k$ and $j_k$, with
$\rho_{W\otimes U}(k)=\rho_W(k)\otimes\rho_U(k)$. In particular, the inverse tensorator in the equivariant tensor product is retained. The trivial representation maps to the equivariant tensor unit $(\mathbf1,(j_k^{-1})_k)$, so $\mathcal I$ is strong monoidal. Moreover,
\[
 \Hom_{\cB}(I(W),I(U))
 \cong\{f\in\Hom_{\mathbb C}(W,U):
 f\rho_W(k)=\rho_U(k)f\text{ for all }k\in G\}
 =\Hom_G(W,U),
\]
where we used $\operatorname{End}_{\cC}(\mathbf1)=\mathbb C$ and the equivariant-morphism condition. Hence $\mathcal I$ is fully faithful.

Both $I(W)$ and $I(U)$ are supported in degree $1$. The equivariantized braiding formula and the crossed unit identities therefore identify $\widetilde c_{I(W),I(U)}$ with the ordinary flip $W\otimes U\to U\otimes W$. Thus $\mathcal I$ is braided and its image $\mathcal E$ is symmetric, with the canonical symmetric tensor equivalence $\Rep(G)\simeq\mathcal E$; in particular, $\mathcal E$ is Tannakian. This canonical embedding is relative to the fixed action $(T_k,J^k,j_k)$ and the equivariantization in Definition~\ref{def:equivariantization}.

The twist on this subcategory is trivial. Indeed, for $E=I(W)$ the underlying object is supported in degree $1$, so 
$\Theta_E=u_1^E\theta_{I(W)}=\theta_{I(W)}=\id_E.$
Here $u_1^E=\id$, and $\theta_{I(W)}=\id$ follows by additivity and naturality from $\theta_{\mathbf1}=j_1=\id$ in Lemma~\ref{lem:theta}. Thus the crossed twist $\theta_{I(W)}:I(W)\to T_1I(W)$ is identified with the ordinary equivariantized twist $\Theta_E:E\to E$ only through $T_1=\id$ and the displayed equivariant structure.

We next compute the monodromy. Let $A=(\underline A,(u_k^A))\in\cB$, with $\underline A=\bigoplus_{g\in G}A_g$, and let $X=A_g$ be homogeneous. By the equivariantized braiding formula,
\[
 \widetilde c_{A,E_\rho}|_{X\bo I(W)}
 =(u_g^{E_\rho}\bo\id_X)c_{X,I(W)},\qquad
 \widetilde c_{E_\rho,A}|_{I(W)\bo X}
 =(u_1^A\bo\id)c_{I(W),X}=c_{I(W),X}.
\]
The factor $j_g^{-1}$ in $u_g^{E_\rho}=(\rho(g)\otimes j_g^{-1})\iota_{g,W}$ cancels the unit comparison in $c_{X,\mathbf1}$ by the crossed unit identity
$l_X(j_g^{-1}\bo\id_X)c_{X,\mathbf1}=r_X$. Consequently, after the canonical unit identifications, the composite is
\begin{equation*}
 \left.\widetilde c_{E_\rho,A}\widetilde c_{A,E_\rho}\right|_{X\bo I(W)}
 =\id_X\bo\rho(g).
\end{equation*}
In particular, the action is $\rho(g)$, not $\rho(g^{-1})$; the order of the two braidings and the convention $c_{X,Y}:X\bo Y\to T_gY\bo X$ determine this direction.

For a general $A$, this monodromy is the block-diagonal direct sum
$\bigoplus_{g\in G}(\id_{A_g}\bo\rho(g))$ under the canonical decomposition of $\underline A\bo I(W)$. Hence distinct degrees cannot cancel each other. If $A\in\mathcal E'$, choose the left regular representation $\rho_{\mathrm{reg}}$ of $G$. For every $g\ne1$, $\rho_{\mathrm{reg}}(g)e_1=e_g\ne e_1$, so $\rho_{\mathrm{reg}}(g)\ne\id$; the displayed monodromy can be the identity only if $A_g=0$. Conversely, if $A$ is supported in degree $1$, then $\rho(1)=\id$ for every representation and the monodromy with every object of $\mathcal E$ is the identity. Since $T_k\cC_1=\cC_1$, the equivariant objects supported in degree $1$ are exactly $\cC_1^G$. Thus, as full braided subcategories of the fixed equivariantization,
$\mathcal E'=\cC_1^G,$
which gives the canonical braided equivalence in the statement. This is the direct monodromy description of 
\cite[Proposition~4.56(i)]{DGNO}. The centralizer identity follows in the finite-group
equivariantization with its braided $G$-crossed structure.

Finally, $\cC_1$ is a fusion category, and each $T_k$ restricts to a strong monoidal autoequivalence of $\cC_1$: it preserves degree $1$, with tensorator $J^k_{X,Y}:T_kX\bo T_kY\to T_k(X\bo Y)$ and unit comparison $j_k:\mathbf1\to T_k\mathbf1$. These data, including their coherence, define the action used in $\cC_1^G$. Since $G$ is finite, 
\cite[Proposition~4.26]{DGNO} applies and gives
$\FPdim(\cC_1^G)=|G|\FPdim(\cC_1).$
The braided strong monoidal equivalence of Proposition~\ref{prop:neutralcomparison} preserves Frobenius--Perron dimensions, so
$\FPdim(\cC_1)=\FPdim(\Rep(V))$. Combining these identities with $\mathcal E'=\cC_1^G$ yields
$\FPdim(\mathcal E')=|G|\FPdim(\Rep(V)),$
as required.
\end{proof}

\begin{corollary}\label{cor:simplecount}
For every $g\in G$, the number of irreducible $g$-twisted modules is
\begin{equation}\label{eq:simplecount}
 |\Irr(\cC_g)|=
 \bigl|\{[M]\in\Irr(\Rep(V)):[T_gM]=[M]\}\bigr|.
\end{equation}
In particular, every twisted sector contains an irreducible object.
\end{corollary}
\begin{proof}
By Theorem~\ref{thm:main}(ii), the extended $\widetilde S$-operator
of Definition~\ref{def:crossedmodular} is invertible. This is
Kirillov's modularity condition for the underlying $G$-equivariant
fusion category, which has finitely many simple isomorphism classes
because Xu's Theorem~1.1 gives a fusion category. Kirillov's
fixed-point lemma \cite[Lemma~10.7]{Kirillov} therefore applies. We spell out
the graded-space calculation, in particular the inverse in the
second index transformation.

With the same indexing as in Kirillov's definition and extended-$S$
formula \cite[Equation~(9.3)]{Kirillov}, our extended Verlinde space is the finite direct sum of the pieces
\[
 \widetilde{\mathcal V}_{k,h}
 =\bigoplus_{[X]\in\Irr(\cC_h)}
       \Hom_{\cC_G}(X,T_kX),
 \qquad
 \widetilde S(\widetilde{\mathcal V}_{k,h})
       \subseteq \widetilde{\mathcal V}_{h^{-1},k}.
\]
The index transformation $(k,h)\mapsto(h^{-1},k)$ is a bijection
(and preserves the support on commuting pairs). Since $\widetilde S$
is invertible on the whole extended Verlinde space and each source
piece maps into its uniquely assigned target piece, its restriction
is an isomorphism onto that target piece: the inverse operator maps
each target piece back to its unique source piece. In particular,
\[
 \widetilde S\big|_{\widetilde{\mathcal V}_{1,g}}:
 \widetilde{\mathcal V}_{1,g}
 \xrightarrow{\ \sim\ }
 \widetilde{\mathcal V}_{g^{-1},1}.
\]
Thus $\widetilde S$ sends $\widetilde{\mathcal V}_{1,g}$ to
$\widetilde{\mathcal V}_{g^{-1},1}$, with the sector label inverted.

Since $T_1=\id$, the definition gives
\[
\begin{aligned}
 \dim\widetilde{\mathcal V}_{1,g}
 &=\sum_{[X]\in\Irr(\cC_g)}
       \dim\Hom_{\cC_G}(X,X)\\
 &=|\Irr(\cC_g)|.
\end{aligned}
\]
Here each $X$ is simple in the finite semisimple category $\cC_G$,
so $\operatorname{End}_{\cC_G}(X)=\mathbb C$; the direct sum is indexed by simple
isomorphism classes, and no Hom-spaces between distinct simples enter
this definition.

For the target piece, $T_{g^{-1}}$ preserves the neutral component
and sends simple objects to simple objects. Hence
\[
\begin{aligned}
 \dim\widetilde{\mathcal V}_{g^{-1},1}
 &=\sum_{[M]\in\Irr(\cC_1)}
       \dim\Hom_{\cC_G}(M,T_{g^{-1}}M)\\
 &=\bigl|\{[M]\in\Irr(\cC_1):
             [T_{g^{-1}}M]=[M]\}\bigr|.
\end{aligned}
\]
Indeed, the summand is zero when $M$ and $T_{g^{-1}}M$ are not
isomorphic, and is one-dimensional when they are isomorphic, by
simplicity and Schur's lemma.

The fixed isomorphism classes for $T_g$ and $T_{g^{-1}}$ coincide.
If $T_gM\simeq M$, applying $T_{g^{-1}}$ and using the underlying
action identity $T_{g^{-1}}T_g=T_1=\id$ gives
$T_{g^{-1}}M\simeq M$. Conversely, applying $T_g$ to an isomorphism
$T_{g^{-1}}M\simeq M$ gives $T_gM\simeq M$. If the action identities
are expressed using their canonical natural isomorphisms, the same
argument gives the identical statement on isomorphism classes.
Consequently,
\[
 \bigl|\{[M]\in\Irr(\cC_1):[T_{g^{-1}}M]=[M]\}\bigr|
 =\bigl|\{[M]\in\Irr(\cC_1):[T_gM]=[M]\}\bigr|.
\]
This inverse-action identification reconciles the stated $g$-fixed
count with the actual target of Equation~(9.3).

Finally, Proposition~\ref{prop:neutralcomparison} gives a tensor
equivalence $F:\cC_1\xrightarrow{\sim}\Rep(V)$ whose underlying
functor is the identity on modules and module morphisms. It therefore
induces a bijection on simple isomorphism classes. The action $T_g$
on the neutral sector is the pullback action
$Y_{T_gM}(v,z)=Y_M(g^{-1}v,z)$, which is the same pullback action on
ordinary $V$-modules under this identity-on-modules comparison.
Thus the bijection induced by $F$ intertwines the actions on
isomorphism classes, and it identifies the two fixed-point counts.
Combining the block isomorphism and the dimension calculations proves
\eqref{eq:simplecount}.

The final nonemptiness assertion also follows directly from the
faithful grading in Xu's Theorem~1.1: every $\cC_g$ is nonzero, and
a nonzero object in a finite semisimple graded component
has a simple constituent in that component. Independently, the
right-hand fixed-point set is nonempty because the vacuum module is
preserved up to isomorphism by every pullback $T_g$. \end{proof}

\begin{corollary}\label{cor:hol}
If $V$ is additionally holomorphic (so $V$ is its only irreducible untwisted module), $\cC_G$ is equivalent, as a faithfully $G$-graded fusion category, to $\mathrm{Vec}_G^\omega$ for some $[\omega]\in H^3(G,\C^\times)$. Here $\mathrm{Vec}_G^\omega$ denotes finite-dimensional $G$-graded vector spaces with associator twisted by a normalized $3$-cocycle $\omega$. Its equivariantization is a modular category of Frobenius--Perron dimension $|G|^2$.
\end{corollary}
\begin{proof}
Since $V$ is holomorphic, $V$ is the only simple object of $\Rep(V)$. Under the standing strong-rationality hypotheses, $\Rep(V)$ is a fusion category with tensor unit $V$, so $\Rep(V)\simeq_{\otimes}\mathrm{Vec}$ and $\FPdim(\Rep(V))=1$. Proposition~\ref{prop:neutralcomparison} is a strong tensor equivalence; hence $\cC_1$ has the single simple isomorphism class $[\mathbf1]$ and $\FPdim(\cC_1)=1$.

By Xu's Theorem~1.1, $\cC_G$ is a fusion category faithfully graded by the finite group $G$. The ENO grading theorem \cite[Proposition~8.20]{ENO} therefore gives $\FPdim(\cC_g)=\FPdim(\cC_1)=1$ for every $g\in G$. As in Section~2, $\FPdim(\cC_g)$ is the sum of squares of the ambient dimensions $\FPdim_{\cC_G}(X)$ over the simple objects labeled by $\cC_g$.

Faithfulness implies that each $\cC_g$ is nonzero, and hence each finite semisimple homogeneous summand contains a simple object. If $X$ is simple, then $\mathbf1$ occurs in $X\otimes X^\vee$ with multiplicity $\dim\operatorname{Hom}(\mathbf1,X\otimes X^\vee)=\dim\operatorname{End}(X)=1$. Positivity and multiplicativity of Frobenius--Perron dimension imply
$\FPdim_{\cC_G}(X)^2=\FPdim_{\cC_G}(X\otimes X^\vee)\geq 1$.
Since
$\FPdim(\cC_g)=\sum_{[X]\in\Irr(\cC_g)}\FPdim_{\cC_G}(X)^2=1,$
there is exactly one simple isomorphism class in each $\cC_g$. Write its representative as $X_g$; then $\FPdim(X_g)=1$. The object $X_g\otimes X_g^\vee$ has FP dimension one and contains $\mathbf1$, so it is isomorphic to $\mathbf1$; similarly $X_g^\vee\otimes X_g\cong\mathbf1$. Thus each $X_g$ is invertible.

Choose $X_1=\mathbf1$. The grading places $X_g\otimes X_h$ in $\cC_{gh}$, and multiplicativity gives it FP dimension one. Since the only simple class in that component is represented by $X_{gh}$, whose FP dimension is one, the semisimple decomposition of $X_g\otimes X_h$ contains exactly one copy of $X_{gh}$. Thus $X_g\otimes X_h\cong X_{gh}$. In particular, $X_{g^{-1}}$ is a two-sided tensor inverse of $X_g$. Consequently the simple invertible isomorphism classes are indexed by $G$, and their tensor product corresponds to group multiplication.

For every $g,h\in G$, choose an isomorphism $\phi_{g,h}:X_g\otimes X_h\xrightarrow{\sim}X_{gh}$, taking $\phi_{1,g}$ and $\phi_{g,1}$ to be the canonical left and right unitors. Let $\alpha$ denote the associator of $\cC_G$. Define isomorphisms from $(X_g\otimes X_h)\otimes X_k$ to $X_{ghk}$ by
\[
\begin{aligned}
L_{g,h,k}&=\phi_{gh,k}\circ(\phi_{g,h}\otimes\id_{X_k}),\\
R_{g,h,k}&=\phi_{g,hk}\circ(\id_{X_g}\otimes\phi_{h,k})\circ\alpha_{X_g,X_h,X_k}.
\end{aligned}
\]
Because $X_{ghk}$ is simple over $\mathbb C$, there is a unique scalar $\omega(g,h,k)\in\mathbb C^\times$ such that $R_{g,h,k}=\omega(g,h,k)L_{g,h,k}$. The unit coherence identities for the associator and the chosen unitors imply that $\omega(g,h,k)=1$ whenever any one of $g,h,k$ is the identity. Thus $\omega$ is normalized.

The pentagon identity for $\alpha$ equates the two composites from $(((X_g\otimes X_h)\otimes X_k)\otimes X_\ell)$ to $X_g\otimes(X_h\otimes(X_k\otimes X_\ell))$. Substituting $R_{a,b,c}=\omega(a,b,c)L_{a,b,c}$ into those composites yields the $3$-cocycle identity
\[
\omega(h,k,\ell)\,\omega(g,hk,\ell)\,\omega(g,h,k)
=\omega(gh,k,\ell)\,\omega(g,h,k\ell).
\]
If the choices are changed to $\phi'_{g,h}=b(g,h)\phi_{g,h}$, with $b$ normalized to preserve the chosen unitors, then
\[
\omega'(g,h,k)=\omega(g,h,k)\,
\frac{b(h,k)b(g,hk)}{b(g,h)b(gh,k)}
=\omega(g,h,k)\,(\delta b)(g,h,k).
\]
Hence the cohomology class $[\omega]\in H^3(G,\mathbb C^\times)$ is independent of these choices.

With the convention that the associator of $\mathrm{Vec}_G^\omega$ on homogeneous degrees $(g,h,k)$ is multiplication by $\omega(g,h,k)$ from $((g\otimes h)\otimes k)$ to $g\otimes(h\otimes k)$, the chosen maps $\phi_{g,h}$ define a strong monoidal functor $\mathrm{Vec}_G^\omega\to\cC_G$ sending the one-dimensional simple object of degree $g$ to $X_g$. The identity $R_{g,h,k}=\omega(g,h,k)L_{g,h,k}$ is exactly its associativity coherence, and the chosen unitors give its unit coherence. The functor preserves the grading and is fully faithful and essentially surjective, since it gives a bijection between simple isomorphism classes and preserves their endomorphism spaces. It is therefore an equivalence of faithfully $G$-graded fusion categories. 

Finally, Theorem~\ref{thm:main}(ii) gives the modularity of the equivariantization $\cC_G^G$. By Theorem~\ref{thm:main}(iii),
$\FPdim(\cC_G^G)=|G|^2\FPdim(\Rep(V))=|G|^2.$
\end{proof}

\subsection*{Use of AI tools}
ChatGPT (OpenAI) was used in mathematical exploration, literature searching, preparation of preliminary drafts, proof checking, language editing, and \LaTeX{} typesetting. The author assumes full responsibility for the mathematical arguments, references, and final text.


\begin{thebibliography}{99}
\bibitem{ABD} T. Abe, G. Buhl and C. Dong,
\emph{Rationality, regularity, and $C_2$-cofiniteness},
Trans. Amer. Math. Soc. \textbf{356} (2004), 3391--3402,
\href{https://doi.org/10.1090/S0002-9947-03-03413-5}{doi:10.1090/S0002-9947-03-03413-5}.
\bibitem{DLM} C. Dong, H. Li and G. Mason,
\emph{Twisted representations of vertex operator algebras},
Math. Ann. \textbf{310} (1998), 571--600,
\href{https://arxiv.org/abs/q-alg/9509005}{arXiv:q-alg/9509005}.
\bibitem{DLMModular} C. Dong, H. Li and G. Mason,
\emph{Modular-invariance of trace functions in orbifold theory and generalized moonshine},
Comm. Math. Phys. \textbf{214} (2000), 1--56,
\href{https://doi.org/10.1007/s002200000242}{doi:10.1007/s002200000242}.
\bibitem{HuangProblems} Y.-Z. Huang,
\emph{Some open problems in mathematical two-dimensional conformal field theory},
Contemp. Math. \textbf{695} (2017), 123--138,
\href{https://doi.org/10.1090/conm/695/13999}{doi:10.1090/conm/695/13999}.
\bibitem{CarnahanMiyamoto} S. Carnahan and M. Miyamoto,
\emph{Regularity of fixed-point vertex operator subalgebras},
\href{https://arxiv.org/abs/1603.05645}{arXiv:1603.05645}.
\bibitem{McRae} R. McRae,
\emph{Twisted modules and $G$-equivariantization in logarithmic conformal field theory},
Comm. Math. Phys. \textbf{383} (2021), 1939--2019,
\href{https://arxiv.org/abs/1910.13226}{arXiv:1910.13226}.

\bibitem{Crossed} S. Xu, \emph{$G$-crossed braided tensor categories of twisted modules}, \href{https://arxiv.org/abs/2610.29130v1}{arXiv:2610.29130v1} [math.QA] (2026).
\bibitem{DH} J. Du and Y.-Z. Huang, \emph{Twisted intertwining operators and tensor products of (generalized) twisted modules}, SIGMA \textbf{22} (2026), 086, 64 pp., \href{https://doi.org/10.3842/SIGMA.2026.086}{doi:10.3842/SIGMA.2026.086}.
\bibitem{FHL} I. B. Frenkel, Y.-Z. Huang and J. Lepowsky, \emph{On axiomatic approaches to vertex operator algebras and modules}, Mem. Amer. Math. Soc. \textbf{104} (1993), no. 494, \href{https://doi.org/10.1090/memo/0494}{doi:10.1090/memo/0494}.
\bibitem{HuangYangAssoc} Y.-Z. Huang and J. Yang,
\emph{Associative algebras for (logarithmic) twisted modules for a vertex operator algebra},
Trans. Amer. Math. Soc. \textbf{371} (2019), 3747--3786,
\href{https://doi.org/10.1090/tran/7490}{doi:10.1090/tran/7490}.
\bibitem{YangTwAssoc} J. Yang,
\emph{On associative algebras, modules and twisted modules for vertex operator algebras},
J. Algebra \textbf{440} (2015), 354--378,
\href{https://doi.org/10.1016/j.jalgebra.2015.05.026}{doi:10.1016/j.jalgebra.2015.05.026}.
\bibitem{HL}
Y.-Z. Huang and J. Lepowsky, \emph{Tensor products of modules for a vertex operator algebra and vertex tensor categories}, in \emph{Lie Theory and Geometry}, Progr. Math. \textbf{123}, Birkh\"auser, 1994, 349--383.
\bibitem{HuangDE} Y.-Z. Huang, \emph{Differential equations and intertwining operators}, Commun. Contemp. Math. \textbf{7} (2005), 375--400, \href{https://doi.org/10.1142/S0219199705001799}{doi:10.1142/S0219199705001799}.
\bibitem{Huang} Y.-Z. Huang, \emph{Rigidity and modularity of vertex tensor categories}, Commun. Contemp. Math. \textbf{10} (2008), suppl. 1, 871--911, \href{https://doi.org/10.1142/S0219199708003083}{doi:10.1142/S0219199708003083}.
\bibitem{Kirillov} A. Kirillov, Jr., \emph{On $G$-equivariant modular categories}, \href{https://arxiv.org/abs/math/0401119}{arXiv:math/0401119}.
\bibitem{DGNO} V. Drinfeld, S. Gelaki, D. Nikshych and V. Ostrik, \emph{On braided fusion categories I}, Selecta Math. (N.S.) \textbf{16} (2010), 1--119, \href{https://doi.org/10.1007/s00029-010-0017-z}{doi:10.1007/s00029-010-0017-z}.
\bibitem{EGNO} P. Etingof, S. Gelaki, D. Nikshych and V. Ostrik,
\emph{Tensor Categories}, Math. Surveys Monogr. \textbf{205}, Amer. Math. Soc., 2015, Section~8.24.
\bibitem{ENO} P. Etingof, D. Nikshych and V. Ostrik, \emph{On fusion categories}, Ann. of Math. (2) \textbf{162} (2005), 581--642, \href{https://arxiv.org/abs/math/0203060}{arXiv:math/0203060}.
\end{thebibliography}
\end{document}